\documentclass[11pt]{amsart}

\usepackage[dvipsnames]{xcolor}
\usepackage[utf8]{inputenc}
\usepackage{comment,enumerate,graphicx}
\usepackage{enumitem,amsmath,amssymb}
\usepackage{mathrsfs}
\usepackage{amsthm}
\usepackage{thmtools}
\usepackage{url}
\usepackage{caption}
\usepackage{subcaption}
\usepackage[top=22 mm, bottom=22 mm, left=20 mm, right= 20 mm]{geometry}
\usepackage{relsize}
\usepackage[hidelinks]{hyperref}
\newcommand{\setupstar}{\hyperref[setup:star]{Setup ($\star$)}}
\makeatletter
\newcommand{\labelinthm}[1]{%
   \label{temp#1}
   \protected@write \@auxout {}{\string \newlabel{#1}{{\emph{\ref{temp#1}}}{\thepage}{\emph{\ref{temp#1}}}{temp#1}{}} }%
}
\makeatother
\usepackage{adjustbox}
\usepackage{tikz}
\usetikzlibrary{math,calc,intersections, scopes}
\usetikzlibrary{positioning,arrows,shapes,decorations.markings,decorations.pathreplacing, decorations.pathmorphing,matrix,patterns}
\tikzstyle{vertex}=[circle,draw=black,fill=black,inner sep=0,minimum size=5pt,text=white,font=\footnotesize]
\tikzstyle{redvertex}=[circle,draw=red,fill=red,inner sep=0,minimum size=5pt,text=white,font=\footnotesize]
\definecolor{amber}{rgb}{1.0, 0.75, 0.0}
\definecolor{darkgreen}{rgb}{0.18, 0.7, 0.46}

\declaretheorem[name=Theorem,numberwithin=section]{theorem}

\newtheorem{lemma}[theorem]{\bf Lemma}

\newtheorem{claim}[theorem]{\bf Claim}
\newtheorem{proposition}[theorem]{\bf Proposition}
\newtheorem{conjecture}[theorem]{\bf Conjecture}

\newtheorem*{theorem*}{\bf Theorem}

\theoremstyle{definition}

\newtheorem{definition}[theorem]{\bf Definition}

\newcommand\claimproofend{\renewcommand{\qedsymbol}{$\boxdot$}
\end{proof}
\renewcommand{\qedsymbol}{$\square$}
}

\newlist{lemenum}{enumerate}{1}
\setlist[lemenum]{label=(\roman*), ref=\thelemma(\roman*)}

\def\eps{\varepsilon}

\def\cC{\mathcal{C}}
\def\cD{\mathcal{D}}
\def\cE{\mathcal{E}}
\def\cF{\mathcal{F}}
\def\cG{\mathcal{G}}
\def\cH{\mathcal{H}}

\def\cN{\mathcal{N}}

\def\cS{\mathcal{S}}
\def\cT{\mathcal{T}}

\def\bE{\mathbb{E}}
\def\bR{\mathbb{R}}
\def\bS{\mathbb{S}}

\def\Pb{\mathbb{P}}

\def\ba{\boldsymbol{a}}

\def\be{\boldsymbol{e}}
\def\bg{\boldsymbol{g}}
\def\bu{\boldsymbol{u}}
\def\bv{\boldsymbol{v}}
\def\bw{\boldsymbol{w}}
\def\bx{\boldsymbol{x}}
\def\by{\boldsymbol{y}}
\def\bz{\boldsymbol{z}}

\newcommand{\hide}[1]{}

\DeclareMathOperator{\spn}{span}
\DeclareMathOperator{\polylog}{polylog}
\DeclareMathOperator{\TV}{TV}
\DeclareMathOperator{\Var}{Var}

\renewcommand{\diamond}{\mathbin{\vcenter{\hbox{\begin{tikzpicture}[scale=0.12, baseline=-0.5ex]
\coordinate (A) at (0,0);
\coordinate (B) at (1.6,0);
\coordinate (C) at (0.8,1.25);
\coordinate (D) at (0.8,-1.25);
\draw (A)--(C)--(B)--(D)--(A);
\draw (A)--(B);
\end{tikzpicture}}}}}

\title{Distinguishability threshold for random geometric graphs}
\date{}

\author{Zach Hunter}
\author{Aleksa Milojevi\'c}
\author{Benny Sudakov}
\thanks{Department of Mathematics, ETH Z\"urich, Switzerland. Email: {\tt \{zach.hunter, aleksa.milojevic, benjamin.sudakov\}@math.ethz.ch}. Research supported in part by SNSF grant 200021-228014.}

\begin{document}

\begin{abstract}
The spherical random geometric graph $G(n,d,p)$ is obtained by sampling
$n$ independent points uniformly on the unit sphere
$\bS^{d-1}\subseteq\bR^d$ and joining pairs of points which are
sufficiently close, where the threshold is chosen so that the edge
probability is $p$. The central question related to this model, and to a
broad class of other models, is the following: when does the underlying
geometry affect the resulting graph in a way which makes it
distinguishable from the Erd\H{o}s--R\'enyi random graph $G(n,p)$, as
measured in total variation distance? The precise answer to this
question was conjectured by Bubeck, Ding, Eldan, and R\'acz, who
predicted that $G(n,d,p)$ and $G(n,p)$ are indistinguishable precisely
when $d \gg n^3p^3(\log p^{-1})^3$, and provided a test for
distinguishing these models in the low-dimensional regime.

Although this conjecture attracted considerable attention from researchers in
probability, theoretical computer science, and high-dimensional
statistics, it was previously fully proved only in the constant-density case.
In this paper, we resolve the distinguishability conjecture in the broad
range $1/3 \geq p \geq n^{-1/5} \text{polylog}(n)$. The key ingredient of our proof is a stronger statement which gives a
precise asymptotic formula for the probability that $G(n,d,p)$ realizes
a prescribed graph $H$: above the conjectured threshold, this probability
is at most $(1+o(1))$ times the corresponding probability for
$G(n,p)$, with the signed triangle count of $H$ appearing as the leading
correction term.
\end{abstract}

\maketitle

\section{Introduction}

The classical binomial random graph model with parameters $n\geq 1$ and $p\in (0, 1)$ is obtained by fixing a set of $n$ vertices and connecting any two among them randomly with probability $p$, independently of all other pairs. This model, which we denote by $G(n, p)$, was introduced by Erd\H{o}s and R\'enyi \cite{ER59} in 1959 and has played a central role in probability and combinatorics ever since. Although the assumption that the edges appear independently enables a precise analysis of this model, it offers a rather crude approximation of the real-world networks. In physical, biological, and social systems, the interactions between objects are typically determined by proximity, with distant objects having weaker interactions than nearby ones.

A familiar instance of this principle arises in the study of synchronization. The Kuramoto model \cite{Kuramoto75}, originally proposed to describe systems of weakly coupled oscillators, has been used to model the collective behavior of biological populations such as fireflies, which tend to synchronize their flashing with their immediate neighbors \cite{Strogatz00}. In this and many related settings, the strength of the coupling is governed by the distance between oscillators, and only sufficiently close pairs exert a noticeable influence on one another.

This perspective motivates a general class of random graph models in which the vertices are points in a metric space, and two vertices are joined by an edge whenever their distance is sufficiently small. The prototypical example is the \emph{random geometric graph} (which we abbreviate RGG), introduced by Gilbert \cite{Gilbert61} in 1961 (see also \cite{Penrose03, DC22} for more recent accounts). This model is defined as follows: given a positive integer $d$, a probability distribution $\nu$ on $\mathbb{R}^d$ with a bounded density, and a parameter $r=r(n)>0$, one samples points $\bx_1,\dots,\bx_n\in\mathbb{R}^d$ independently from $\nu$ and joins $\bx_i$ to $\bx_j$ whenever $\|\bx_i-\bx_j\|<r$; the resulting graph is denoted by $G(n,r)$. In the classical setting, the dimension $d$ is fixed and one studies the asymptotic behavior of various graph parameters as $n\to\infty$, possibly tuning $r=r(n)$ to obtain graphs of appropriate density. For example, questions related to connectivity \cite{Pen99}, clique number \cite{Mu08}, chromatic number \cite{MM11} and Hamiltonicity \cite{BBKMW11} of $G(n,r)$ have been extensively studied.

In recent years, the \emph{high-dimensional} regime, where the ambient dimension $d=d(n)$ is allowed to grow with $n$, has attracted considerable attention. Two models have emerged as particularly prominent. The \emph{spherical random geometric graph} $G(n,d,p)$ is obtained by sampling vectors $\bx_1,\dots,\bx_n$ independently and uniformly from the unit sphere $\mathbb{S}^{d-1}\subset \mathbb{R}^d$, and joining $i$ and $j$ by an edge whenever $\langle \bx_i,\bx_j\rangle \geq c_{p,d}/\sqrt{d}$, where the threshold $c_{p,d}$ is chosen so that every individual edge appears with probability $p$. Note that the inner product condition stated above can be equivalently reformulated as a distance condition, since all points lie on the unit sphere. The \emph{Gaussian random geometric graph} is defined in the same way, with the vectors instead sampled independently from the $d$-dimensional Gaussian $\mathcal{N}(0, \frac{1}{d}I_d)$. In this work, we will focus on the spherical model.

A first result in this direction is due to Devroye, Gy\"orgy, Lugosi, and Udina \cite{DGLU11}, who studied the clique number of $G(n,d,p)$, motivated by the problem of testing for dependencies in a collection of normal random variables. They showed that, for fixed $p$, the clique number $\omega(G(n,d, p))$ matches $\omega(G(n,p))$ once $d$ is at least polylogarithmic in $n$. The question of understanding the clique number of $G(n, d, p)$ has recently found a striking application in Ramsey theory: in a recent breakthrough, Ma, Shen, and Xie \cite{MSX26} obtained the first exponential improvement over Erd\H{o}s's classical lower bound \cite{Erd47} on the Ramsey numbers $R(\ell, C\ell)$ (where $C>1$ is a fixed constant) by using a construction based on geometric random graphs. This improvement relies crucially on a precise analysis of clique probabilities in a carefully parametrized spherical RGG (see also \cite{HMS25}).

This application highlights that the value of the high-dimensional random geometric graph model lies precisely in its departure from the binomial random graph $G(n,p)$: the geometric structure introduces correlations between edges which are absent in the binomial random graphs. It is therefore natural to ask:
\[\text{For which values of $d$ and $p$ does $G(n,d,p)$ differ from $G(n,p)$?}\] The natural measure of closeness between two such distributions is the \emph{total variation distance}: for two probability distributions $\mu, \mu'$ on a finite set $\Omega$ it is defined as $\TV(\mu, \mu')=\frac{1}{2}\sum_{x\in \Omega}|\mu(x)-\mu'(x)|$. Concretely, the goal is to determine the threshold for $d=d(n,p)$ at which $\TV(G(n,p), G(n,d,p))$ transitions from being close to $1$, in which case the two models are statistically distinguishable, to being close to $0$, in which case they are indistinguishable. This is one of the central questions related to random geometric graphs, of interest to high-dimensional statistics, probability, and theoretical computer science communities. This question was pioneered in the visionary work of Bubeck, Ding, Eldan, and R\'acz \cite{BDER16}, who formulated the following conjecture. Here and throughout, we write $a\gg b$ to mean that $\lim_{n\to \infty} a/b= \infty$.

\begin{conjecture}[\cite{BDER16}]\label{conj:BDER}
If $d\gg n^3p^3(\log 1/p)^{3}$, then $\TV(G(n,p), G(n,d,p))\to 0$ as $n\to\infty$.
\end{conjecture}

Let us give some context for the form of the threshold appearing in this conjecture. Due to their underlying geometry, random geometric graphs exhibit a triangle-closing property, i.e., triangles arise more frequently in $G(n,d,p)$ than in $G(n,p)$. The intuition is based on the triangle inequality: if a vertex $i$ is close to both $j$ and $k$, then $j$ and $k$ are themselves more likely to be close to each other. Bubeck, Ding, Eldan, and R\'acz made this quantitative by showing that for any triple of vertices $i, j, k$, the probability that they form a triangle in $G(n, d, p)$ is at least $p^3(1+c/\sqrt{d})$, for some constant $c\asymp (\log p^{-1})^{3/2}>0$ (see also \cite{LMSY22}). The expected number of triangles in $G(n, d, p)$ therefore exceeds the corresponding expectation in $G(n, p)$ by an additive factor of order $cn^3p^3/\sqrt{d}$. Whenever this gap dominates the standard deviation of the triangle counts in $G(n, p)$ and $G(n, d, p)$, the two distributions can be told apart simply by counting triangles.

A key observation of \cite{BDER16} is that one can do substantially better by passing to the \emph{signed triangle count}: for a graph $H$ on $[n]$ with adjacency matrix $A$, define
\begin{equation}\label{eqn:signed_triangle_count_definition}
N_{\triangle}(H)=\sum_{1\leq i<j<k\leq n}(A_{ij}-p)(A_{jk}-p)(A_{ik}-p).
\end{equation} The centering ensures that $\bE[N_{\triangle}(G(n, p))]=0$, while the variance $\Var(N_{\triangle}(G(n, p)))=O(n^3p^3)$, which is much smaller than the variance of the unsigned triangle count. On the other hand, $\mathbb{E}[N_{\triangle}(G(n,d,p))]\gtrsim n^3p^3/\sqrt{d}$, so the gap in expectations is unchanged. Balancing this signal against the standard deviation leads to the threshold $d\asymp n^3 p^3$ that appears in Conjecture~\ref{conj:BDER}. The above analysis was sharpened in \cite{LMSY22}, where it was shown that the signed triangle test succeeds whenever $d\ll n^3p^3(\log p^{-1})^3$.\footnote{In fact, Conjecture~\ref{conj:BDER} was originally stated without the precise logarithmic correction, and this correction was added by the authors of \cite{LMSY22}.} Hence, a natural interpretation of Conjecture~\ref{conj:BDER} is that, in fact, the signed triangle test is optimal among all tests for distinguishing between $G(n, p)$ and $G(n, d, p)$.

Since its formulation, Conjecture~\ref{conj:BDER} has attracted considerable interest from the probability, theoretical computer science, and high-dimensional statistics communities. The original work of Bubeck, Ding, Eldan, and R\'acz \cite{BDER16} established the conjecture in the dense setting $p=\Theta(1)$, where they identified the sharp threshold $d\asymp n^3$. Hence, the future works, including ours, focused on the regime where $p\to 0$. A couple of years later, Brennan, Bresler, and Nagaraj \cite{BBN20} obtained the first results beyond the dense case, proving indistinguishability for $d\gg \widetilde \Omega(\min\{pn^3, p^2n^{7/2}\})$. The current best results on this conjecture are given by Liu, Mohanty, Schramm, and Yang \cite{LMSY22}, who proved that $G(n, p)$ and $G(n, d, p)$ are indistinguishable whenever $d\gg p^2n^3$. They also obtained an almost sharp result in the case $p=\Theta(1/n)$, showing that $d\gg (\log n)^{36}$ suffices for indistinguishability in this case, thereby almost resolving Conjecture~\ref{conj:BDER} in that case. Beyond the original spherical model, a number of variants of the conjecture have been studied in related geometric settings, including when points are sampled from anisotropic Gaussian distributions \cite{EM20, BBH24}, when a certain amount of noise is added to the edges \cite{LR21, LR23}, and for the $L_\infty$ metric on the torus \cite{BaB24}. Further, \cite{PV21} studied the relaxation of the conjecture with the total-variation distance replaced by so-called inclusion-divergence.

In spite of the progress mentioned above, Conjecture~\ref{conj:BDER} has remained open throughout the intermediate range $1/n\ll p\ll 1$. The main result of this paper is a resolution of the conjecture, for every $p\gtrsim n^{-1/5}$. We provide a tight bound for the threshold dimension $d$ at which $G(n, d, p)$ can be distinguished from $G(n, p)$.

\begin{theorem}\label{thm:main}
For every $\eps>0$, there exists a large constant $C_\eps>0$  with the following property. Let $n$ be a large positive integer and $p\in (0, 1)$ a probability satisfying $1/3\geq p\geq n^{-1/5}(\log n)^A$, where $A$ is a large absolute constant. For every $d\geq C_\eps n^3p^3 (\log p^{-1})^3$, we have
\[\TV(G(n, p), G(n, d, p))\leq \eps.\]
\end{theorem}

Our entry point to the world of random geometric graphs was the exciting recent paper of Ma, Shen, and Xie \cite{MSX26} discussed above. In \cite{HMS25}, we significantly simplified their construction using Gaussian random geometric graphs and, as a byproduct, obtained quantitatively better lower bounds on the off-diagonal Ramsey numbers. The central technical ingredient in \cite{HMS25} was a precise estimate of the probability that a prescribed set of points in the random geometric graph forms a clique.

However, the techniques developed in \cite{HMS25} for controlling the clique probabilities provide only the first steps in the proof of Theorem~\ref{thm:main}.
In the present paper, we go considerably further: beyond estimating the probability of a clique or an independent set in $G(n, d, p)$, we give the corresponding estimate for essentially any graph $H$ on $n$ vertices, with the signed triangle count of $H$ appearing as the leading correction term. In fact, our proof gives an understanding of the probability $\Pb[G(n, d, p)=H]$ also below the distinguishability threshold, which may be useful for studying RGGs in the lower-dimensional regime.
Beyond the signed triangle counts we further identify the lower-order correction terms which correspond to signed counts of larger subgraphs (see Proposition \ref{prop:main_prop}). In the remainder of this section, we discuss these technical ingredients and some additional ideas in more detail.

Besides the testing problem, the high-dimensional random geometric graph $G(n,d,p)$ has been studied from a number of other angles. For example, spectral properties of the adjacency matrix of $G(n, d, p)$ have been studied by \cite{LMSY23,ABI24,CZ25}, and in \cite{BB24} it was shown that random geometric graphs of appropriately high dimension can be sandwiched between two Erd\H{o}s-R\'enyi random graphs.

\subsection{Sketch of the proof}\label{subsec:proof_outline}
Let us give an overview of the proof of Theorem~\ref{thm:main}. The centerpiece of our proof is the following estimate on the probability that a random graph $G(n, d, p)$ realizes some fixed graph $H$. In order to state it, let us introduce the parameters $z_p$ and $a$, which will be used throughout the proof. We define $z_p$ to be the unique real number for which $\Pb[\cN(0, 1)\geq z_p]=p$, i.e., $z_p=\Phi^{-1}(1-p)$  (where $\Phi$ is the Gaussian CDF). Also, we let $a=\phi(z_p)=e^{-z_p^2/2}/\sqrt{2\pi}$ be the value of the Gaussian PDF at $z_p$. Finally, recall that $N_{\triangle}(H)$ denotes the signed triangle count defined above. Taking $d\geq  Cn^3p^3(\log p^{-1})^3$, $1/3\geq p\geq n^{-1/5}\polylog(n)$ (where the upper bound on $p$ is a technicality ensuring that $z_p$ is not close to $0$), we will essentially show that for any pseudorandom graph $H$ after intersecting with a high probability event $\cF$ (which we define and motivate later), we have

\begin{equation}\label{eqn:main_estimate}
    \Pb\big[(G(n,d,p)=H)\wedge \cF\big]\leq p^{e(H)}(1-p)^{\binom{n}{2}-e(H)}\exp\Big(\frac{(1+o(1))a^3}{p^3(1-p)^3\sqrt d}N_{\triangle}(H)\Big)
\end{equation}

The pseudorandomness assumptions on $H$ required to prove this estimate are quite mild: for example, they include the assumption that the degrees and codegrees are well-concentrated around their means. Although our proof technique can be used to give an estimate on the probability of realizing any graph $H$, having a pseudorandomness assumption is very useful in controlling various error terms in the proof, and allows us to show that all of them are negligible. We thus define the collection of graphs $\cE$ which satisfy the required assumptions (see Definition~\ref{def:graph_event_E_reg} for the precise conditions), such that both $G(n, p)$ and $G(n, d, p)$ lie in $\cE$ with high probability. Then, bounding the total variation distance only on the set of graphs $\cE$ suffices to bound the total variation distance of the whole distributions.

On the other hand, $\cF$ is the event that the vectors $\bx_1, \dots, \bx_n$ defining the graph $G(n, d, p)$ are pseudorandom, in the sense that they are almost orthogonal and that all of their coordinates are small in absolute value (see Definition~\ref{def:vector_event_F} for an exact formulation).  Again, we will show that $\bx_1, \dots, \bx_n$ satisfy these pseudorandomness conditions with high probability, which means that intersecting with the event $\cF$ has a negligible effect on the total variation distance between the two distributions.

The main message of the estimate (\ref{eqn:main_estimate}) is that the likelihood of realizing $H$ as a spherical random geometric graph is determined, to the first order, by the structure of triangles in $H$. As we alluded to earlier, the underlying geometry makes triangles more frequent in $G(n, d, p)$ compared to $G(n, p)$. The result (\ref{eqn:main_estimate}) shows that conversely, if there are very few triangles in $H$, then $H$ is hard to realize as a RGG. Beyond triangles, the estimate takes into account all patterns of edges on 3 vertices. For example, having exactly two edges on 3 vertices is less likely in the RGG than in the binomial random graph. Since this pattern contributes to $N_\triangle(H)$ with negative weight (its weight is $-p(1-p)^2$), we see that having many such structures in $H$ makes it harder to realize $H$ as an outcome of $G(n, d, p)$. By aggregating the contributions of all three-vertex patterns, we obtain the signed triangle count of $H$. Finally, let us note that, a priori, there is nothing special about patterns of edges on 3 vertices - the signed count of four-cycles and larger graphs in $H$ also influences $\Pb[G(n, d, p)=H]$. However, four-cycles and larger graphs have weaker effects, and therefore we choose to highlight $N_\triangle(H)$ as the leading term. More on the effects of larger graphs will be said in Section~\ref{sec:main_estimate}.

Let us briefly explain why this estimate, as formalized in Proposition~\ref{prop:main_prop}, implies Theorem~\ref{thm:main}. As $p\to 0$, we have $z_p=(1+o(1))\sqrt{2\log p^{-1}}$ and $a=\Theta(p\sqrt{2\log p^{-1}})$. When $d\gg n^3p^3(\log p^{-1})^3$, the signed count of triangles is not significantly different between $G(n, p)$ and $G(n, d, p)$, and it is with high probability on the order of $O(n^{3/2}p^{3/2})$. This shows that for almost all graphs we have $\left|\frac{a^3}{p^3(1-p)^3\sqrt d}N_{\triangle}(H)\right|\lesssim \sqrt{\frac{n^3p^3(\log p^{-1})^3}{d}}=o(1)$. Thus, on the event $(G(n, d, p)\in \cE)\wedge \cF$, which has probability $1-o(1)$, we get an upper bound $\Pb[G(n, d, p)=H\wedge \cF]\leq (1+o(1))\Pb[G(n, p)=H]$, which is sufficient to show $\TV(G(n, d, p), G(n, p))=o(1)$.
\medskip

Let us now give some insights into the proof of (\ref{eqn:main_estimate}), which is the subject of Proposition~\ref{prop:main_prop}. The first step of our proof is to use the Gram-Schmidt orthogonalization procedure, which is a standard idea in the analysis of random geometric graphs (see, e.g., \cite{BBN20}). Let $\bx_1,\dots,\bx_n$ be the points defining $G(n,d,p)$. We choose an orthonormal basis $\be_1,\dots,\be_d$ by Gram--Schmidt so that $\mathrm{span}\{\be_1,\dots,\be_s\}=\mathrm{span}\{\bx_1,\dots,\bx_s\}$ for every $1\leq s\leq n$. To emphasize the change of basis, let us denote the vectors $\bx_1, \dots, \bx_n$ by $\bv_1, \dots, \bv_n$ in this new basis. Then, the matrix $M$ with rows $\bv_1,\dots,\bv_n$ is lower triangular. The key property of this basis is that even after the first $s-1$ columns have been exposed, the distribution of the remaining coordinates $\bv_i(s)=\langle \bv_i,\be_s\rangle$, $i>s$, is essentially the spherical distribution (see Lemma~\ref{lemma:distribution_v_i} for a precise statement).

The proof of (\ref{eqn:main_estimate}) then proceeds by an inductive argument. If we fix an integer $0\leq s\leq n-1$, and assume that the first $s$ columns of the matrix $M$ have been exposed (in other words, we reveal the coordinates $\bv_i(j)$ for all $i,j$ with $j\leq s$), then our goal is to estimate the probability that the resulting graph agrees with $H$ on the vertices $\{s+1,\dots,n\}$, when the unrevealed coordinates are sampled according to the prescribed conditional distribution. The precise form of this estimate is given in Lemma~\ref{lemma:induction-statement}.

In the course of an induction step from $s$ to $s-1$, we will estimate the probability that, given the first $s-1$ columns, the edges from vertex $s$ to $\{s+1,\dots,n\}$ agree with $H$. Recall that the vertices $1\leq s< i\leq n$ are connected if $\langle \bv_s, \bv_i\rangle=\bv_s(s)\bv_i(s)+\sum_{j< s}\bv_s(j)\bv_i(j)\geq c_p/\sqrt{d}$. Crucially, this expression contains only one random variable, namely $\bv_i(s)$ (note that $\bv_s(s)$ can be computed from $\bv_s(j), j<s$). This has two consequences: firstly, the edges incident to the vertex $s$ are now conditionally independent, and their probabilities can be easily computed. Moreover, the event that $si$ is an edge introduces only a simple truncation on $\bv_i(s)$, which allows us to maintain a good understanding of their conditional distribution. This computation is performed in Lemma~\ref{lemma:vertex_s}, where the description of the resulting conditional distribution is also given.

Having estimated the probability that the edges incident to $s$ agree with $H$, we use the induction hypothesis to estimate the probability that $G(n, d, p)$ agrees with the remaining graph $H_s$ on $\{s+1,\dots,n\}$. An important point is that the bound given by the induction hypothesis depends on the coordinates $\bv_i(s)$, $i>s$. The final bound that $G(n, d, p)$ agrees with $H$ on $\{s, \dots, n\}$ is then obtained by averaging the bound from the induction hypothesis over the realizations of $\bv_i(s)$, $i>s$ that guarantee $s$ is connected to the correct set of vertices. For details, see the proof of Lemma~\ref{lemma:induction-statement}. Computing this expectation precisely is the main technical challenge of the proof, and the needed estimate is presented in Lemma~\ref{lemma:ratio}, whose proof spans the remainder of the paper.

\medskip\noindent
\textbf{Notation and conventions.} Throughout this paper, we will use the standard asymptotic notation $O(\cdot)$ and $\widetilde O(\cdot)$. If $n$ is a growing parameter, we write $f(n)=g(n)+O(h(n))$ if there is a constant $C$ such that $g(n)-Ch(n)\leq f(n)\leq g(n)+Ch(n)$ for sufficiently large $n$, and we write $f(n)\leq g(n)+O(h(n))$ or $f(n)\geq g(n)+O(h(n))$ when only one of these two inequalities is meant. Further, $\widetilde O(\cdot)$ is allowed to contain constant-power logarithmic terms depending on $n, p$ and $d$. When we write $f\gg g$, we mean that as $n\to \infty$ we have $f/g\to \infty$ as well, although we use this notation mostly informally. Finally, ${\rm polylog}(n)$ denotes a sufficiently high but constant power of $(\log n)$. Further, since we will be often dealing with tuples of vectors, say $\bv_1, \dots, \bv_n$, we denote the $j$-th coordinate of the vector $\bv_i$ by $\bv_i(j)$.

\medskip\noindent
\textbf{Paper organization.} In Section~\ref{sec:prelims}, we outline a set of useful preliminaries about the sphere and Gaussian distributions. Since the proofs of these results are mostly standard, we postpone them to the appendix and encourage the reader to simply read through the statements on the first pass. In Section~\ref{sec:main_estimate}, we state the main estimate on the pointwise probabilities of the RGGs, and use it to show Theorem~\ref{thm:main}. The proof of Proposition~\ref{prop:main_prop} then begins in Section~\ref{sec:induction}, where we outline the inductive scheme by exposing one vertex at a time, as outlined in the above sketch. However, we do not complete the proof in Section~\ref{sec:induction}, as we leave the bound on the one-step contribution in Lemma~\ref{lemma:ratio} unjustified. In Section~\ref{sec:perturbation} we prepare the proof of Lemma~\ref{lemma:ratio}, by analyzing the conditional distribution of the variables $\bv_i(s)$, $s<i\leq n$. In Section~\ref{sec:ratio} we then show the only ingredient left unproven in Section~\ref{sec:induction}, Lemma~\ref{lemma:ratio}, modulo an estimate on a certain moment generating function given in Lemma~\ref{lemma:S_bound}. The proof is finally completed in Section~\ref{sec:MGF}, where we establish Lemma~\ref{lemma:S_bound}.

\medskip\noindent
\textbf{Acknowledgments.} We would like to thank Kiril Bangachev for telling us about Conjecture~\ref{conj:BDER}, Yuval Wigderson and Sahar Diskin for valuable comments which improved the presentation of this paper, and Nina Kam\v{c}ev for insightful discussions about RGGs. Additionally, we would like to acknowledge the use of AI tools for polishing this paper, and note that all of the main proof ideas were entirely human-generated.

\medskip\noindent
{\bf Note added in proof.} In the final stages of polishing this paper, we learned that Conjecture \ref{conj:BDER} was independently proven in the regime $p\gg n^{-2/3}/\log n$ by Du, Mao, Sun, Wu, and Xu \cite{DMSWX26}. Their approach is rather different from ours and is based on KL divergence expansion and Pinsker's inequality.

\section{Preliminaries}\label{sec:prelims}

In this section, we set up notation and list a set of useful preliminary observations.
Omitted proofs are postponed to the appendix.

\subsection{Gaussians and sphere distributions.}\label{subsec:gaussian_and_sphere}

Let $\phi(t)={e^{-t^2/2}}/{\sqrt{2\pi}}$ be the probability density function of the standard Gaussian $\cN(0, 1)$ and let $\Phi(t)=\int_{-\infty}^t \phi(x) \,dx$ be its cumulative distribution function. Let us define $\lambda(t)={\phi(t)}/{\Phi(t)}$ and present two properties of $\Phi$, which we can think of as log-concavity and second-order Taylor expansion.

\begin{lemma}\label{lemma:tail-expansion}
For every $\alpha,\beta\in \bR$ we have
\begin{align}\label{eqn:log_concavity}
\Phi(\alpha+\beta) &\leq \Phi(\alpha)\exp\Big(\lambda(\alpha)\beta\Big),\text{ and }\\
\label{eqn:second_order_taylor}
\Phi(\alpha+\beta) &\leq \Phi(\alpha)\exp\Big(\lambda(\alpha)\beta-\lambda(\alpha)^2\frac{\beta^2}{2}-\alpha\lambda(\alpha)\frac{\beta^2}{2}+|\beta|^3\Big).
\end{align}
\end{lemma}

\begin{lemma}\label{lemma:mills_ratio}
The function $\lambda(t)$ is analytic on $\bR$, with $-1\leq \lambda'(t)\leq 0, 0\leq \lambda''(t)\leq 2$ and $\lambda'(t)=-\lambda(t)(t+\lambda(t))$ for all $t\in \bR$. Also, for every $t\leq 0$ we have $|t|\leq \lambda(t)\leq |t|+1$.
\end{lemma}

Truncated Gaussian random variables will also play an important role in the paper. If $Z\sim \cN(0, \sigma^2)$ is a one-dimensional Gaussian and $(a, b)\subseteq \bR$ is an interval, the \textit{truncated Gaussian} is the distribution of $Z$ conditioned on $Z\in (a, b)$. Our distributions will always be truncated only on one side, i.e. we will only consider cases where $a=-\infty$ or $b=+\infty$. If $a\neq -\infty$, we have a lower truncated variable, and if $b\neq+\infty$ we have an upper truncated variable. We begin our discussion of truncated Gaussians by computing their expectation.

\begin{lemma}\label{lemma:expectation_truncated_gaussian}
Let $t$ be a real number and let $X\sim \cN(0, \sigma^2)$. If $X_t^+$ denotes the variable $X$ conditioned on $X\geq t\sigma$, and $X_t^-$ denotes the variable $X$ conditioned on $X\leq t\sigma$, then
\begin{align*}
    &\bE\big[X_t^-\big]=-\lambda(t)\sigma \qquad \text{ and }\qquad\Var\big(X_t^-\big)=\sigma^2\Big(1-t\lambda(t)-\lambda(t)^2\Big),\text{ and }\\
    &\bE\big[X_t^+\big]=\lambda(-t)\sigma\qquad \text{ and }\qquad\Var\big(X_t^+\big)=\sigma^2\Big(1+t\lambda(-t)-\lambda(-t)^2\Big).
\end{align*}
\end{lemma}

Further, for a positive integer $k\geq 3$, we denote by $\cS_k$ the law of $\sqrt{k}\by(1)$, where $\by$ is a uniform random point on the unit sphere $\bS^{k-1}\subseteq \bR^k$ in $k$ dimensions. The PDF and CDF of the distribution $\cS_k$ will be denoted by $\psi_k$ and $\Psi_k$, respectively. The reason for the scaling factor $\sqrt{k}$ is the following theorem, which shows that $\phi, \Phi$ and $\psi_k, \Psi_k$ are very close.

\begin{lemma}\label{lemma:sphere_to_gaussian_CDF}
Let $k\geq 3$ be a positive integer and let $t$ be any real number. Then, we have
\begin{equation*}
 |\Psi_k(t)-\Phi(t)|\leq O\Big(\frac{1+t^4}{k}\Big)\Phi(t),\qquad \text{ and}\qquad  |\psi_k(t)-\phi(t)|\le O\Big(\frac{1+t^4}{k}\Big)\phi(t).
\end{equation*}
\end{lemma}

Finally, for a fixed value $p\in (0, 1/3)$, let $z_p=\Phi^{-1}(1-p)$ and $c_{p, k}=\Psi_k^{-1}(1-p)$ be the unique real numbers for which $\Pb[X\geq z_p]=p$, where $X\sim \cN(0, 1)$, and $\Pb[Y\geq c_{p, k}]=p$, where $Y\sim \cS_k$. Since we will mostly be working in $\bR^d$, we write $c_p=c_{p, d}$. These values have the following asymptotic behavior.

\begin{lemma}\label{lemma:quantile_asymptotics}
Let $k\geq 3$ be an integer and let $p\in (0, 1/3)$ be a probability. We have
\begin{equation*}
z_p=(1+o(1))\sqrt{2\log p^{-1}}\text{ as $p\to 0$}, \text{ and if $p\geq 1/k$ we have }c_{p, k}=z_p+\widetilde O\Big(\frac{1}{k}\Big).
\end{equation*}
\end{lemma}

\begin{lemma}\label{lemma:sphere_coupling_moments}
Let $t\in \mathbb R$ and $m\ge 1$ be such that $|t|<m^{1/8}$. Let $X$ be distributed as the first coordinate of a uniform point on $\bS^{m-1}$, conditioned on being at least $t/\sqrt{m}$, and let $Y\sim \cN(0,1/m)$ conditioned on being at least $t/\sqrt{m}$. Then
\[ \big|\bE[X]-\bE[Y]\big| \le O\Big(\frac{|t|^{5}+1}{m^{3/2}}\Big)\quad \text{ and }\quad \big|\bE[X^2]-\bE[Y^2]\big| \le O\Big(\frac{|t|^{6}+1}{m^{2}}\Big)\]
\end{lemma}

\subsection{Subgaussian random variables.}
A random variable $X$ is called \textit{subgaussian} if there exists a positive number $\sigma^2$ such that \[\bE\big[\exp\big(t (X-\bE X)\big)\big]\leq \exp\big(\sigma^2t^2/2\big) \text{ for all }t\in \bR.\] The smallest such constant $\sigma^2$ is called the \textit{variance proxy}. Two basic examples of subgaussian variables are Gaussian random variables, as follows from a direct calculation, and bounded random variables, by Hoeffding's lemma.

\begin{lemma}\label{lemma:hoeffding}
Let $Z$ be a random variable with $|Z|\leq M$. Then, $Z$ is subgaussian with variance proxy at most $\sigma^2\leq O(M^2)$.
\end{lemma}

For a proof of Hoeffding's lemma, see  e.g. Exercise 2.24 of \cite{Ver} or Example 2.4 of \cite{Wainwright}. A very important theorem about the subgaussian random variables is the Hanson-Wright inequality, which states the following (for a proof, see the very nice article by Rudelson and Vershynin \cite{RV13}).

\begin{theorem}\label{thm:hanson_wright}
Let $M$ be a symmetric $m\times m$ matrix with zero diagonal and let $X_1, \dots, X_m$ be independent centered subgaussian random variables of variance proxy $\sigma^2$. If $X=(X_1, \dots, X_m)$ and $0<\theta\leq c/(\|M\|_{op}\sigma^2)$, then we have
\[\bE[e^{\theta X^TMX}]\leq e^{C \theta^2 \|M\|_{HS}^2\sigma^4},\]
where $c, C>0$ are absolute constants.
\end{theorem}

In the above statement, $\|M\|_{op}$ denotes the operator norm of $M$, defined by $\|M\|_{op}=\max_{\|v\|_2=1}\|Mv\|_2$, and $\|M\|_{HS}$ denotes its Hilbert-Schmidt (or Frobenius) norm, defined by $\|M\|_{HS}^2=\sum_{i, j=1}^m M_{ij}^2$. In general, we have $\|M\|_{op}\leq \|M\|_{HS}$, and we will use this at various points in the argument.

\subsection{Gram-Schmidt procedure.}\label{subsec:gram_schmidt}
Using the Gram-Schmidt procedure, we can convert the vectors $\bx_i$ into the following form. For a sequence of $n$ linearly independent points $\bx_1, \dots, \bx_n\in \bS^{d-1}$, let us reveal the subspaces ${\rm span}\{\bx_1\}, {\rm span}\{\bx_1, \bx_2\}$, $\dots, {\rm span}\{\bx_1, \dots, \bx_n\}$. Let us denote by $\be_1, \dots, \be_n$ the orthonormal basis of ${\rm span}\{\bx_1, \dots, \bx_n\}$ in which ${\rm span}\{\be_1, \dots, \be_i\}={\rm span}\{\bx_1, \dots, \bx_i\}$ for all $1\leq i\leq n$. Moreover, let $\bv_1, \dots, \bv_n$ be the representation of the vectors in the basis $\be_1, \dots, \be_n$, or more formally, let $\bv_i(j)=\langle \bx_i, \be_j\rangle$. The set of vectors $\bv_1, \dots, \bv_n$ obtained in this way is also known as the Bartlett decomposition.

A very important piece of notation related to this basis will be $\pi_s(\bx)$, which denotes the projection of a vector $\bx$ to the space spanned by $\be_1, \dots, \be_s$.

Note that $\bv_i$ is a unit vector all of whose coordinates after the first $i$ are zero. The crucial property of the vectors $\bv_1, \dots, \bv_n$ which we will use is the following.

\begin{lemma}\label{lemma:distribution_v_i}
Let $1\leq s< i\leq n$. Then, conditional on the coordinates $\bv_i(1), \dots, \bv_i(s-1)$, the value $\bv_i(s)/r_i(s)$ is distributed as the first coordinate of a uniform random point $\by$ on the unit sphere $\bS^{d-s}\subseteq \bR^{d-s+1}$, where $r_i(s)=\sqrt{1-\sum_{k=1}^{s-1} \bv_i(k)^2}$ is the residual length of the vector $\bv_i$ after coordinate $s$. In other words, the first $i-1$ coordinates of $\bv_i$ follow the same distribution as the first $i-1$ coordinates of a uniform random point on the sphere $\bS^{d-1}\subseteq \bR^d$.
\end{lemma}
\begin{proof}
Condition on $\bv_i(1), \dots, \bv_i(s-1)$, or equivalently on the projection of $\bx_i$ to $\spn\{\be_1,\dots,\be_{s-1}\}$. The remaining component of $\bx_i$ lies in the orthogonal complement of this $(s-1)$-dimensional subspace and has length $r_i(s)$. By rotational invariance of the uniform measure on the sphere, its direction is uniform on the unit sphere in this orthogonal complement. Since $\be_s$ is one fixed unit vector in that complement, $\bv_i(s)/r_i(s)=\langle \bx_i-\pi_{s-1}(\bx_i),\be_s\rangle/r_i(s)$ has the same distribution as the first coordinate of a uniform random point on $\bS^{d-s}\subseteq \bR^{d-s+1}$.
\end{proof}

\subsection{\texorpdfstring{The pseudorandomness events $\cE$ and $\cF$}{The events E and F}.}\label{subsec:events}

In Section~\ref{subsec:proof_outline}, we have mentioned that in the proof of Proposition~\ref{prop:main_prop}, we will require a pseudorandomness assumption on the graph $H$ and the vectors defining $G(n, d, p)$. We will now state these conditions precisely, and show that $G(n, d, p)$ satisfies them with high probability.

Let us begin by introducing some notation which will be used throughout the proof. Let us fix a graph $H$ on $n$ vertices. For a pair of distinct vertices $i, j\in [n]$, we define $\lambda_{ij}=a/p$ if $ij\in E(H)$ and $\lambda_{ij}=-a/(1-p)$ otherwise (recall that $a=\phi(z_p)$ and $z_p=\Phi^{-1}(1-p)$ were defined in Section~\ref{subsec:proof_outline}). Observe that if $A$ is the adjacency matrix of $H$, then $\lambda_{ij}=\frac{a}{p(1-p)}(A_{ij}-p)$ for $i\neq j$. Finally, for the sake of convenience we also define $\lambda_{ii}=0$ for all $1\leq i\leq n$.

The event $\cE$ encodes regularity properties of $H$, such as the concentration of degrees and codegrees around their expected values, as well as a bound on the eigenvalues of the adjacency matrix.

Note that one may equivalently state all of the required conditions in terms of the centered adjacency matrix $A-pJ$ and in terms of the above-defined variables $\lambda_{ij}$ (here, $J$ denotes the all-ones matrix). Stating the condition in the first way will make proving it easier, while the application of these conditions in the proof will require the second formulation. Therefore, we write both conditions, keeping in mind that $\lambda_{ij}=\frac{a}{p(1-p)}(A_{ij}-p)=\Theta\big(z_p (A_{ij}-p)\big)$.

\begin{definition}\label{def:graph_event_E_reg}
Let $K^{\cE}$ be a large constant. Define $\cE$ to be the collection of graphs $H$ on the vertex set $[n]$ which satisfy the following inequalities (where $A$ denotes the adjacency matrix of $H$):
\begin{enumerate}[label=(\roman*)]
    \item \label{item:degree_concentration} For $k\in \{1, 2\}$ and any distinct $s_1, \dots, s_k\in [n]$ we have
    \[\sum_{1\leq i\leq n} |A_{s_1i}-p|\cdots |A_{s_ki}-p|\leq K^{\cE}p^k n, \quad \text{ which implies}\quad
    \sum_{1\leq i\leq n} |\lambda_{s_1i}\cdots \lambda_{s_ki}| \leq O(K^{\cE} z_p^k p^k n).\]

    \item \label{item:degree_cancellation} For $k\in \{1, 2\}$, any interval $I\subseteq [n]$ and any distinct $s_1, \dots, s_k\in [n]$ we have
    \[\Big|\!\sum_{i\in I} (A_{s_1i}-p)\cdots (A_{s_ki}-p)\Big|\!\leq\! K^{\cE}\sqrt{p^k n\log n}, \text{ which implies } \Big|\!\sum_{i\in I} \!\!\lambda_{s_1i}\cdots \lambda_{s_ki}\Big|\!\leq \!O(K^{\cE}z_p^k \sqrt{p^k n\log n}).\]

    \item \label{item:eigenvalue_bound} We have $\|A-pJ\|_{op}\leq K^{\cE}\sqrt{np}$. If the $n\times n$ matrix $\Lambda$ is defined by $\Lambda_{ij}=\lambda_{ij}$, this also implies $\|\Lambda\|_{op}\leq O(K^{\cE}z_p\sqrt{np})$.

    \item \label{item:triangle_concentration} For any $s\in [n]$ we have \[\sum_{s<i<j\leq n}\big|(A_{si}-p)(A_{sj}-p)(A_{ij}-p)\big|\leq K^{\cE}p^3 n^2, \text{which implies} \sum_{s<i<j\leq n}\big|\lambda_{si}\lambda_{sj}\lambda_{ij}\big|\leq O(K^{\cE}n^2p^3z_p^3).\]
\end{enumerate}
\end{definition}

\begin{lemma}\label{lemma:geometric_graph_event_E}
There exist absolute constants $K^{\cE}, C$ such that for $1/3\geq p\geq n^{-1/3}\polylog(n)$ and $d\geq C n^2p\polylog(n)$ we have $\Pb[G(n, d, p)\in \cE]=1-o(1)$ as $n\to \infty$.
\end{lemma}

Let us note that it should not be surprising that any of these pseudorandomness properties should hold with high probability for $G(n, d, p)$. Indeed, all of these properties can be verified rather simply for the corresponding binomial random graph $G(n, p)$. So, if $G(n, d, p)$ exhibited a different behavior than $G(n, p)$ in regard to any of these properties, we would have a test for distinguishing these two models even above the $(np\log p^{-1})^3$ threshold.

Let us now discuss the pseudorandomness properties of the vectors $\bv_1, \dots, \bv_n$.

\begin{definition}[Vector event $\cF$.]\label{def:vector_event_F}
Fix an integer $0\leq s\leq n$, and suppose that $\bv_1, \dots, \bv_n$ is an $n$-tuple of vectors sampled from the unit sphere $\bS^{d-1}\subseteq \bR^d$.
\begin{enumerate}[label=(\roman*)]
    \item \label{item:entry_size} Let $\cT_s$ (\textit{the truncation event}) be the event that for any $1\leq s<i\leq n$ we have
    \[|\bv_i(s)|\leq 10\sqrt{\frac{\log d}{d}}.\]
    \item \label{item:norm_concentration}\label{item:inner_product} Let $\Pi_s$ (\textit{the projection event}) be the event that
    \[ \Big| \|\pi_s(\bv_i)\|_2^2-\frac{s}{d}\Big|\leq 10\frac{\sqrt{n\log d}}{d}\text{ for all $s<i\leq n,$ and }  \big|\langle \pi_s(\bv_i),\pi_s(\bv_j)\rangle\big|\leq 10\frac{\sqrt{n\log d}}{d}\text{ for all $s<i<j\leq n$}.\]
\end{enumerate}
Finally, set $\cF=\bigwedge_{s=1}^n(\cT_s\wedge \Pi_s)$.
\end{definition}

Note that $\Pi_s$ depends only on the coordinates $\bv_i(j)$ with $j\leq s$.

\begin{lemma}\label{lemma:vector_event_F}
Let $\bv_1,\dots,\bv_n$ be sampled according to Lemma~\ref{lemma:distribution_v_i} (i.e. a Bartlett decomposition of an $n$-tuple of independent uniform vectors on $\bS^{d-1}$), and assume that $n\leq d/2$. Then, we have \[\Pb[(\bv_1,\dots,\bv_n)\text{ satisfy } \cF]=1-o(1) \text{ as }n\to \infty.\]
\end{lemma}

Throughout the proof, we will often work only on the event $\cF$ or on the event $\cT_s$. Therefore, in several sections we will find it very useful to work with truncated variables $\bu_i(s)$ instead of $\bv_i(s)$, where we define $\bu_i(s)=\bv_i(s)\mathbf{1}_{|\bv_i(s)|\leq 10\sqrt{\log d}/\sqrt{d}}$.

\section{Pointwise probability bound and the proof of Theorem~\ref{thm:main}}\label{sec:main_estimate}

In this section, we discuss the centerpiece of our paper, Proposition~\ref{prop:main_prop}, which gives a very precise bound on $\Pb[G(n, d, p)=H\wedge \cF]$. We begin the section by stating and discussing this estimate. Once we do this, we will derive Theorem~\ref{thm:main} from it.

To state Proposition~\ref{prop:main_prop}, we will need a piece of notation. Recall, we have already defined the signed triangle count as $N_{\triangle}(H)=\sum_{1\leq i<j<k\leq n} (A_{ij}-p)(A_{jk}-p)(A_{ik}-p)$. In general, we define the signed count of a fixed graph $F$ in $H$ as
\[N_F(H)=\sum_{F\subseteq K_n} \prod_{ij\in E(F)} (A_{ij}-p),\]
where the sum runs over all unlabeled copies of $F$ in $K_n$. The signed count of four-cycles is denoted by $N_{\square}(H)$ and the signed count of diamonds is denoted by $N_{\diamond}(H)$ (in this paper, a diamond is the unique graph with four vertices and five edges).

\begin{proposition}\label{prop:main_prop}
For every constant $K^{\cE}>0$, there is a constant $C$ such that the following holds.
Let $H\in \cE$ be a fixed graph on $[n]$, let $p$ be a probability such that $1/3\geq p\geq n^{-1/3} (\log n)^A$ (for a sufficiently large absolute constant $A$) and let $d\geq C n^2p (\log n)^A$. Then
\begin{align}\label{eqn:full_estimate}
\Pb\big[G(n,d,p)=\!H\wedge \cF\big]\!\leq \!p^{e(H)}(1\!-\!p)^{\binom{n}{2}-e(H)}\!
\exp\Big(\frac{\kappa_{\triangle}N_{\triangle}(H)}{\sqrt{d}}
\!+\!\frac{\kappa_\square N_{\square}(H)}{d}\!+\!\frac{\kappa_{\diamond} N_{\diamond}(H)}{d}
\!+\!\widetilde O\Big(n R_{n, d, p}\!\Big)\!\Big),
\end{align}
where $\kappa_\triangle=\frac{a^3}{p^3(1-p)^3}, \kappa_\square=\frac{a^4}{p^4(1-p)^4}-\frac{2a^6}{p^5(1-p)^5}$, $\kappa_{\diamond}=\frac{a^5}{p^5(1-p)^5}\Big(z_p-\frac{a(1-2p)}{p(1-p)}\Big)$, and the error term $R_{n, d, p}$ is given by $R_{n, d, p}=\frac{n^{3/2}p^{1/2}}{d}+\frac{n^3 p^2+n^{5/2}}{d^{3/2}}+\frac{n^4p^2}{d^2}$.
\end{proposition}

Let us make some remarks about the statement of Proposition~\ref{prop:main_prop}. Firstly, observe that the estimate (\ref{eqn:full_estimate}) works well below the distinguishability threshold $d\gg (np\log 1/p)^3$. While $p^{e(H)}(1-p)^{\binom{n}{2}-e(H)}$ is the probability of obtaining $H$ as a sample from $G(n, p)$, one should think of ${\kappa_{\triangle}N_{\triangle}(H)}/{\sqrt{d}}$ as the leading order correction, determined by the signed triangle count of $H$. In fact, the above estimate shows that the real reason the threshold $d\gg (np\log 1/p)^3$ enters the picture is so that the main term $\kappa_{\triangle}N_{\triangle}(H)/\sqrt{d}$ would be negligible. More on this will be said later in this section, when we derive Theorem~\ref{thm:main} from Proposition~\ref{prop:main_prop}.

Further, it is the shape of the error term $R_{n, d, p}$ that determines the range of $p$ in which we can prove Conjecture~\ref{conj:BDER}. Namely, $nR_{n, d, p}$ contains the term $n^{5/2}p^{1/2}/d$, which decays only when $p\gg n^{-1/5} {\rm polylog}(n)$ (if we insist on taking $d\approx (np\log 1/p)^3$).

Secondly, although \eqref{eqn:main_estimate} is stated as an upper bound, most of the steps in our proof are not lossy, and we believe that it may be possible to turn it into an equality (see additional discussion on this around Lemma~\ref{lemma:ratio}). However, for the purposes of proving Theorem~\ref{thm:main}, the upper bound stated in \eqref{eqn:full_estimate} suffices, and we do not concern ourselves with establishing the lower bound.

Finally, another interesting feature about the statement of Proposition~\ref{prop:main_prop} are the curious coefficients $\kappa_{\triangle}, \kappa_{\square}$ and $\kappa_{\diamond}$ which appear in it. While understanding these values may not be directly relevant for the proof of Theorem~\ref{thm:main}, we feel that these observations may help to understand the distribution of spherical RGGs.

Due to the definition of $G(n, d, p)$, we have already observed that a triple of vertices $i, j, k\in [n]$ is more likely to form a triangle in $G(n, d, p)$ than in $G(n, p)$. Indeed, knowing that $ij, ik$ are edges means that both $\bv_j$ and $\bv_k$ are close to $\bv_i$ and so more likely to be close themselves. Another way to put this is to imagine that $\bv_i$ is one of the coordinate vectors, and observe that since $ij, ik\in E(G)$, we must have $\langle \bv_i, \bv_j\rangle, \langle \bv_i, \bv_k\rangle>c_p/\sqrt{d}$. This means that both $\bv_j$ and $\bv_k$ have a positive coordinate in direction $\bv_i$, making their inner product more likely to be large. One can even calculate precisely that $\Pb[ijk\text{ form }\triangle]= p^3\exp\big(\frac{a^3}{p^3\sqrt{d}} +\widetilde O(1/d)\big)=p^3\exp\big(\kappa_\triangle (1-p)^3/\sqrt{d}+\widetilde O(1/d)\big)$.

Analogously, the probability that $ij, ik\in E(G)$ but $jk\notin E(G)$ (i.e. $ijk$ forms a cherry) is slightly smaller than in $G(n, p)$, in fact, it equals $\Pb[ij, ik\in E(G), jk\notin E(G)]= p^2(1-p)\exp\big(-\kappa_{\triangle} p(1-p)^2/\sqrt{d} +\widetilde O(1/d)\big)$. Having written this equation, one observes that the triangle probability must be corrected by a factor depending on $\kappa_\triangle (1-p)^3/\sqrt{d}$, while the cherry probability is corrected by a factor depending on $-\kappa_{\triangle} p(1-p)^2/\sqrt{d}$, which exactly corresponds to the weights with which the triangles and the cherries are counted in the signed triangle count $N_{\triangle}(H)$. If one goes on and computes the probabilities of appearance of other three-vertex graphs in $G(n, d, p)$, it will be clear that the pattern persists: the correction in case $ij\in E(G), jk, ik\notin E(G)$ is proportional to $\kappa_\triangle p^2 (1-p)/\sqrt{d}$, and if $ij, jk, ik\notin E(G)$, the correction is proportional to $-\kappa_\triangle p^3/\sqrt{d}$.

In light of this observation, the leading order of the estimate (\ref{eqn:full_estimate}) can be represented as
\[\exp\Big(\frac{\kappa_{\triangle}N_{\triangle}(H)}{p^3(1-p)^3\sqrt{d}}\Big)\approx\prod_{1\leq i<j<k\leq n}\frac{\Pb[i, j, k\text{ form the same pattern in $G(n, d, p)$ and $H$}]}{\Pb[i, j, k\text{ form the same pattern in $G(n, p)$ and $H$}]}.\]
The interpretation of this formula is quite simple: if $H$ prescribes a pattern of edges on $i, j, k$ which is easier to realize in $G(n, d, p)$, that should boost the probability of obtaining $H$ as a sample from $G(n, d, p)$. Moreover, the contribution of all three-vertex patterns must be aggregated together in order to obtain a total correction to the probability of obtaining $H$ as a sample from $G(n, d, p)$.

At first glance, it may not be obvious why the corrections corresponding to distinct triples $i, j, k$ should be multiplied, i.e. why the effects of different triples add together independently. Here is a piece of intuition which we found useful while thinking about these graphs. Suppose we sample the vectors $\bv_1, \dots, \bv_n$ in order, and suppose that a triple $1\leq i<j<k\leq n$ forms a triangle in $H$. Having sampled vectors $\bv_1, \dots, \bv_{j-1}$ already, one must ensure that $\bv_j, \bv_k$ are close to $\bv_i$, and hence a bit closer to each other. In other words, the vertex $i$ ``votes'' that the vectors $\bv_j, \bv_k$ are close to each other. Suppose that there was now another vertex $i'<j<k$ also forming a triangle with $j$ and $k$. Then, $i'$ also ``votes'' that $\bv_j, \bv_k$ are close to $\bv_{i'}$, and hence again a bit closer to each other. But the vectors $\bv_i$ and $\bv_{i'}$ are nearly orthogonal (since $d\gg n$)! So, the information about the edge $jk$ obtained from the votes of $i$ and $i'$ is essentially independent.

Put differently, suppose that the orthogonal vectors $\bv_i$ and $\bv_{i'}$ are viewed as two coordinate directions. The conditions
that $ij, ik, i'j, i'k\in E(G)$ 
then force both $\bv_j$ and $\bv_k$ to have large positive projections onto each of the directions $\bv_i$ and $\bv_{i'}$. Since gaussian vectors are rotationally invariant  these events happen independently. Moreover this
makes the inner product of $\bv_j, \bv_k$  even more likely to be at least $c_p/\sqrt{d}$. Hence, the probability that $jk$ is an edge must be corrected twice as much as if we only know that $ijk$ is a triangle.

Having considered triangles and other three-vertex subgraphs, a natural question arises: why stop there? What are the effects of larger subgraph counts in $H$ on $\Pb[G(n, d, p)=H]$? Here, the story is broadly the same, and we will only describe it in broad terms. Again, intuitively speaking, the probability that the quadruple $i<j<k<\ell$ forms a four-cycle in $G(n, d, p)$ is slightly higher than in $G(n, p)$. The justification is similar to before: if $\bv_i, \bv_j$ are already sampled and we know they form an edge, they must be somewhat close. Also, since $i\ell, jk$ are edges, then the pairs $(\bv_i, \bv_\ell)$ and $(\bv_j, \bv_k)$ must be somewhat close. So, the likelihood that $\bv_\ell, \bv_k$ are sufficiently close increases slightly. Of course, since the path linking $k$ and $\ell$ is longer, it is expected that the strength of this effect would be smaller. This is indeed true: the correction per $4$-cycle is on the order of $1/d$ rather than $1/\sqrt{d}$ (as it was for triangles). So, by analyzing carefully the probabilities of obtaining $4$-cycles and their various induced subgraphs, we arrive at the conclusion that the second correction term in (\ref{eqn:full_estimate}) can be written as
\[\exp\Big(\frac{a^4}{p^4(1-p)^4d}N_{\square}(H)\Big)\approx\prod_{(i, j, k, \ell)}\frac{\Pb[G(n, d, p)\text{ agrees with $H$ on the edges $ij, jk, k\ell, \ell i$}]}{\Pb[G(n, p)\text{ agrees with $H$ on the edges $ij, jk, k\ell, \ell i$}]}.\]

The last term comes from the count of diamonds in $H$, and it is more subtle for the following reason. A diamond contains two triangles and a four-cycle, and these subgraphs already affect the probability of obtaining $H$ as a RGG. However, if $H$ has a diamond, there are additional dependencies between edges, which go beyond just having two triangles and a four-cycle. Hence, if one wants to account for these dependencies, we must adjust by a factor of $\Pb[ijk\ell=\diamond\text{ in }G(n, d, p)]/(\Pb[ijk=\triangle\text{ in }G(n, d, p)]^2\Pb[ijk\ell=\square\text{ in }G(n, d, p)])$ for each diamond present in $H$. This is the reason for the more complicated coefficients $\kappa_{\square}$ and $\kappa_{\diamond}$: in particular, this adjustment contributes the negative second term in $\kappa_{\square}$. The additional contributions of yet larger graphs, such as $K_4$ or graphs on at least $5$ vertices are of lower order, and they are absorbed in the error term.

\subsection{Proof of Theorem~\ref{thm:main}}

In this section, we discuss how Proposition~\ref{prop:main_prop} implies Theorem~\ref{thm:main}. To do this, we discuss the typical sizes of $N_\triangle(H)$, $N_{\square}(H)$, and $N_{\diamond}(H)$. More precisely, we define the collection of graphs $\cD$ in which the signed counts of these subgraphs are controlled, and show that it suffices to compare the distributions $G(n, p)$ and $G(n, d, p)$ just on $\cD$.

\begin{definition}[Subgraph count event $\cD$.]
Let $K^{\cD}$ be a large constant. Define $\cD$ to be the collection of graphs $H$ on the vertex set $[n]$ for which we have
\begingroup
\renewcommand{\theenumi}{(\roman{enumi})}
\setcounter{enumi}{0}
\refstepcounter{enumi}\label{item:triangle_cancellation}
\refstepcounter{enumi}\label{item:c4_cancellation}
\refstepcounter{enumi}\label{item:diamond_cancellation}
\[\textup{\ref{item:triangle_cancellation}}\ \big|N_{\triangle}(H)\big|\leq K^{\cD} n^{3/2}p^{3/2},\quad \textup{\ref{item:c4_cancellation}}\ \big|N_{\square}(H)\big|\leq K^{\cD} n^{2}p^{2},\quad \textup{\ref{item:diamond_cancellation}}\ \big|N_{\diamond}(H)\big|\leq K^{\cD} n^{2}p^{5/2}.\]
\endgroup
\end{definition}

\begin{lemma}\label{lemma:geometric_graph_event_D}
For every $\eps>0$ there exists $K^{\cD}>0$, which may depend on $\eps$, such that the following holds. Let $G\sim G(n, d, p)$ with $d\geq Cn^3p^3(\log p^{-1})^3$ and $n^{-1/5}(\log n)^A\leq p\leq 1/3$, for $C$ sufficiently large compared to $K^{\cD}$. Then, we have $\Pb[G\in \cD]\geq 1-\eps$.
\end{lemma}

Lemma~\ref{lemma:geometric_graph_event_D} should not be surprising at all, at least if one believes Conjecture~\ref{conj:BDER}. Indeed, we know that $G(n,p)$ typically satisfies the event $\cD$, and so if $G(n, d, p)$ did not typically satisfy it, then this could be used to distinguish the two models even above the distinguishability threshold. However, recent work of Bok, Li and Yu \cite{BLY26} and Bangachev and Bresler \cite{BB24} shows that $G(n, d, p)$ and $G(n, p)$ are not distinguishable using constant-degree polynomials, meaning that Lemma~\ref{lemma:geometric_graph_event_D} should hold. To avoid breaking the flow of the proof, we discuss its proof in Section~\ref{sec:geometric_graph_event_D}.

Note that the fact that $K^{\cD}$ may depend on $\eps$ is not problematic, since in the context of the proof of Theorem~\ref{thm:main}, our goal will be to take $\Pb[G(n, d, p)\in \cD]\geq 1-\eps/4$. Since we will work with a single fixed value of $\eps$ throughout the proof, we will fix the appropriate value of $K^{\cD}$ making these inequalities hold. We now discuss the proof of Theorem~\ref{thm:main} assuming Proposition~\ref{prop:main_prop}.

\begin{proof}[Proof of Theorem~\ref{thm:main} assuming Proposition~\ref{prop:main_prop}.]
Let us fix a positive constant $\eps>0$, with the goal of showing that $\TV(G(n, d, p), G(n, p))\leq \eps$. Let $\mu$ and $\nu$ be the probability measures induced by the models $G(n, p)$ and $G(n, d, p)$, respectively. By a slight abuse of notation, we will write $\nu$ for both the measure on the space of $n$-tuples of latent vectors, as well as for its induced measure on the space of graphs.

By Lemmas~\ref{lemma:geometric_graph_event_E},~\ref{lemma:vector_event_F} and~\ref{lemma:geometric_graph_event_D}, if we choose large enough constants $K^{\cE}, K^{\cD}$, we have that
\[\nu(\cE)\geq 1-\eps/4,\quad \nu(\cD)\geq 1-\eps/4, \quad \text{ and }\quad \nu(\cF)\geq 1-\eps/4.\]

Let us denote by $\cH$ the set of $n$-vertex graphs $H$ for which $\nu(H)>\mu(H)$. Since the total mass over all graphs is $1$ under both $\mu$ and $\nu$, we have $\nu(\cH)-\mu(\cH)=\mu(\cH^c)-\nu(\cH^c)$.
Hence, we can write the total variation distance as follows:
\begin{align*}
\TV(\nu,\mu)&=\nu(\cH)-\mu(\cH)\leq \nu(\cH\wedge \cF)-\mu(\cH)+\nu(\cF^c).
\end{align*}
The last term is at most $\nu(\cF^c)\leq \eps/4$. On the other hand, the first term can be further bounded as follows
\begin{align*}
\TV(\nu,\mu)&\leq \sum_{H\in \cE\cap \cD\cap \cH} \big(\nu(H\wedge \cF)-\mu(H)\big)+\sum_{H\in (\cE^c\cup \cD^c)\cap \cH} \big(\nu(H\wedge \cF)-\mu(H)\big)+\eps/4.
\end{align*}
The second sum is at most $\nu(\cE^c\cup \cD^c)\leq \nu(\cE^c)+\nu(\cD^c)\leq \eps/2$. Thus,
\begin{align*}
\TV(\nu,\mu)&\leq \sum_{H\in \cE\cap \cD\cap \cH} \big(\nu(H\wedge \cF)-\mu(H)\big)+\frac{3\eps}{4}.
\end{align*}
Further, by Proposition~\ref{prop:main_prop}, for every $H\in \cE\cap \cD\cap \cH$ we have
\begin{align*}
\nu(H\wedge \cF)&\leq \mu(H)\exp\Big(\frac{\kappa_{\triangle}N_{\triangle}(H)}{\sqrt{d}}
\!+\!\frac{\kappa_\square N_{\square}(H)}{d}\!+\!\frac{\kappa_{\diamond} N_{\diamond}(H)}{d}
\!+\!\widetilde O\Big(n R_{n, d, p}\!\Big)\!\Big),
\end{align*}
where $\kappa_\triangle=\frac{a^3}{p^3(1-p)^3}, \kappa_\square=\frac{a^4}{p^4(1-p)^4}-\frac{2a^6}{p^5(1-p)^5}$, and $\kappa_{\diamond}=\frac{a^5}{p^5(1-p)^5}\Big(z_p-\frac{a(1-2p)}{p(1-p)}\Big)$. Since $a=\Theta(p\sqrt{\log 1/p})$, we have $\kappa_\triangle=\Theta((\log 1/p)^{3/2})$, and $\kappa_\square=O((\log 1/p)^2)$, $\kappa_{\diamond}=O((\log 1/p)^{5/2})$. Further, since $H\in \cD$, we have $|N_{\triangle}(H)|\leq K^{\cD} n^{3/2}p^{3/2}$, $|N_{\square}(H)|\leq K^{\cD} n^2p^2$, and $|N_{\diamond}(H)|\leq K^{\cD} n^2p^{5/2}$. Hence, for $H\in \cD$ we have $\kappa_\square N_{\square}(H)/d=O(n^2p^2(\log 1/p)^2/d)=o(1)$ and $\kappa_{\diamond}N_{\diamond}(H)/d=O(n^2p^{5/2}(\log 1/p)^{5/2}/d)=o(1)$, if $d\geq (np\log 1/p)^3$ and $p\geq n^{-1/5}$ (say).

On the other hand, we have $nR_{n, d, p}=\frac{n^{5/2}p^{1/2}}{d}+\frac{n^4p^2+n^{7/2}}{d^{3/2}}+\frac{n^5p^2}{d^2}$, and if $p\geq n^{-1/5}(\log n)^A$ and $d\geq (np\log 1/p)^3$, then we have
\[\widetilde O(nR_{n, d, p})=\widetilde O\Big(\frac{n^{5/2}p^{1/2}}{n^3p^3}+\frac{n^4p^2+n^{7/2}}{n^{9/2}p^{9/2}}+\frac{n^5p^2}{n^6p^6}\Big)=\widetilde O\Big(\frac{1}{n^{1/2}p^{5/2}}+\frac{1}{np^{9/2}}+\frac{1}{np^4}\Big)=o(1).\]

Hence, we can write $\nu(H\wedge \cF)\leq \mu(H)\exp\Big(\frac{\kappa_{\triangle}N_{\triangle}(H)}{\sqrt{d}}+o(1)\Big)$. By property~\ref{item:triangle_cancellation} of the event $\cD$, we get that $|N_\triangle(H)|\leq K^{\cD} p^{3/2} n^{3/2}$. Recalling that $d\geq Cn^3p^3 (\log p^{-1})^3$, we have $\frac{\kappa_\triangle N_{\triangle}(H)}{\sqrt{d}}\leq \frac{K^{\cD}\sqrt{p^3n^3(\log p^{-1})^3}}{\sqrt{Cn^3p^3 (\log p^{-1})^3}}\leq \eps/6$, as long as $C$ is sufficiently larger than $K^{\cD}$.

The conclusion is that $\nu(H\wedge \cF)/\mu(H)\leq e^{\eps/6+o(1)}\leq e^{\eps/5}\leq 1+\frac{\eps}{4}$. Hence, we have that
\begin{align*}
\TV(\nu,\mu)&\leq \sum_{H\in \cH\cap \cE\cap \cD}\big(\nu(H\wedge \cF)-\mu(H)\big)+\frac{3\eps}{4}\leq \sum_{H\in \cH\cap \cE\cap \cD}\frac{\eps}{4}\mu(H)+\frac{3\eps}{4}\leq \frac{\eps}{4}+\frac{3\eps}{4}=\eps.    \qedhere
\end{align*}
\end{proof}

\section{Vertex exposure and the main induction}\label{sec:induction}

In this section we will begin our work towards the proof of Proposition~\ref{prop:main_prop}. Throughout the remainder of the paper, that is in Sections~\ref{sec:induction}--\ref{sec:MGF}, the hypotheses of Proposition~\ref{prop:main_prop} remain in force. In particular, $d\geq Cn^2p(\log n)^A$, $p\geq n^{-1/3}(\log n)^A$ for a large absolute constant $A$.
As we have already indicated, the proof will be inductive, and we will now give the exact inductive hypothesis.

We write $\cG_s$ for the $\sigma$-algebra generated by the first $s$ coordinates of all vectors $\bv_1,\dots,\bv_n$, i.e. by the coordinates $\bv_i(j)$ for all $1\leq i\leq n$ and $1\leq j\leq s$. Thus $\cG_0$ is trivial and $\cG_n$ determines the whole graph $G(n,d,p)$. Let us also observe that $\cG_{s-1}$ determines $\bv_s(s)$, since $\|\bv_s\|^2=1$. For $0\leq s\leq n-1$, let $H_s$ denote the subgraph of $H$ induced by the vertex set $\{s+1,\dots,n\}$. Additionally, recall that we have defined $\lambda_{ij}=\frac{a}{p(1-p)}(A_{ij}-p)$, where $A$ is the adjacency matrix of $H$.

\begin{lemma}[Induction statement]\label{lemma:induction-statement}
Assume that $H\in \cE$. Then for every $0\leq s\leq n-1$,
{\relscale{0.95}\begin{align*}
\Pb\Big[G(n,d,p)\big|_{[s+1,n]}\!=\!H_s\wedge \cF \Big|\cG_s\Big]
\leq \!\!\!\!\!\!\prod_{s<i<j\leq n}\!\!\!\!\!\Phi_{sij}
\exp\!\Big(Z_s + &\frac{\kappa_{\triangle}N_{\triangle}(H_s)}{\sqrt{d}}
+\frac{\kappa_\square N_{\square}(H_s)}{d}+\frac{\kappa_{\diamond} N_{\diamond}(H_s)}{d}
+ \widetilde O\Big((n\!-\!s)R_{n, d, p}\Big)\!\Big),
\end{align*}}
where $\kappa_\triangle=\frac{a^3}{p^3(1-p)^3}, \kappa_\square=\frac{a^4}{p^4(1-p)^4}-\frac{2a^6}{p^5(1-p)^5}$, and $\kappa_{\diamond}=\frac{a^5}{p^5(1-p)^5}\Big(z_p-\frac{a(1-2p)}{p(1-p)}\Big)$,
\[
\Phi_{sij}=
\begin{cases}
\Phi\Big(-z_p+\sqrt d\,\langle \pi_s(\bv_i),\pi_s(\bv_j)\rangle\Big),& ij\in E(H),\\[2mm]
\Phi\Big(z_p-\sqrt d\,\langle \pi_s(\bv_i),\pi_s(\bv_j)\rangle\Big),& ij\notin E(H),
\end{cases}
\]
and
$Z_s=\sum_{\substack{s<i< j\leq n\\s<r<n}}\lambda_{ij}\lambda_{ir}\lambda_{jr}(z_p-\lambda_{ij})\langle \pi_s(\bv_i), \pi_s(\bv_j)\rangle .$
\end{lemma}

Note that Proposition~\ref{prop:main_prop} follows directly from Lemma~\ref{lemma:induction-statement} by setting $s=0$. Since $\pi_0(\bv_i)$ is the zero vector for all $i$, we have $Z_0=0$ and $H_0=H$. Finally, if $ij$ is an edge of $H$, then $\Phi_{0ij}=\Phi(-z_p)=p$ and if $ij$ is a nonedge, then $\Phi_{0ij}=\Phi(z_p)=1-p$. Hence, $\prod_{1\leq i<j\leq n}\Phi_{0ij}=p^{e(H)}(1-p)^{\binom{n}{2}-e(H)}$, and we get precisely the bound of Proposition~\ref{prop:main_prop}.

Let us try to provide some intuition for the statement of Lemma~\ref{lemma:induction-statement}. As we already indicated, the idea is to describe the probability that $G(n, d, p)$ agrees with $H$ on vertices $\{s+1, \dots, n\}$. One should think of the product $\prod_{s<i<j\leq n} \Phi_{sij}$ as the leading term of this probability, where $\Phi_{sij}$ corresponds to our current estimate of the probability of the edge $ij$ appearing in $G(n, d, p)$, given the coordinates $\bv_i(1), \dots, \bv_i(s), \bv_j(1),\dots, \bv_j(s)$. Indeed, if $ij\in E(H)$ should be an edge, then $\Phi_{sij}$ equals $\Phi(-z_p)=p$ to first order. Furthermore, there is a correction term based on the already revealed coordinates - if, say, the vectors $\bv_i, \bv_j$ already made progress towards having a large inner product due to $\langle \pi_s(\bv_i), \pi_s(\bv_j)\rangle$ being large and positive, then we should increase the expected probability of the edge $ij$ in the graph $G(n, d, p)$.

Beyond this, there are correction terms similar to those we have already seen in Proposition~\ref{prop:main_prop}. These terms should not be surprising: if $H_s$ has an unfavorable triangle arrangement for being realized as a sample of $G(n, d, p)$ (i.e. it has a very low signed triangle count), then this should decrease the probability of obtaining it on the vertices $s+1, \dots, n$ in $G(n,d,p)$. Finally, the term $Z_s$ describes the influence of the coordinates we have already seen on the realization of the remaining graph.

We now discuss the proof of Lemma~\ref{lemma:induction-statement}. The basic idea of the proof is to argue by induction, starting from $s=n-1$ all the way to $s=0$. We observe that $G(n, d, p)$ and $H$ agree on vertices $s, \dots, n$ if and only if they agree on the edges incident to the vertex $s$ and also agree on vertices $s+1, \dots, n$. The probability of the latter can be controlled through the induction hypothesis, while the probability of the former can be explicitly calculated. However, conditioning on the first event affects the distribution of the coordinates $\bv_i(s)$ for $i\in \{s+1, \dots, n\}$, and our first steps will be to precisely understand the effects of this conditioning.

\begin{lemma}\label{lemma:vertex_s}
Let $1\leq s< n$ be an integer and let $H\in \cE$ be an $n$-vertex graph. Suppose that the $\bv_1, \dots, \bv_n$ are random vectors on $\bS^{d-1}\subset \bR^d$, sampled as in Lemma~\ref{lemma:distribution_v_i}, and let $G$ be the corresponding random geometric graph. Suppose that the set of coordinates $\bv_i(j)$ are revealed for $j\leq s-1$, denote this set $\cG_{s-1}$, and assume that they satisfy the event $\Pi_{s-1}$.

\begin{enumerate}
\item[(a)] The probability, over the remaining randomness, that $G$ and $H$ agree on edges $si$ for all $i\in \{s+1, \dots, n\}$ is
\begin{align}\label{eqn:prob_E_si}
    \Pb\Big[si\in E(G)\Longleftrightarrow si\in E(H)\text{ for all } i\in \{s+1, \dots, n\}\Big| \cG_{s-1}\Big]\leq e^{\widetilde O(n^{3/2}p^{1/2}/{d})}\prod_{s<i\leq n}\Phi_{s-1, s, i}.
\end{align}
\item[(b)] Conditioning on $\cG_{s-1}$ and on the event that $G, H$ agree on the edges $si$ for all $i\in \{s+1, \dots, n\}$, the coordinates $\bv_{s+1}(s), \dots, \bv_n(s)$ have the following distribution:\\
\begin{center}
\fbox{
\parbox{15cm}{\phantomsection\label{setup:star}\emph{Setup ($\star$)}: Suppose that the set of coordinates $\{\bv_i(j)| i\in [n], 1\leq j\leq s-1\}$ is revealed and satisfies the event $\Pi_{s-1}$.
For $i\in \{s+1, \dots, n\}$, let $r_i(s)=\sqrt{1-\sum_{k=1}^{s-1} \bv_i(k)^2}$ be the residual length of the vector $\bv_i$ after coordinate $s$. If $si\in E(H)$, then $\bv_i(s)/r_i(s)$ is distributed as the first coordinate on the unit sphere $\bS^{d-s}\subseteq \bR^{d-s+1}$, conditioned on being at least \[\tau_{si}=\frac{c_p}{\sqrt{d}\bv_s(s)r_i(s)}-\frac{\langle \pi_{s-1}(\bv_i), \pi_{s-1}(\bv_s)\rangle}{\bv_s(s)r_i(s)}.\]
If $si\notin E(H)$, then $\bv_i(s)/r_i(s)$ is distributed as the first coordinate on the unit sphere $\bS^{d-s}\subseteq \bR^{d-s+1}$, conditioned on being smaller than $\tau_{si}$.

Moreover, all the coordinates $\bv_i(s)$ are independent.

Additionally, define $\bu_i(s)$ as the truncated random variable $\bu_i(s)=\bv_i(s)\mathbf{1}_{|\bv_i(s)|\leq 10\sqrt{\log d}/\sqrt{d}}$}
}
\end{center}
\end{enumerate}
\end{lemma}
\begin{proof}
For $s<i\leq n$, define the event $E_{si}$ by
\[
E_{si}=
\begin{cases}
\{\langle \bv_i,\bv_s\rangle\geq c_p/\sqrt d\}=\{\bv_i(s)\bv_s(s)+\langle \pi_{s-1}(\bv_i), \pi_{s-1}(\bv_s)\rangle \geq c_p/\sqrt{d}\},& si\in E(H),\\
\{\langle \bv_i,\bv_s\rangle<c_p/\sqrt d\}=\{\bv_i(s)\bv_s(s)+\langle \pi_{s-1}(\bv_i), \pi_{s-1}(\bv_s)\rangle < c_p/\sqrt{d}\},& si\notin E(H).
\end{cases}
\]
Conditioned on $\cG_{s-1}$, the only random quantity involved in the event $E_{si}$ is the coordinate $\bv_i(s)$. Indeed, if $si\in E(H)$, $E_{si}$ occurs precisely when $\bv_i(s)\geq c_p/(\sqrt{d}\bv_s(s))-\langle \pi_{s-1}(\bv_i), \pi_{s-1}(\bv_s)\rangle/\bv_s(s)$, and if $si\notin E(H)$, $E_{si}$ occurs when $\bv_i(s)$ is smaller than the same threshold. Since the coordinates $\bv_i(s)$ are independent before conditioning on $\bigwedge_{s<i\leq n}E_{si}$, they remain independent after conditioning too. Hence, Lemma~\ref{lemma:distribution_v_i} shows they follow precisely the distribution from \setupstar{}. This shows part (b) of the statement.

Let us now focus on proving the estimate (\ref{eqn:prob_E_si}). Since the events $E_{si}$ are independent after conditioning on $\cG_{s-1}$, we have
\[ \Pb\Big[\bigwedge_{s<i\leq n}E_{si}\,\Big|\cG_{s-1}\Big] = \prod_{s<i\leq n}\Pb\big[E_{si}| \cG_{s-1}\big].\]

Let us denote by $t_{si}=z_p-\sqrt{d}\langle \pi_{s-1}(\bv_i), \pi_{s-1}(\bv_s)\rangle$, so that $(t_{si}+c_p-z_p)/(\sqrt{d}\bv_s(s))$ is the threshold for the event $E_{si}$. Due to Lemma~\ref{lemma:distribution_v_i}, $\bv_i(s)$ conditioned on $\pi_{s-1}(\bv_i)$ is distributed as $\sqrt{1-\|\pi_{s-1}(\bv_i)\|^2}$ times the first coordinate of a uniform random point $\by\sim \bS^{d-s}\subseteq \bR^{d-s+1}$. If $si\in E(H)$, we have
\begin{align}\label{eqn:prob_E_si_via_Psi_1}
\Pb[E_{si}|\cG_{s-1}] = \Pb\Big[\bv_i(s)\geq \frac{t_{si}+c_p-z_p}{\sqrt{d}\bv_s(s)}\Big|\cG_{s-1}\Big] &= \Pb\Big[\by_1 \sqrt{1-\|\pi_{s-1}(\bv_i)\|^2} \geq \frac{t_{si}+c_p-z_p}{\sqrt{d}\bv_s(s)}\Big|\cG_{s-1}\Big]\nonumber\\
&= \Psi_{d-s+1}\Big(\sqrt{\frac{d-s+1}{d}}\cdot \frac{-t_{si}-c_p+z_p}{\bv_s(s)\sqrt{1-\|\pi_{s-1}(\bv_i)\|^2}}\Big),
\end{align}
where the last line follows from the definition $\Psi_{k}(t)=\Pb_{\bz\sim \bS^{k-1}}[\bz_1\leq t/\sqrt{k}]$. Similarly, if $si\notin E(H)$, then
{\relscale{0.95}\begin{align}\label{eqn:prob_E_si_via_Psi_2}
    \Pb[E_{si}|\cG_{s-1}]\! =\! 1-\Psi_{d-s+1}\Big(\sqrt{\frac{d\!-\!s\!+\!1}{d}}\cdot \frac{-t_{si}-c_p+z_p}{\bv_s(s)\sqrt{1\!-\!\|\pi_{s-1}(\bv_i)\|^2}}\Big)=\Psi_{d-s+1}\Big(\sqrt{\frac{d\!-\!s\!+\!1}{d}}\cdot \frac{t_{si}+c_p-z_p}{\bv_s(s)\sqrt{1\!-\!\|\pi_{s-1}(\bv_i)\|^2}}\Big).
\end{align}}

Recall, we are trying to show
\[\prod_{s<i\leq n}\Pb\big[E_{si}| \cG_{s-1}\big]\leq e^{\widetilde O(n^{3/2}p^{1/2}/d)}\prod_{s<i\leq n}\Phi_{s-1, s, i},\]
where $\Phi_{s-1, s, i}=\Phi\big(-z_p+\sqrt{d}\langle \pi_{s-1}(\bv_i), \pi_{s-1}(\bv_s)\rangle\big) =\Phi(-t_{si})$ if $si\in E(H)$ and $\Phi_{s-1, s, i}=\Phi(t_{si})$ otherwise. Thus, to complete the proof, we need to convert $\Psi_{d-s+1}$ into $\Phi$ using Lemma~\ref{lemma:sphere_to_gaussian_CDF} and get rid of the various factors standing around $t_{si}$ in the expressions (\ref{eqn:prob_E_si_via_Psi_1}), (\ref{eqn:prob_E_si_via_Psi_2}). We begin by the latter, using the following simple claim.

\begin{claim}\label{claim:bookkeeping}
If the event $\Pi_{s-1}$ holds for the revealed coordinates, we have
{\relscale{0.95}\[\sqrt{\frac{d\!-\!s\!+\!1}{d}}=\Big(1-\frac{s}{2d}\Big)\Big(1+O\big(\frac{\sqrt{n}}{d}\big)\Big), \frac{1}{\bv_s(s)}=\Big(1+\frac{s}{2d}\Big)\Big(1+\widetilde O\big(\frac{\sqrt{n}}{d}\big)\Big), \frac{1}{\sqrt{1\!-\!\|\pi_{s-\!1}(\bv_i)\|_2^2}}=\Big(1+\frac{s}{2d}\Big)\Big(1+\widetilde O\big(\frac{\sqrt{n}}{d}\big)\Big).\]}
\end{claim}
\begin{proof}
The first expression is a direct consequence of the Taylor expansion of $\sqrt{1-t}=1-t/2+O(t^2)$ and the standing assumption on $d$ and $p$. As for the latter two,
by property~\ref{item:norm_concentration} of the event $\Pi_{s-1}$, we have that
\[ \Big|\sum_{j=1}^{s-1} \bv_i(j)^2-\frac{s-1}{d}\Big|=\widetilde O\big(\frac{\sqrt{n}}{d}\big)\qquad \text{ and  }\qquad \Big|\sum_{j=1}^{s-1} \bv_s(j)^2-\frac{s-1}{d}\Big|=\widetilde O\big(\frac{\sqrt{n}}{d}\big). \]
From here, it follows that
{\relscale{0.95}\begin{align*}
&\frac{1}{\sqrt{1-\!\|\pi_{s-1}(\bv_i)\|^2}}=\sqrt{\frac{d}{d\!-\!s\!+\!1}}\Big(1+\widetilde O\Big(\frac{\sqrt{n}}{d}\Big)\Big)=\Big(1+\frac{s-1}{2(d\!-\!s\!+\!1)}\Big)\Big(1+\widetilde O\Big(\frac{\sqrt{n}}{d}\Big)\Big)=\Big(1+\frac{s}{2d}\Big)\Big(1+\widetilde O\big(\frac{\sqrt{n}}{d}\big)\Big), \\
&\text{ and }\frac{1}{\bv_s(s)}=\frac{1}{\sqrt{1-\sum_{j=1}^{s-1} \bv_s(j)^2}}=\sqrt{\frac{d}{d\!-\!s\!+\!1}}\Big(1+\widetilde O\Big(\frac{\sqrt{n}}{d}\Big)\Big)=\Big(1+\frac{s}{2d}\Big)\Big(1+\widetilde O\big(\frac{\sqrt{n}}{d}\big)\Big).\qedhere
\end{align*}}
\end{proof}

Hence, if $si\in E(H)$, by using Claim~\ref{claim:bookkeeping} and Lemma~\ref{lemma:quantile_asymptotics} (which states that $|c_p-z_p|\leq \widetilde O(1/d)$), we can simplify the above expression to
\[\Pb[E_{si}|\cG_{s-1}] = \Psi_{d-s+1}\Big( -t_{si}\Big(1+\frac{s}{2d}\Big)^2\Big(1-\frac{s}{2d}\Big)\Big(1+\widetilde O\big(\frac{\sqrt{n}}{d}\big)\Big)\Big)=\Psi_{d-s+1}\Big(-t_{si}\Big(1+\frac{s}{2d}+\widetilde O\big(\frac{\sqrt{n}}{d}\big)\Big)\Big).\]
We can now use Lemma~\ref{lemma:sphere_to_gaussian_CDF} to replace $\Psi_{d-s+1}$ by $\Phi$, and then we can apply the log-concavity of $\Phi$ (i.e. inequality (\ref{eqn:log_concavity})) to estimate the resulting term:
\begin{align*}
\Pb[E_{si}|\cG_{s-1}] &= \Phi\Big(-t_{si}\Big(1+\frac{s}{2d}+\widetilde O\big(\frac{\sqrt{n}}{d}\big)\Big)\Big)\Big(1+O\Big(\frac{1+t_{si}^4}{d}\Big)\Big)\\
&\leq \Phi(-t_{si})\exp\Big(-\frac{\phi(-t_{si})}{\Phi(-t_{si})}t_{si}\Big(\frac{s}{2d}+\widetilde O\big(\frac{\sqrt{n}}{d}\big)\Big)\Big)\Big(1+O\Big(\frac{1+t_{si}^4}{d}\Big)\Big).
\end{align*}
Similarly, if $si\notin E(H)$, we can compute
\begin{align*}
\Pb[E_{si}|\cG_{s-1}] \leq \Phi(t_{si})\exp\Big(\frac{\phi(t_{si})}{\Phi(t_{si})}t_{si}\Big(\frac{s}{2d}+\widetilde O\big(\frac{\sqrt{n}}{d}\big)\Big)\Big)\Big(1+O\Big(\frac{1+t_{si}^4}{d}\Big)\Big).
\end{align*}

Before taking the product over all $i$, let us make a couple of observations. Note that $|t_{si}-z_p|=\sqrt{d}|\langle \pi_{s-1}(\bv_i), \pi_{s-1}(\bv_s)\rangle |\leq \widetilde O(\sqrt{{n/d}})$, so $t_{si}=\widetilde O(1)$.
Further, since $\frac{\phi(t)}{\Phi(t)}$ is a $1$-Lipschitz function, we have \[\frac{\phi(-t_{si})}{\Phi(-t_{si})}=\frac{\phi(-z_p)}{\Phi(-z_p)}+\widetilde O(\sqrt{n/d})=\frac{a}{p}\Big(1+\widetilde O\big(\sqrt{\frac{n}{d}}\big)\Big)\text{ and }\frac{\phi(t_{si})}{\Phi(t_{si})}=\frac{\phi(z_p)}{\Phi(z_p)}+\widetilde O(\sqrt{n/d})=\frac{a}{1-p}\Big(1+\widetilde O\big(\sqrt{\frac{n}{dp}}\big)\Big).\]

Hence, for $si\in E(H)$
\begin{align*}
\Pb[E_{si}|\cG_{s-1}] &\leq \Phi(-t_{si})\exp\Big(-\frac{\phi(-t_{si})}{\Phi(-t_{si})}t_{si}\Big(\frac{s}{2d}+\widetilde O\big(\frac{\sqrt{n}}{d}\big)\Big)\Big)\Big(1+\widetilde O\Big(\frac{1}{d}\Big)\Big)\\
&\leq \Phi(-t_{si})\exp\Big(-\frac{a}{p}z_p\frac{s}{2d}+\widetilde O\Big(\frac{\sqrt{n}}{d}\Big)+\widetilde O\Big(\frac{1}{d}\Big)\Big)
\end{align*}
Similarly, for $si\notin E(H)$
\begin{align*}
\Pb[E_{si}|\cG_{s-1}] \leq \Phi(t_{si})\exp\Big(\frac{a}{1-p}z_p\frac{s}{2d}+\widetilde O\Big(\frac{p\sqrt{n}}{d}\Big)+\widetilde O\Big(\frac{1}{d}\Big)\Big).
\end{align*}
Finally, taking the product over all $s<i\leq n$ and observing that $\Phi(-t_{si})=\Phi_{s-1, s, i}$ when $si\in E(H)$ and $\Phi(t_{si})=\Phi_{s-1, s, i}$ when $si\notin E(H)$, we get
\[ \Pb\Big[\bigwedge_{s<i\leq n}E_{si}\,\Big|\cG_{s-1}\Big] \leq \prod_{s<i\leq n}\Phi_{s-1,s,i}\cdot \exp\Big(-\frac{s}{2d}z_p\sum_{s< i\leq n} \lambda_{si}  +\widetilde O\Big(\frac{pn^{3/2}}{d}\Big)\Big). \]
Note that the error term contains $p$ since there are at most $O(pn)$ indices $i$ for which $si\in E(H)$, due to property \ref{item:degree_concentration} of the event $\cE$.
Since $\Big|\sum_{s<i\leq n}\lambda_{si}\Big|\leq \widetilde O(\sqrt{np})$ due to the fact $H\in \cE$, we have $\frac{s}{2d}z_p\sum_{s<i\leq n}\lambda_{si}=\widetilde O(\sqrt{np}\cdot \frac{n}{d})$. Thus, we finally arrive at
\[ \Pb\Big[\bigwedge_{s<i\leq n}E_{si}\,\Big|\cG_{s-1}\Big] \leq \prod_{s<i\leq n}\Phi_{s-1,s,i}\cdot \exp\Big(\widetilde O\Big(\frac{n^{3/2}p^{1/2}}{d}\Big)\Big), \]
which is what we needed to show.
\end{proof}

We will now give a proof of Lemma~\ref{lemma:induction-statement}, modulo one estimate, contained in Lemma~\ref{lemma:ratio}, which we will postpone for the latter sections.

\begin{proof}[Proof of Lemma~\ref{lemma:induction-statement} assuming Lemma~\ref{lemma:ratio}.] We argue by backwards induction on $s$. For $s=n-1$, the product and all the signed subgraph counts are empty and $Z_{n-1}=0$. Moreover, $\cF$ is determined by $\cG_{n-1}$, so the left-hand side is either $0$ or $1$. Thus, the claim follows after enlarging the implicit constant in the error term if necessary.

Fix $1\leq s\leq n-1$, assume the statement is known for $s$, and prove it for $s-1$. If the coordinates revealed by $\cG_{s-1}$ do not satisfy $\Pi_{s-1}$, then the event $\cF$ is impossible and there is nothing to prove. We may therefore assume that $\Pi_{s-1}$ holds.

Observe that the event $G(n, d, p)|_{[s, n]}=H_{s-1}$ is the intersection of the event $G(n, d, p)|_{[s+1, n]}=H_{s}$ with $\bigwedge_{s<i\leq n}E_{si}$. Combinatorially, $G(n, d, p)$ and $H$ agree on vertices $s, \dots, n$ if they agree on vertices $s+1, \dots, n$ and also agree on the edges touching the vertex $s$. By the law of total probability, we then have
\begin{align}
\Pb\Big[G(n, d, p)|_{[s, n]}=H_{s-1}\wedge \cF\Big|\cG_{s-1}\Big]=\bE\Big[\Pb\Big[G(n, d, p)|_{[s+1, n]}=H_{s}\wedge \cF\Big|\cG_{s}\Big]&\Big|\bigwedge_{s<i\leq n}E_{si}, \cG_{s-1}\Big]\nonumber\\
&\cdot \Pb\Big[\bigwedge_{s<i\leq n}E_{si}\Big|\cG_{s-1}\Big].\label{eqn:law_of_total_probability}
\end{align}

The second term can be estimated directly from Lemma~\ref{lemma:vertex_s}, since the revealed coordinates from $\cG_{s-1}$ must satisfy the event $\Pi_{s-1}$. The first term of (\ref{eqn:law_of_total_probability}) can be estimated using the induction hypothesis. In order to slightly simplify the notation, we will introduce
\[f(H)=\frac{\kappa_{\triangle} N_{\triangle}(H)}{\sqrt{d}} + \frac{\kappa_{\square} N_{\square}(H)}{d} + \frac{\kappa_{\diamond}N_{\diamond}(H)}{d}.\]
By the induction hypothesis, over the randomness of coordinates $s+1, \dots, n$, we have
\[
\Pb\Big[G(n, d, p)|_{[s+1, n]}=H_{s}\wedge \cF\Big|\cG_{s}\Big] \leq \mathbf{1}_{\cT_s}\!\!\!\!\prod_{s<i<j\leq n}\!\!\!\!\Phi_{sij}\, \exp\Big(Z_s+f(H_s)+ \widetilde O\Big(R_{n, d, p}(n\!-\!s)\Big)\Big).
\]
Note that we have inserted the indicator $\mathbf{1}_{\cT_s}$ in the above estimate, since if the event $\cT_s$ is not satisfied for the coordinates from $\cG_s$, then $\cF$ cannot be satisfied either and the resulting probability is $0$.
Thus, Lemma~\ref{lemma:vertex_s} and the induction hypothesis give
\begin{align*}
\Pb\Big[G(n, d, p)|_{[s, n]}=H_{s-1}\wedge \cF\Big|\cG_{s-1}\Big] \leq \bE\Big[\mathbf{1}_{\cT_s}\!\!\!\!\prod_{s<i<j\leq n}\!\!\!\!\Phi_{sij}\, e^{Z_s} &\Big|\bigwedge_{s<i\leq n}E_{si}, \cG_{s-1}\Big]\cdot \prod_{s<i\leq n}\Phi_{s-1,s,i}\\
&\cdot \exp\Big(f(H_{s})+\widetilde O\Big(R_{n, d, p}(n\!-\!s\!+\!1)\Big)\Big).
\end{align*}

To complete the proof, it would suffice to show
\[
\bE\Big[\mathbf{1}_{\cT_s}\!\!\!\!\prod_{s<i<j\leq n}\!\!\!\!\Phi_{sij}e^{Z_s}\,\Big|\,\bigwedge_{s<i\leq n}E_{si}, \cG_{s-1}\Big]\leq \prod_{s<i<j\leq n}\!\!\!\!\Phi_{s-1, ij}\exp\Big(Z_{s-1}+f(H_{s-1})-f(H_s)+\widetilde O\Big(R_{n, d, p}\Big)\Big).
\]
This expectation is taken with respect to $\bv_{s+1}(s), \dots, \bv_n(s)$. Importantly, Lemma~\ref{lemma:vertex_s} (b) determines the distribution of these variables conditioned on $\bigwedge_{s<i\leq n}E_{si}$ and $\cG_{s-1}$. Therefore, we have reduced proving Lemma~\ref{lemma:induction-statement} to the following estimate. Although we could state the following estimate in terms of the signed subgraph counts, we find that it is much easier to reason about the sums of $\lambda_{ij}$, which are essentially equivalent to subgraph counts, due to the relation $\lambda_{ij}=\frac{a}{p(1-p)}(A_{ij}-p)$.

\begin{lemma}\label{lemma:ratio}
Let us consider the \setupstar{}. Under the standing hypotheses, we have
\begin{align*}
\bE\bigg[\mathbf{1}_{\cT_s} e^{Z_s-Z_{s-1}}\!\!\!\!\prod_{s<i<j\leq n}\frac{\Phi_{sij}}{\Phi_{s-1, ij}}\bigg]&\leq
\exp\Big( \frac{1}{\sqrt{d}}\sum_{s<i<j\leq n}\lambda_{si}\lambda_{sj}\lambda_{ij}-\frac{1}{2d}\sum_{\substack{s<i<j\leq n}}\!\!\!\!\!\!\lambda_{si}^2\lambda_{sj}^2\lambda_{ij}^2
+\frac{1}{d}\sum_{\substack{s<r\leq n\\s<i<j\leq n}}\lambda_{si}\lambda_{sj}\lambda_{ri}\lambda_{rj}\\
&+\frac{1}{d}\sum_{\substack{s<r\leq n\\s<i<j\leq n}}\lambda_{si}\lambda_{sj}\lambda_{ri}\lambda_{rj}\big[\lambda_{ij}(z_p-\lambda_{ij})+\lambda_{sr}(z_p-\lambda_{sr})\big]+\widetilde O\Big(R_{n, d, p}\Big)\Big).
\end{align*}
\end{lemma}

The proof of this lemma will be spread over the following sections. Here, we focus on explaining the statement itself and showing how it suffices to complete the proof of Lemma~\ref{lemma:induction-statement}.

The first two terms are sums over triangles, where a triangle $sij$ receives weight $\lambda_{si}\lambda_{sj}\lambda_{ij}/\sqrt{d}-\lambda_{si}^2\lambda_{sj}^2\lambda_{ij}^2/(2d)$. Observe that the second summand is always nonpositive, and may therefore be omitted. The latter two terms sum over four-cycles containing $s$. For every such four-cycle, the two terms in the brackets correspond to adding one of its two diagonals, and hence to the two types of diamonds containing $s$.

We now verify that the remaining terms give precisely $f(H_{s-1})-f(H_s)$. Define the discrete derivative $\Delta_sN_F=N_F(H_{s-1})-N_F(H_s)$, for $F\in\{ \triangle, \square, \diamond \}$. Since $\lambda_{ij}=\frac{a}{p(1-p)}(A_{ij}-p)$, we have
\[\sum_{s<i<j\leq n}\lambda_{si}\lambda_{sj}\lambda_{ij}=\frac{a^3}{p^3(1-p)^3}\Delta_sN_{\triangle}, \quad\text{and }\sum_{\substack{s<r\leq n\\s<i<j\leq n}}\lambda_{si}\lambda_{sj}\lambda_{ri}\lambda_{rj}=\frac{a^4}{p^4(1-p)^4}\Delta_sN_{\square}.\]
Note that in the latter sum $i$ and $j$ are the two neighbors of $s$ on the four-cycle and $r$ is the vertex opposite $s$, so every unlabeled four-cycle including $s$ occurs exactly once.

For $s<i<j\leq n$, the identity $(A_{ij}-p)^2=(1-2p)(A_{ij}-p)+p(1-p)$ gives
\[\lambda_{ij}(z_p-\lambda_{ij})=\Big(z_p-\frac{a(1-2p)}{p(1-p)}\Big)\lambda_{ij}-\frac{a^2}{p(1-p)}.\]
Consequently,
\begin{align*}
&\sum_{\substack{s<r\leq n\\s<i<j\leq n}}\lambda_{si}\lambda_{sj}\lambda_{ri}\lambda_{rj}\big[\lambda_{ij}(z_p-\lambda_{ij})+\lambda_{sr}(z_p-\lambda_{sr})\big]\\
&\qquad=\frac{a^5}{p^5(1-p)^5}\Big(z_p-\frac{a(1-2p)}{p(1-p)}\Big)\Delta_sN_{\diamond}-\frac{2a^2}{p(1-p)}\frac{a^4}{p^4(1-p)^4}\Delta_sN_{\square}.
\end{align*}
Again, every diamond containing $s$ occurs exactly once: its unique four-cycle determines $i,j,r$, while its fifth edge is either $ij$ or $sr$. Combining the last three displays and the definitions of $\kappa_{\triangle},\kappa_{\square},\kappa_{\diamond}$, Lemma~\ref{lemma:ratio} yields
\begin{align*}
\bE\Big[\mathbf{1}_{\cT_s}e^{Z_s-Z_{s-1}}\!\!\!\!\prod_{s<i<j\leq n}\frac{\Phi_{sij}}{\Phi_{s-1,ij}}\Big]&\leq \exp\Big(\frac{\kappa_{\triangle}\Delta_sN_{\triangle}}{\sqrt d}+\frac{\kappa_{\square}\Delta_sN_{\square}}{d}+\frac{\kappa_{\diamond}\Delta_sN_{\diamond}}{d}+\widetilde O(R_{n,d,p})\Big)\\
&=\exp\Big(f(H_{s-1})-f(H_s)+\widetilde O(R_{n,d,p})\Big).
\end{align*}
This is the estimate required above. Substituting it into (\ref{eqn:law_of_total_probability}) and absorbing the two one-step error terms gives the induction statement at $s-1$, with error $\widetilde O((n-s+1)R_{n,d,p})$. This completes the induction.
\end{proof}

\section{Perturbative analysis of \texorpdfstring{\setupstar{}}{Setup (star)}}\label{sec:perturbation}
Let us note that \setupstar{} will play a crucial role throughout the paper, due to the fact that the only remaining thing to prove is Lemma~\ref{lemma:ratio}, which is stated under the \setupstar{}. Therefore, we will now present a couple of statements about it whose purpose is both to give an intuition about it, as well as to prepare various steps in the proof of Lemma~\ref{lemma:ratio}.

As we have already alluded to in Lemma~\ref{lemma:sphere_to_gaussian_CDF}, the first coordinate of a uniform random point on the unit sphere $\bS^{k-1}$ has a very similar distribution as a gaussian $\cN(0, 1/k)$. So, for $i\in \{s+1, \dots, n\}$, let us define $g_i$ to be a $\cN(0, 1/d)$ truncated above $z_p/\sqrt{d}$ if $si\in E(H)$ and below $z_p/\sqrt{d}$ if $si\notin E(H)$. Recall, $z_p=\Phi^{-1}(1-p)$ is the unique real number for which $\Pb[\cN(0, 1)\geq z_p]=p$.

Then, we will think of $\bv_i(s)$ as a slightly perturbed version of $g_i$, and the following statement justifies this intuition.

\begin{proposition}\label{prop:comparing_moments}
Let $\bu_i(s)$ be random variables sampled according to the \setupstar{}. Let $g_i$ be a $\cN(0, 1/d)$ truncated above $z_p/\sqrt{d}$ if $si\in E(H)$ and below $z_p/\sqrt{d}$ if $si\notin E(H)$, where $z_p=\Phi^{-1}(1-p)$.
\begin{enumerate}[label=(\roman*)]
    \item \label{item:comparing_first_order}
    We then have
    \begin{equation}\label{eqn:expectation_v_i}
        \bE\big[\bu_i(s)\big]=\bE[g_i]+\lambda_{si}(z_p-\lambda_{si})\langle \pi_{s-1} (\bv_i), \pi_{s-1} (\bv_s)\rangle + \widetilde O\Big(\frac{n}{d^{3/2}}\Big).
    \end{equation} In particular, since $\bE[g_i]=\frac{\lambda_{si}}{\sqrt{d}}$, we have $\bE[\bu_i(s)]=\frac{\lambda_{si}}{\sqrt{d}}+ \widetilde O\Big(\frac{|\lambda_{si}|\sqrt{n}}{d}\Big)$.
    \item \label{item:comparing_general_order}
    We also have \begin{equation}\label{eqn:comparing_moments}
    \big|\bE[g_i^2]-\bE[\bu_i(s)^2]\big|\leq \widetilde O\Big(\frac{|\lambda_{si}|\sqrt{n}}{d^{3/2}}\Big), \text{ and so} \Var(\bu_i(s))=\Big(1+\widetilde O\Big(\frac{\sqrt{n}}{\sqrt{d}}\Big)\Big)\frac{1+\lambda_{si}(z_p-\lambda_{si})}{d}.
    \end{equation}
\end{enumerate}
\end{proposition}
\begin{proof}
Let us fix $i\in \{s+1, \dots, n\}$ with $si\in E(H)$ and prove the statement for it. The case of $si\notin E(H)$ is analogous, and modulo the appropriate sign changes.

First of all, note that it makes no difference whether we prove the above lemma for $\bu_i(s)$ or $\bv_i(s)$, since the difference between their first and second moments is at most $\Pb\big[|\bv_i(s)|\geq 10\sqrt{\frac{\log d}{d}}\big]$, since $|\bv_i(s)|\leq 1$ always. But this probability is at most $d^{-2}$ (say), and so the effect on the moments is negligible. Hence, we will work with $\bv_i(s)$ in this proof.

Recall our definition $r_i(s)=\sqrt{1-\|\pi_{s-1}(\bv_{i})\|^2}$ and \[\tau_{si}=\frac{c_p}{\sqrt{d}\bv_s(s)r_i(s)}-\frac{\langle \pi_{s-1}(\bv_i), \pi_{s-1}(\bv_s)\rangle}{\bv_s(s)r_i(s)}.\] Due to \setupstar{}, we have $\bv_i(s)/r_i(s)$ has the distribution of a first coordinate of a uniform random point on $\bS^{m-1}$, truncated above $\tau_{si}$, where $m=d-s+1$, and $\bu_i(s)=\bv_i(s)\mathbf{1}_{|\bv_i(s)|\leq 10\sqrt{\log d}/\sqrt{d}}$.

The main idea of the proof is the following. Let $\tilde{g}_i$ be the normal variable $\cN(0, 1/m)$ truncated on being larger than $\tau_{si}$. We will first show that the moments of $\bv_i(s)$ are basically the same as the moments of $\tilde{g}_i$, and then we will show how the moments of $\tilde{g}_i$ and $g_i$ relate to each other.

However, before we do that, let us briefly observe that we have \begin{equation}\label{eqn:tau_i_asymptotics}
\tau_{si}=\frac{z_p-\sqrt{d}\langle \pi_{s-1}(\bv_i), \pi_{s-1}(\bv_s)\rangle}{\sqrt{m}}+\widetilde O\Big(\frac{n}{d^{3/2}}\Big).
\end{equation}
Indeed, since Claim~\ref{claim:bookkeeping} implies that $\bv_s(s)=1+\widetilde O(n/d)$ and $r_i(s)=1+\widetilde O(n/d)$, it is easy to conclude $\tau_{si}=(1+\widetilde O(n/d)) \frac{c_p-\sqrt{d}\langle \pi_{s-1}(\bv_i), \pi_{s-1}(\bv_s)\rangle}{\sqrt{d}}=(1+\widetilde O(n/d)) \frac{c_p-\sqrt{d}\langle \pi_{s-1}(\bv_i), \pi_{s-1}(\bv_s)\rangle}{\sqrt{m}}$, where we have used $\sqrt{m/d}=1+O(n/d)$, again from Claim~\ref{claim:bookkeeping}. Finally, Lemma~\ref{lemma:quantile_asymptotics} gives that $|c_p-z_p|\leq \widetilde O(1/d)$, and so we arrive at (\ref{eqn:tau_i_asymptotics}). Observe that this also implies $\tau_{si}=O(z_p/\sqrt{d})$, which we will also find useful throughout the proof.

\medskip
\noindent
To perform the first step of our plan, we appeal to Lemma~\ref{lemma:sphere_coupling_moments}.
Since $\bv_i(s)/r_i(s)$ is truncated to be above $\tau_{si}$ and $t=\tau_{si}\sqrt{m}\leq O(z_p)<m^{1/8}$, we may apply Lemma~\ref{lemma:sphere_coupling_moments}, which gives, for $k=1,2$, that
\begin{equation}\label{eqn:moment_comparison_1}
\big|\bE[\tilde{g}_i^k]-\bE[\bv_i(s)^k]/r_i(s)^k\big|\leq O\Big(\frac{z_p^{k+4}}{d^{k/2+1}}\Big).
\end{equation}
From Claim~\ref{claim:bookkeeping}, we have that $r_i(s)=1+\widetilde O(n/d)$, which gives $r_i(s)^k=1+\widetilde O(n/d)$. Since $\bE[\tilde{g}_i^k]\leq O(z_p^k/d^{k/2})$ due to Lemma~\ref{lemma:expectation_truncated_gaussian}, we would get
\begin{equation}\label{eqn:moment_comparison_2}
\Big|\bE\big[\tilde{g}_i^k\big]-\bE\big[\bv_i(s)^k\big]\Big|\leq \big|r_i(s)^k\bE[\tilde{g}_i^k]-\bE[\bv_i(s)^k]\big|+|r_i(s)^k-1|\cdot |\bE[\tilde{g}_i^k]|\leq
O\Big(\frac{z_p^{k+4}}{d^{k/2+1}}\Big)+\widetilde O\Big(\frac{n}{d} \cdot \frac{z_p^k}{d^{k/2}}\Big) =\widetilde O\Big(\frac{z_p^kn}{d^{k/2+1}}\Big).
\end{equation}

Observe that we have now reduced Proposition~\ref{prop:comparing_moments} to comparing the moments of $\tilde{g}_i$ and $g_i$, both of which are truncated Gaussians, albeit with different variances and thresholds. Luckily, the first two moments of truncated Gaussians can be precisely computed from Lemma~\ref{lemma:expectation_truncated_gaussian}.

We begin by discussing the second moments. It will be convenient to write $t_1=\sqrt m\tau_{si}$, $t_2=z_p$ and $h(t)=1+t\lambda(-t)$. From Lemma~\ref{lemma:expectation_truncated_gaussian},
\begin{align*}
\bE[g_i^2]&=\bE[g_i]^2+\Var(g_i)=\frac{1}{d}\lambda(-z_p)^2+\frac{1}{d}(1+z_p\lambda(-z_p)-\lambda(-z_p)^2)=\frac{1+z_p\lambda(-z_p)}{d}, \text{ and}\\
\bE[\tilde{g}_i^2]&=\bE[\tilde{g}_i]^2+\Var(\tilde{g}_i)=\frac{1}{m}\lambda(-t_1)^2+\frac{1}{m}(1+t_1\lambda(-t_1)-\lambda(-t_1)^2)=\frac{1+t_1\lambda(-t_1)}{m}, \text{ and so }\\
\bE[g_i^2]-\bE[\tilde{g}_i^2] &= \frac{1+z_p\lambda(-z_p)}{d}-\frac{1+t_1\lambda(-t_1)}{m}=\frac{h(t_2)}{d}-\frac{h(t_1)}{m}.
\end{align*}
Since $t_1, t_2=O(z_p)$, for all $t$ between $t_1$ and $t_2$, we have $\lambda(-t)\leq |t|+1=O(z_p)$, and using Lemma~\ref{lemma:mills_ratio} we can compute
\[
h'(t)=\lambda(-t)+t(\lambda(-t))'=\lambda(-t)+t\lambda(-t)(\lambda(-t)-t),\text{ and so }|h'(t)|\leq  O(|\lambda(-z_p)|z_p^2)=O(|\lambda_{si}|z_p^2)
\]
for all $t$ between $t_1$ and $t_2$. Since $t_1=\sqrt{m}\tau_{si}$ and the event $\Pi_{s-1}$ holds,
\[
|t_1-t_2|\leq \sqrt d|\langle \pi_{s-1}(\bv_i),\pi_{s-1}(\bv_s)\rangle|+\widetilde O\Big(\frac{n}{d}\Big)\leq \widetilde O\Big(\sqrt{\frac{n}{d}}\Big).
\]
Therefore,
\begin{align*}
\big|\bE[g_i^2]-\bE[\tilde g_i^2]\big|
=\left|\frac{h(t_2)}{d}-\frac{h(t_1)}{m}\right|
&\leq \left|\frac{1}{d}-\frac{1}{m}\right||h(t_2)|+\frac{|h(t_2)-h(t_1)|}{m}\\
&\leq O\Big(\frac{n}{d^2}\cdot z_p^2\Big)+O\Big(\frac{|\lambda_{si}|z_p^2}{d}\sqrt{\frac{n}{d}}\Big)\leq \widetilde O\Big(\frac{|\lambda_{si}|\sqrt{n}}{d^{3/2}}\Big).
\end{align*}
In passing to the last line, we have used that $|h(t_1)-h(t_2)|\leq |t_1-t_2|\cdot\max_{t\in (t_1, t_2)}|h'(t)|$, and also that $m=d-s+1\geq d/2$, so that $nz_p^2/d^2$ is negligible.

Combining this with \eqref{eqn:moment_comparison_2} for $k=2$, we obtain
\begin{align*}
\big|\bE[g_i^2]-\bE[\bv_i(s)^2]\big| &\leq \big|\bE[g_i^2]-\bE[\tilde g_i^2]\big|+\big|\bE[\tilde g_i^2]-\bE[\bv_i(s)^2]\big| \leq \widetilde O\Big(\frac{|\lambda_{si}|\sqrt{n}}{d^{3/2}}\Big).
\end{align*}

This proves the first part of \textit{\ref{item:comparing_general_order}}. Now, we discuss \textit{\ref{item:comparing_first_order}}.

Since we have shown that $\big|\bE[\bv_i(s)]-\bE[\tilde{g}_i]\big|\leq \widetilde O(\frac{n}{d^{3/2}})$, it would be sufficient to prove
\[\bE[\tilde{g}_i]=\bE[g_i]+\lambda_{si}(z_p-\lambda_{si})\langle \pi_{s-1} (\bv_i), \pi_{s-1} (\bv_s)\rangle + \widetilde  O\Big(\frac{n}{d^{3/2}}\Big).\]
Since both $\tilde{g}_i$ and $g_i$ are truncated Gaussians, we can compute their expectations very precisely using Lemma~\ref{lemma:expectation_truncated_gaussian}, as follows: $\bE[\tilde{g}_i]=\lambda(-t_1)/\sqrt{m}$ and $\bE[g_i]=\lambda(-t_2)/\sqrt{d}$. Moreover, using the second-order Taylor expansion of $\lambda$ we can write
\[\lambda(-t_1)=\lambda(-t_2)+(t_2-t_1)\lambda'(-t_2)+ O(|t_2-t_1|^2)=\lambda(-t_2)+(t_2-t_1)\lambda(-t_2)(t_2-\lambda(-t_2))+ O(|t_2-t_1|^2),\]
where we have used that $\lambda'(x)=-\lambda(x)(x+\lambda(x))$ and also that $|\lambda''(x)|\leq 1$ for all $x\in \bR$ (see Lemma~\ref{lemma:mills_ratio}). Since $t_2=z_p$ and $t_1=z_p-\sqrt{d}\langle \pi_{s-1}(\bv_i), \pi_{s-1}(\bv_s)\rangle+\widetilde O(n/d)$, we have $|t_1-t_2|^2\leq d|\langle \pi_{s-1}(\bv_i), \pi_{s-1}(\bv_s)\rangle|^2+\widetilde O(n^2/d^2)\leq \widetilde O(n/d)$, due to property \ref{item:inner_product} of the event $\Pi_{s-1}$ (which is satisfied by the vectors $\bv_1, \dots, \bv_n$ due to the \setupstar{}). Recalling that $\lambda(-z_p)=\frac{a}{p}=\lambda_{si}$, we obtain
\[\lambda(-t_1)=\lambda_{si}+\lambda_{si}(z_p-\lambda_{si}) \sqrt{d}\langle \pi_{s-1}(\bv_i), \pi_{s-1}(\bv_s)\rangle+ \widetilde  O\Big(\frac{n}{d}\Big).\]
Finally, since $1/\sqrt{m}=(1+O(n/d))/\sqrt{d}$, we conclude
\begin{align*}
    \bE[\tilde{g}_i]=\frac{\lambda(-t_1)}{\sqrt{m}}&=\Big(1+O(n/d)\Big)\frac{\lambda_{si}+\lambda_{si}(z_p-\lambda_{si}) \sqrt{d}\langle \pi_{s-1}(\bv_i), \pi_{s-1}(\bv_s)\rangle+ \widetilde O\Big(\frac{n}{d}\Big)}{\sqrt{d}}\\
    &=\bE[g_i]+\lambda_{si}(z_p-\lambda_{si}) \langle \pi_{s-1}(\bv_i), \pi_{s-1}(\bv_s)\rangle+\widetilde O\Big(\frac{n}{d^{3/2}}\Big).
\end{align*}
The contribution of the $O(n/d)$ factor in the last equation was absorbed into the $\widetilde O(n/d^{3/2})$ term using
that $|\lambda_{si}+\lambda_{si}(z_p-\lambda_{si}) \sqrt{d}\langle \pi_{s-1}(\bv_i), \pi_{s-1}(\bv_s)\rangle|/\sqrt{d}=\widetilde O(1/\sqrt{d})$.
This completes the proof of \textit{\ref{item:comparing_first_order}}. To derive the second statement of \textit{\ref{item:comparing_general_order}}, simply note that 
\[\big|\Var(\bu_i(s))-\Var(g_i)\big|\leq \big|\bE [\bu_i(s)^2]-\bE [g_i^2]\big|+\big|\bE [\bu_i(s)]^2-\bE [g_i]^2\big|\leq \widetilde O\Big(\frac{\sqrt{n}}{d^{3/2}}\Big)+\widetilde O\Big(\frac{\sqrt{n}}{d^{3/2}}\Big).\]
Since Lemma~\ref{lemma:expectation_truncated_gaussian} tells us that $\Var(g_i)=\frac{1+\lambda_{si}(z_p-\lambda_{si})}{d}\geq \widetilde\Omega(1/d)$, we conclude that $\Var(\bu_i(s))=(1+\widetilde O(\sqrt{n/d}))(1+\lambda_{si}(z_p-\lambda_{si}))/d$.
\end{proof}

The probability that $G(n, d, p)$ agrees with $H_{s-1}$ on the vertices $s, \dots, n$ is, in the leading order, determined by the factors $\Phi_{sij}$, defined as $\Phi_{sij}=\Phi(-z_p+\sqrt{d}\langle \pi_{s-1}(\bv_i),\pi_{s-1}(\bv_j)\rangle)$ if $ij\in E(H)$ and $\Phi_{sij}=\Phi(z_p-\sqrt{d}\langle \pi_{s-1}(\bv_i),\pi_{s-1}(\bv_j)\rangle)$ otherwise (see Lemma~\ref{lemma:induction-statement}. Hence, we will introduce the notation
\[\alpha_{sij}=
\begin{cases}
-z_p+\sqrt d\,\langle \pi_{s-1}(\bv_i),\pi_{s-1}(\bv_j)\rangle,& ij\in E(H),\\[2mm]
z_p-\sqrt d\,\langle \pi_{s-1}(\bv_i),\pi_{s-1}(\bv_j)\rangle,& ij\notin E(H),
\end{cases} \]
so that we have $\Phi_{s-1, ij}=\Phi(\alpha_{sij})$. Also, when we preform the Taylor expansions to bound the one-step contribution in the induction, we will encounter the quantities $\lambda_{sij}$ defined as $\lambda_{sij}=\lambda(\alpha_{sij})=\phi(\alpha_{sij})/\Phi(\alpha_{sij})$ if $ij\in E(H)$ and $\lambda_{sij}=-\lambda(\alpha_{sij})=-\phi(\alpha_{sij})/\Phi(\alpha_{sij})$ otherwise. In order to prepare the calculations we will perform later on, let us prove a short estimate on their size, which shows that $\lambda_{sij}\approx \lambda_{ij}$, to first order at least.

\begin{claim}\label{claim:error_in_lambda}
For all $s<i<j\leq n$, if the event $\Pi_{s-1}$ holds for the revealed coordinates, then
\[ \lambda_{sij} = \lambda_{ij} + \lambda_{ij}(z_p-\lambda_{ij})
\sqrt d\,\langle \pi_{s-1}(\bv_i), \pi_{s-1}(\bv_j)\rangle
+ \widetilde O\Big(\frac{n}{d}\Big). \]
As a consequence, we have $\lambda_{sij}=\lambda_{ij}+\widetilde O\big(|\lambda_{ij}|\sqrt{n/d}\big)$.
\end{claim}
\begin{proof}
In this proof, we will consider two cases - if $ij\in E(H)$ and if $ij\notin E(H)$. Assume first that $ij\in E(H)$. Recall that we have defined $\lambda_{ij}=\lambda(-z_p)$ and $\lambda_{sij}=\lambda(\alpha_{sij})$, where $\alpha_{sij}=-z_p+\sqrt{d}\langle \pi_{s-1}(\bv_i), \pi_{s-1}(\bv_j)\rangle$.

Using the Taylor series of $\lambda$, due to Lemma~\ref{lemma:mills_ratio} we can write
\[\lambda_{sij}=\lambda(\alpha_{sij})=\lambda(-z_p)+(\alpha_{sij}+z_p)\lambda'(-z_p)+O(|\alpha_{sij}+z_p|^2). \]
Since $\lambda'(-z_p)=z_p\lambda(-z_p)-\lambda(-z_p)^2=\lambda_{ij}(z_p-\lambda_{ij})$,
we get
\[ \lambda_{sij}=\lambda_{ij}+\lambda_{ij}(z_p-\lambda_{ij})\sqrt d\,\langle \pi_{s-1}(\bv_i), \pi_{s-1}(\bv_j)\rangle+\widetilde O\Big(\frac{n}{d}\Big).\]
In the above equation, we have used that $\alpha_{sij}+z_p=\sqrt{d}\langle \pi_{s-1}(\bv_i), \pi_{s-1}(\bv_j)\rangle$, which implies $|\alpha_{sij}+z_p|^2\leq d\langle \pi_{s-1}(\bv_i), \pi_{s-1}(\bv_j)\rangle^2\leq \widetilde O(n/d)$, due to the property \ref{item:inner_product} of the event $\Pi_{s-1}$.

If $ij\notin E(H)$, the proof is very similar, except a couple of sign changes. Using the Taylor series expansion, we have
\[\lambda_{sij}=-\lambda(\alpha_{sij})=-\lambda(z_p)-(\alpha_{sij}-z_p)\lambda'(z_p)+O(|\alpha_{sij}-z_p|^2),\]
where $\alpha_{sij}=z_p-\sqrt{d}\langle \pi_{s-1}(\bv_i), \pi_{s-1}(\bv_j)\rangle$ and $\lambda'(z_p)=-z_p\lambda(z_p)-\lambda(z_p)^2=\lambda_{ij}(z_p-\lambda_{ij})$. So, we get
\[\lambda_{sij}=\lambda_{ij}+\lambda_{ij}(z_p-\lambda_{ij})\sqrt d\,\langle \pi_{s-1}(\bv_i), \pi_{s-1}(\bv_j)\rangle+\widetilde O\Big(\frac{n}{d}\Big).\qedhere\]
\end{proof}

\section{Bounding the one-step contribution}\label{sec:ratio}

In this section, we will prove Lemma~\ref{lemma:ratio}, modulo an auxiliary lemma, which we will state and prove later. We work under the \setupstar{}, and we fix a graph $H\in \cE$ and a probability $1/3\geq p\geq n^{-1/3}\polylog n$ and a dimension $d\geq Cn^2p(\log n)^A$. Then, Lemma~\ref{lemma:ratio} states that
\begin{align*}
\bE\bigg[\mathbf{1}_{\cT_s} e^{Z_s-Z_{s-1}}\!\!\!\!\prod_{s<i<j\leq n}\frac{\Phi_{sij}}{\Phi_{s-1, ij}}\bigg]&\leq
\exp\Big( \frac{1}{\sqrt{d}}\sum_{s<i<j\leq n}\lambda_{si}\lambda_{sj}\lambda_{ij}-\frac{1}{2d}\sum_{\substack{s<i<j\leq n}}\!\!\!\!\!\!\lambda_{si}^2\lambda_{sj}^2\lambda_{ij}^2
+\frac{1}{d}\sum_{\substack{s<r\leq n\\s<i<j\leq n}}\lambda_{si}\lambda_{sj}\lambda_{ri}\lambda_{rj}\\
&+\frac{1}{d}\sum_{\substack{s<r\leq n\\s<i<j\leq n}}\lambda_{si}\lambda_{sj}\lambda_{ri}\lambda_{rj}\big[\lambda_{ij}(z_p-\lambda_{ij})+\lambda_{sr}(z_p-\lambda_{sr})\big]+\widetilde O\Big(R_{n, d, p}\Big)\Big).
\end{align*}
where
\[ \Phi_{sij}=
\begin{cases}
\Phi\Big(-z_p+\sqrt d\,\langle \pi_s(\bv_i),\pi_s(\bv_j)\rangle\Big),& ij\in E(H),\\[2mm]
\Phi\Big(z_p-\sqrt d\,\langle \pi_s(\bv_i),\pi_s(\bv_j)\rangle\Big),& ij\notin E(H),
\end{cases} \]
and $Z_s$ was defined in the statement of the Lemma~\ref{lemma:induction-statement}. Also, the expectation is taken over the randomness of the coordinates $\{\bv_i(s)|s<i\leq n\}$.

\begin{proof}[Proof of Lemma~\ref{lemma:ratio}.]
We begin the proof by defining $\beta_{sij}=\sqrt{d}\bu_i(s)\bu_j(s)$ and \[\alpha_{sij}=
\begin{cases}
-z_p+\sqrt d\,\langle \pi_{s-1}(\bv_i),\pi_{s-1}(\bv_j)\rangle,& ij\in E(H),\\
z_p-\sqrt d\,\langle \pi_{s-1}(\bv_i),\pi_{s-1}(\bv_j)\rangle,& ij\notin E(H),
\end{cases} \]
so that on the event $\cT_s$ we have $\Phi_{s-1, ij}=\Phi(\alpha_{sij})$ and $\Phi_{sij}=\Phi(\alpha_{sij}+\beta_{sij})$ if $ij\in E(H)$ (and $\Phi_{sij}=\Phi(\alpha_{sij}-\beta_{sij})$ otherwise). We will also set $\lambda_{sij}=\lambda(\alpha_{sij})=\phi(\alpha_{sij})/\Phi(\alpha_{sij})$ if $ij\in E(H)$ and $\lambda_{sij}=-\lambda(\alpha_{sij})=-\phi(\alpha_{sij})/\Phi(\alpha_{sij})$ otherwise. Our first step will be to bound the product of ratios $\frac{\Phi_{sij}}{\Phi_{s-1, ij}}$ by using the second-order Taylor expansion of $\Phi$ given in Lemma~\ref{lemma:tail-expansion}. Once we do this, we will state the needed auxiliary lemmas and complete the proof using them.

\begin{claim}\label{claim:reduction_to_Sigma_1}
If the revealed coordinates $\{\bv_i(j)|1\leq j< s\}$ satisfy the event $\Pi_{s-1}$ and the random coordinates $\{\bv_i(s)|s<i\leq n\}$ satisfy the event $\cT_s$, then we deterministically have
\[\prod_{s<i<j\leq n}\frac{\Phi_{sij}}{\Phi_{s-1, ij}}\leq \exp\Big(\Sigma_1+\widetilde O\Big(\frac{n^{3/2}\sqrt{p}}{d}+\frac{n^{5/2}}{d^{3/2}}\Big)\Big),\]
where \[\Sigma_1=\sum_{s<i<j\leq n}\lambda_{sij}\beta_{sij}-\sum_{s<i<j\leq n}\lambda_{sij}^2\frac{\beta_{sij}^2}{2}.\]
\end{claim}
\begin{proof}
We already noted that $\Phi_{s-1,ij}=\Phi(\alpha_{sij})$ and $\Phi_{sij}=\Phi(\alpha_{sij}+\beta_{sij})$ if $ij\in E(H)$ ($\Phi_{sij}=\Phi(\alpha_{sij}-\beta_{sij})$ if $ij\notin E(H)$). Thus, we can use equation (\ref{eqn:second_order_taylor}) from Lemma~\ref{lemma:tail-expansion} to obtain
\[\frac{\Phi_{sij}}{\Phi_{s-1, ij}}\leq \exp\Big(\lambda_{sij}\beta_{sij}-\lambda_{sij}^2\frac{\beta_{sij}^2}{2}\qquad+|\alpha_{sij}|\lambda_{sij}\frac{\beta_{sij}^2}{2}\qquad+|\beta_{sij}|^3\Big).\]
Due to the event $\Pi_s$, $\alpha_{sij}$ is positive when $ij\in E(H)$ and negative otherwise. So, the signs in the above equation correctly capture the two cases when $ij\in E(H)$ and $ij\notin E(H)$. In order to control the product of the above terms, we define
\[\Sigma_2=\sum_{s<i<j\leq n}|\alpha_{sij}|\lambda_{sij}\frac{\beta_{sij}^2}{2},\quad \Sigma_3=\sum_{s<i<j\leq n}|\beta_{sij}|^3,\quad \text{ so that}\quad
\prod_{s<i<j\leq n} \frac{\Phi_{sij}}{\Phi_{s-1, ij}}\leq \exp(\Sigma_1+\Sigma_2+\Sigma_3).
\]
From here, it is clear that we need to bound $|\Sigma_2|$ and $|\Sigma_3|$ in order to complete the proof. Bounding $\Sigma_3$ turns out to be quite simple: since $|\bu_i(s)|, |\bu_j(s)|\leq \widetilde O(1/\sqrt{d})$, we have $|\beta_{sij}|\leq \widetilde O(1/\sqrt{d})$ and so \[ |\Sigma_3|=\sum_{s<i<j\leq n}|\beta_{sij}|^3\leq  \widetilde O\Big(\frac{n^2}{d^{3/2}}\Big). \]

On the other hand, $\Sigma_2$ can be written as $\Sigma_2=\bw^T M \bw$, where $\bw=(\bu_{s+1}(s)^2, \dots, \bu_{n}(s)^2)$ and $M$ is the $(n-s)\times(n-s)$ symmetric matrix with zero diagonal given by $M_{ij}=d|\alpha_{sij}|\lambda_{sij}/4$.

\begin{claim}\label{clm:op_bound}
We have $\|M\|_{op}\leq \widetilde O(d\sqrt{np}+n^{3/2}\sqrt{d})$.
\end{claim}
\begin{proof}
We begin by observing that $|\alpha_{sij}|\lambda_{sij}$ is very close to $z_p\lambda_{ij}$. Indeed, observe that $|\alpha_{sij}|=z_p\pm \sqrt{d}\langle \pi_{s-1}(\bv_i), \pi_{s-1}(\bv_j)\rangle=z_p+\widetilde O(\sqrt{n/d})$, due to the event $\Pi_{s-1}$. On the other hand, since $\lambda(t)= \phi(t)/\Phi(t)$ is $1$-Lipschitz (see Lemma~\ref{lemma:mills_ratio}), we have $|\lambda_{sij}-\lambda_{ij}| \leq \big||\alpha_{sij}|-z_p\big| \leq \widetilde O\Big(\frac{\sqrt n}{\sqrt d}\Big)$, which means that $\lambda_{sij}=\lambda_{ij}+\widetilde O(\sqrt{n}/\sqrt{d})$. Therefore
\[ \big||\alpha_{sij}|\lambda_{sij}-z_p\lambda_{ij}\big|= \Big|\Big(z_p+\widetilde O\Big(\frac{\sqrt{n}}{\sqrt{d}}\Big)\Big)\Big(\lambda_{ij}+\widetilde O\Big(\frac{\sqrt{n}}{\sqrt{d}}\Big)\Big)-z_p\lambda_{ij}\Big|
\leq \widetilde O\Big(\frac{\sqrt n}{\sqrt d}\Big) \text{ (since $|\lambda_{ij}|, z_p=\widetilde O(1)$)}. \]

Hence, we can write $M=\Lambda+R$, where $\Lambda_{ij}=dz_p\lambda_{ij}/4$ and all entries of $R$ are bounded by $\widetilde O(\sqrt{nd})$, which implies $\|R\|_{op}\leq \widetilde O(n^{3/2}\sqrt{d})$. By property~\ref{item:eigenvalue_bound} of the event $\cE$, we have $\|\Lambda\|_{op}\leq \widetilde O(d\sqrt{np})$ and so $\|M\|_{op}\leq \widetilde O(d\sqrt{np}+n^{3/2}\sqrt{d})$ (by triangle inequality).
\end{proof}

Finally, observe that $\|\bw\|_2^2=\sum_{s<i\leq n}\bu_i(s)^4\leq \widetilde O(n/d^2)$, since $|\bu_i(s)|\leq \widetilde O(1/\sqrt{d})$. Thus, we have $|\Sigma_2|\leq \|M\|_{op}\|\bw\|_2^2\leq \widetilde O(\frac{n^{3/2}\sqrt{p}}{d}+\frac{n^{5/2}}{d^{3/2}})$. This completes the proof of Claim~\ref{claim:reduction_to_Sigma_1}.
\end{proof}

Having bounded the product of ratios in terms of $\Sigma_1$, our goal is to estimate $\bE\big[\mathbf{1}_{\cT_s}e^{Z_s-Z_{s-1}+\Sigma_1}\big]$. In order to do this, we begin by explaining what $Z_s-Z_{s-1}$ actually is.

\begin{claim}\label{claim:Z_s_manipulation}
Let $q_{sij}=\lambda_{ij}(z_p-\lambda_{ij}) \sum_{s<r\leq n}\lambda_{ir}\lambda_{jr}/{\sqrt{d}}$ and
{\relscale{0.95}
\begin{align*}
A_s\!=\!\!\!\!\!\!\sum_{s<i<j\leq n}\!\!\!\!\!\!\lambda_{ij}\lambda_{si}\lambda_{sj}\big(&
 (z_p\!-\!\lambda_{ij})\langle\pi_{s-1}(\bv_i),\pi_{s-1}(\bv_j)\rangle\!+\!(z_p\!-\!\lambda_{si})\langle\pi_{s-1}(\bv_i),\pi_{s-1}(\bv_s)\rangle\!+\!(z_p\!-\!\lambda_{sj})\langle\pi_{s-1}(\bv_j),\pi_{s-1}(\bv_s)\rangle\big).
\end{align*}
}
Then, on the event $\cT_s$ we have $Z_s-Z_{s-1}=-A_s+\sum_{s<i<j\leq n}q_{sij}\beta_{sij}$.
\end{claim}
\begin{proof}
Let us compare the terms contained in $Z_s$ and $Z_{s-1}$. The terms in $Z_{s-1}$ split into the terms with $r,i,j>s$ and the boundary terms in which one of $r,i,j$ is equal to $s$. For the first set of terms, the index set is the same as in $Z_s$, and only the set of coordinates where the projection happens changes. Since $\langle \pi_s(\bv_i),\pi_s(\bv_j)\rangle=\langle \pi_{s-1}(\bv_i),\pi_{s-1}(\bv_j)\rangle+\bv_i(s)\bv_j(s)$, these common terms contribute $\sum_{\substack{s<i<j\leq n\\s<r\leq n}}\lambda_{ij}\lambda_{ri}\lambda_{rj}(z_p-\lambda_{ij})\bv_i(s)\bv_j(s)=\sum_{s<i<j\leq n}q_{sij}\beta_{sij}$ to $Z_s-Z_{s-1}$. On the other hand, the terms where one of the indices $r,i,j$ is equal to $s$ contribute precisely $-A_s$. Hence, we have  $Z_s-Z_{s-1}=-A_s+\sum_{s<i<j\leq n}q_{sij}\beta_{sij}$.
\end{proof}

Note that $A_s$ does not depend on the random coordinates $\{\bv_i(s)|s<i\leq n\}$, and for the purposes of estimating $\bE\big[\mathbf{1}_{\cT_s}e^{Z_s-Z_{s-1}+\Sigma_1}\big]$, we can temporarily forget about it. Let $S=\sum_{s<i<j\leq n}(\lambda_{sij}+q_{sij})\beta_{sij}-\lambda_{sij}^2\beta_{sij}^2/2$. By Claim~\ref{claim:Z_s_manipulation}, we have $\Sigma_1+Z_s-Z_{s-1}=S-A_s$, and we bound $\bE\big[\mathbf{1}_{\cT_s}e^{S}\big]\leq \bE\big[e^{S}\big]$ in the following lemma.

\begin{lemma}\label{lemma:S_bound}
Under the \setupstar{}, we have
\begin{align*}
\log \bE[e^{S}]&\leq A_s+\frac{1}{\sqrt{d}}\sum_{s<i<j\leq n}\lambda_{si}\lambda_{sj}\lambda_{ij}-\frac{1}{2d}\sum_{s<i<j\leq n}\lambda_{si}^2\lambda_{sj}^2\lambda_{ij}^2
+\frac{1}{d}\!\!\sum_{\substack{s<r\leq n\\ s<i< j\leq n}}\!\!\!\!\!\! \lambda_{si}\lambda_{sj}\lambda_{ri}\lambda_{rj}\\
&+\frac{1}{d}\!\!\sum_{\substack{s<r\leq n\\ s<i< j\leq n}}\!\!\!\!\!\! \lambda_{si}\lambda_{sj}\lambda_{ri}\lambda_{rj}\big(\lambda_{ij}(z_p-\lambda_{ij})+\lambda_{sr}(z_p-\lambda_{sr})\big) +\widetilde O\Big(\frac{n^{3/2}p^{1/2}}{d}+\frac{p^2n^3}{d^{3/2}}+\frac{p^2n^4}{d^2}\Big) .
\end{align*}
\end{lemma}

Combining Claim~\ref{claim:Z_s_manipulation} and Lemma~\ref{lemma:S_bound}, we obtain precisely the estimate needed by Lemma~\ref{lemma:ratio}, and so the proof is complete.\qedhere
\end{proof}

\section{Bounding the remaining MGF}\label{sec:MGF}

In this section, we will show the bound on $\bE[e^{S}]$ required by Lemma~\ref{lemma:S_bound}. Recall the definition of $S$:
\begin{align*}
S=\underbrace{\!\!\!\!\!\!\sum_{s<i<j\leq n}\!\!\!\! (\lambda_{sij}+q_{sij})\beta_{sij}}_{X}-\underbrace{\sum_{s<i<j\leq n}\!\!\!\! \frac{\lambda_{sij}^2\beta_{sij}^2}{2}}_{Y},\text{ where } \beta_{sij}&=\sqrt{d}\bu_i(s)\bu_j(s)\text{ and }
q_{sij}=\frac{\lambda_{ij}(z_p-\lambda_{ij})}{\sqrt{d}}\!\!\!\!\sum_{s<r\leq n}\!\!\!\!\lambda_{ri}\lambda_{rj}.
\end{align*}

A nice way to think about $S$ is as a quadratic form of the vector $\bu$ defined as $\bu=(\bu_{s+1}(s), \dots, \bu_{n}(s))$. Indeed, if we set $m=n-s$ and define an $m\times m$ matrix $M$ by $M_{ji}=M_{ij}=\frac{\sqrt{d}}{2}(\lambda_{sij}+q_{sij})$ for $i<j$ and $M_{ii}=0$, then the first sum of $S$ equals $\bu^T M \bu$. Additionally, we will write $S=X-Y$, where $X=\bu^T M \bu$ and $Y=\sum_{s<i<j\leq n}\lambda_{sij}^2\beta_{sij}^2/2$. Note that $\lambda_{sij}+q_{sij}$ is should be thought of as very close to $\lambda_{ij}$, because Claim~\ref{claim:error_in_lambda} shows that $\lambda_{sij}=(1+\widetilde O(\sqrt{n/d})\lambda_{ij}$ and $|q_{sij}|\leq \widetilde O\big(|\lambda_{ij}|\big|\sum_{s<r\leq n}\lambda_{sr}\big|/\sqrt{d}\big)\leq \widetilde O\big(|\lambda_{ij}|\sqrt{n/d}\big)$, where we have used property~\ref{item:degree_cancellation} of the event $\cE$.

We bound $\bE[e^{S}]$ by using the following two equalities:
\begin{align*}
\bE[e^{X-Y}] &=\bE[e^X]\exp\Big(-\bE[Y]+\widetilde O\Big(\frac{n^{3/2}p}{d}+\frac{n^4p^2}{d^2}\Big)\Big),\text{ and}\\
\bE[e^X] &=\exp\Big(\bE X+\frac{1}{2}\Var(X) +\widetilde O\Big(\frac{n^{5/2}p^{3/2}}{d^{3/2}} +\frac{n^{7/2}p^{3/2}}{d^2}\Big)\Big).
\end{align*}
Since $d\geq Cn^2p(\log n)^A$ and $p\geq n^{-1/3}\polylog(n)$, the error terms in the second equality are absorbed by those in the first. Hence, $\bE[e^{X-Y}] =\exp\Big(\bE X+\frac{1}{2}\Var(X)-\bE Y +\widetilde O\Big(\frac{n^{3/2}p}{d}+\frac{n^4p^2}{d^2}\Big)\Big)$.
The final step of the argument will be to compute the value of $\bE X+\frac{1}{2}\Var(X)-\bE Y$. We divide the rest of the section into three parts: in the first two, we prove the above two inequalities, while in the last we calculate the value of $\bE X+\frac{1}{2}\Var(X)-\bE Y$ precisely, and put all the pieces together to prove Lemma~\ref{lemma:S_bound}.

\subsection{Separating the variables} Let us begin by showing a general trick for separating the variables in the expectations of the type $\bE[e^{X-Y}]$, and then verify that this trick applies in our situation.

\begin{lemma}\label{lemma:tilting}
Let $X$ be a random variable with the property that $\bE[e^{2(X-\bE[X])}]\leq 2$. Then, for any nonnegative random variable $Y$, we have \[\bE[e^{X-Y}]=\bE[e^X]e^{-\bE[Y]+O(\sqrt{\Var(Y)}+\sqrt{\bE[Y^4]})}.\]
\end{lemma}
\begin{proof}
In this proof, we may assume that $X$ is centered, i.e. $\bE[X]=0$, since subtracting $\bE X$ from $X$ only scales the inequality $\bE[e^{X-Y}]\leq \bE[e^X]e^{-\bE[Y]+O(\sqrt{\Var(Y)}+\sqrt{\bE[Y^4]})}$ by $e^{-\bE X}$ on both sides.
Let $\nu$ be the tilted measure given by $e^X/\bE[e^X]$. Since for any positive $t\geq 0$ one has $1-t\leq e^{-t}\leq 1-t+t^2$ and since $Y\geq 0$, we have
\[\frac{\bE[e^{X-Y}]}{\bE[e^X]}= \bE_\nu[e^{-Y}]\leq \bE_\nu[1-Y+Y^2]=1-\bE_\nu[Y]+\bE_\nu[Y^2]\leq e^{-\bE_\nu[Y]+\bE_\nu[Y^2]}.\]
On the lower bound side, we have $\bE_\nu[e^{-Y}]\geq e^{\bE_\nu[-Y]}$, and so $\bE[e^{X-Y}]=\bE[e^X]e^{-\bE_\nu[Y]+O(\bE_\nu[Y^2])}$. Next, observe that $\bE[Y]$ and $\bE_\nu[Y]$ are quite close. To see this, we use Jensen's inequality we can write $\bE[e^X]\geq e^{\bE[X]}=e^0\geq 1$, and so we can write
\begin{align*}
    \Big| \bE_\nu[Y]\!-\!\bE[Y]\Big|&\!=\!\Big| \frac{\bE[e^XY]}{\bE[e^X]}\!-\!\bE[Y]\Big| \!=\! \Big| \frac{\bE[e^XY]\!-\!\bE[e^X]\bE[Y]}{\bE[e^X]}\Big|\!\leq\! \Big|\bE\big[(e^X\!-\!\bE e^X )(Y\!-\!\bE Y)\big]\Big|.
\end{align*}
By the Cauchy-Schwarz inequality, we can bound the last term as follows:
\begin{align*}
    \Big| \bE_\nu[Y]-\bE[Y]\Big|&\leq \bE\big[(e^X-\bE e^X )^2\big]^{1/2}\bE\big[(Y-\bE Y)^2\big]^{1/2}\leq \bE[e^{2X}]^{1/2}\sqrt{\Var(Y)}\leq O(\sqrt{\Var(Y)}).
\end{align*}
On the other hand, to complete the proof, we observe that by the definition of the measure $\nu$ and the Cauchy-Schwarz inequality again, we have
\[\bE_\nu[Y^2]= \frac{\bE[e^X Y^2]}{\bE[e^X]}\leq \bE[e^{2X}]^{1/2} \bE[Y^4]^{1/2}\leq O(\sqrt{\bE[Y^4]}).\]
Hence, $\bE[e^{X-Y}]= \bE[e^X]\bE_\nu[e^{-Y}]= \bE[e^X]e^{-\bE[Y]+O(\sqrt{\Var(Y)}+\sqrt{\bE[Y^4]})}$, as desired.
\end{proof}

To apply Lemma~\ref{lemma:tilting}, we need $\bE[e^{2(X-\bE X)}]\leq 2$ and good bounds on $\Var(Y), \bE[Y^4]$. The bound on $\bE[e^{2(X-\bE X)}]$ will follow from the later detailed analysis of $\bE e^X$, so we will not focus on it now. Instead, we turn our attention to establishing the bounds on $\Var(Y), \bE[Y^4]$.

\begin{claim}\label{claim:Y_bounds}
We have $\Var(Y)\leq \widetilde O\Big(\frac{n^3p^2}{d^2}\Big)$ and $\bE[Y^4]\leq \widetilde O\Big(\frac{n^8p^4}{d^4}\Big)$.
\end{claim}
\begin{proof}
By definition, we have $Y=\sum_{s<i<j\leq n} \lambda_{sij}^2\beta_{sij}^2/2$, so
\[\Var(Y)=\sum_{s<i<j\leq n} \lambda_{sij}^4\Var(\beta_{sij}^2)/4+\sum_{\substack{s<i<j\leq n\\s<i'<j'\leq n\\(i, j)\neq (i', j')}} \lambda_{sij}^2\lambda_{si'j'}^2{\rm Cov}(\beta_{sij}^2, \beta_{si'j'}^2)/4.\]

Since $|\bu_i(s)|\leq \widetilde O(1/\sqrt{d})$ for all $s<i\leq n$, we have that $|\beta_{sij}|\leq \widetilde O(1/\sqrt{d})$. Hence, $\Var(\beta_{sij}^2)=\widetilde O(1/d^2)$ and ${\rm Cov}(\beta_{sij}^2, \beta_{si'j'}^2)=\widetilde O(1/d^2)$. Additionally, if $\{i,j\}\cap\{i',j'\}=\varnothing$, then $\beta_{sij}$ and $\beta_{si'j'}$ are independent. So, we have
\[\Var(Y)=\widetilde O\Big(\frac{1}{d^2}\sum_{s<i<j\leq n} \lambda_{sij}^4+\frac{1}{d^2}\sum_{\substack{s<j<k\leq n\\s<i\leq n}} \lambda_{sij}^2\lambda_{sik}^2\Big).\]
By Claim~\ref{claim:error_in_lambda}, we have $\lambda_{sij}^2=O(\lambda_{ij}^2)=\widetilde O(|\lambda_{ij}|)$ and so
\[\Var(Y)=\widetilde O\Big(\frac{1}{d^2}\sum_{s<i<j\leq n} |\lambda_{ij}|+\frac{1}{d^2}\sum_{\substack{s<j<k\leq n\\s<i\leq n}} |\lambda_{ij}\lambda_{ik}|\Big)\leq \widetilde O\Big(\frac{n^2p}{d^2}+\frac{n^3p^2}{d^2}\Big),\]
where we have used property~\ref{item:degree_concentration} of $\cE$. Finally, summing the deterministic bounds on $\beta_{sij}$ we have
\[ Y=\sum_{s<i<j\leq n}\beta_{sij}^2\lambda_{sij}^2/2\leq \widetilde O\Big(\frac{1}{d}\sum_{s<i<j\leq n}|\lambda_{sij}|\Big)\leq \widetilde O\Big(\frac{n^2p}{d}\Big),\]
where we have again used property~\ref{item:degree_concentration} of $\cE$.
Hence $\bE[Y^4]\leq \widetilde O(n^8p^4/d^4)$, which completes the proof.
\end{proof}

\subsection{\texorpdfstring{Bounding the MGF of $X$}{Bounding the MGF of X}}

Let us now show how to bound $\bE[e^X]=\bE[e^{\bu^T M\bu}]$, by using the following general tool.

\begin{proposition}\label{prop:lindeberg}
There exist absolute constants $C>c>0$, such that the following holds.

Let $\xi \in \bR^m$ be a vector with independent subgaussian coordinates, each with expectation $0$ and variance proxy at most $\sigma^2$.  Also, fix a symmetric zero-diagonal $m\times m$ matrix $M$, and a vector $\ba\in\bR^m$. Let us define
\[\delta_1=\max\big\{\sigma^2\|M\|_{op}, \sigma\|\ba\|_\infty\big\} \text{  and }\delta_2=\sigma^4\|M\|_{HS}^2+\sigma^2\|\ba\|_2^2.\]
If $\delta_1\leq c$, then
\begin{align*}
\log\bE e^{\xi^TM\xi+\ba^T\xi}
=\frac{1}{2}\Var(\xi^TM\xi+\ba^T\xi)+O(e^{C\delta_2} \delta_1\delta_2)
\end{align*}
\end{proposition}

Since Proposition~\ref{prop:lindeberg} is a completely standalone statement, we will postpone its proof to Section~\ref{sec:lindeberg} and focus on how to apply it to prove Lemma~\ref{lemma:S_bound}.

\begin{lemma}\label{lemma:MGF_X}
For $X$ defined as above,
\[\bE[e^X]= \exp\Big(\bE X+\frac{1}{2}\Var(X)+\widetilde O\Big(\frac{n^{5/2}p^{3/2}}{d^{3/2}}+\frac{n^{7/2}p^{3/2}}{d^{2}}\Big)\Big).\]
\end{lemma}
\begin{proof}
Since $X=\bu^T M\bu$, we can apply Proposition~\ref{prop:lindeberg} to bound $\bE[e^X]$. However, Proposition~\ref{prop:lindeberg} requires us to use centered variables, while $\bu$ is not centered. This is simple to fix, as follows. Let $\xi_i=\bu_i(s)-\bE \bu_i(s)$ be the centered version of $\bu_i$, and let $\xi=(\xi_{s+1}, \dots, \xi_n)$. Then, $\xi_i$ are subgaussian of variance proxy $\widetilde O(1/d)$, by Hoeffding's lemma. Also, we have
\[X=\bu^T M \bu=(\bE\bu)^T M \bE\bu+2\xi^T M \bE\bu+\xi^T M\xi.\]
The first of these three terms is deterministic, and therefore we can write
\begin{align*}
\bE[e^X]&= e^{(\bE\bu)^T M \bE\bu}\cdot \bE\Big[\exp\big(2\xi^T M \bE\bu+\xi^T M\xi\big)\Big]=e^{\bE X}\cdot \bE\Big[\exp\big(2\xi^T M \bE\bu+\xi^T M\xi\big)\Big].
\end{align*}
If we set $\ba=2M^T \bE \bu$, then Proposition~\ref{prop:lindeberg} gives
\[\bE[e^X]=e^{\bE X+\frac{1}{2}\Var(X)}e^{O(e^{C\delta_2}\delta_1\delta_2)},\]
where $\delta_1=\max\big\{\sigma^2\|M\|_{op}, \sigma\|\ba\|_\infty\big\}$ and $\delta_2=\sigma^4\|M\|_{HS}^2+\sigma^2\|\ba\|_2^2$, provided $\delta_1$ is sufficiently small. Let us first show that $\delta_1$ is indeed sufficiently small.

\begin{claim}\label{claim:calculations_norms}
We have $\|M\|_{op}\leq \widetilde O(\sqrt{dnp}+\sqrt{n^3p}), \|M\|_{HS}^2\leq \widetilde O(dn^2p)$, $\|\ba\|_{\infty}\leq \widetilde O(\sqrt{np})$, and $\|\ba\|_{2}^2\leq \widetilde O(n^2p)$.
\end{claim}
\begin{proof}
We can write $M_{ij}=\frac{\sqrt{d}}{2}(\lambda_{ij}+(\lambda_{sij}-\lambda_{ij}+q_{sij}))$. If $\Lambda$ denotes the $m\times m$ matrix of $\lambda_{ij}$, then we can write $M=\frac{\sqrt{d}}{2}\Lambda+R$, where $R$ is a matrix whose entries are given by $R_{ij}=\frac{\sqrt{d}}{2}(\lambda_{sij}-\lambda_{ij}+q_{sij})$. Recall that we have $\|\Lambda\|_{op}\leq \widetilde O(\sqrt{np})$ and $|\lambda_{sij}-\lambda_{ij}|, |q_{sij}|\leq \widetilde O(|\lambda_{ij}|\sqrt{n/d})$ (by properties~\ref{item:degree_cancellation} and~\ref{item:eigenvalue_bound} of the event $\cE$ and Claim~\ref{claim:error_in_lambda}). The latter inequality implies that $\|R\|_{HS}^2\leq \sum_{s<i<j\leq n}|\lambda_{ij}|^2\cdot \widetilde O(n)\leq \widetilde O(n^3p)$, by property~\ref{item:degree_concentration} of the event $\cE$.

In total, we have $\|M\|_{op}\leq \sqrt{d}\|\Lambda\|_{op}+\|R\|_{op}\leq \widetilde O(\sqrt{dnp})+\|R\|_{HS}\leq \widetilde O(\sqrt{dnp}+\sqrt{n^3p})$.

To compute $\|M\|_{HS}^2$, we argue similarly. We have $\|M\|_{HS}^2\leq O(d\| \Lambda\|_{HS}^2+\|R\|_{HS}^2)\leq \widetilde O(dn^2p+n^3p)=\widetilde O(dn^2p)$, where we have used that $d\geq n$ and property~\ref{item:degree_concentration} of the event $\cE$ to bound $\|\Lambda\|_{HS}^2=2\sum_{s<i<j\leq n}\lambda_{ij}^2\leq \widetilde O(n^2p)$.

Finally, we have $\ba=2M\bE\bu$, and so $\ba_i=2\sum_{s<j\leq n}M_{ij}\bE[\bu_j(s)]$ for each $i$. Let us fix $i$ so that $||\ba||_\infty = |\ba_i|$. By Proposition~\ref{prop:comparing_moments} \ref{item:comparing_first_order}, we have $\bE[\bu_j(s)]= (1+\widetilde O(\sqrt{n/d}))\lambda_{sj}/\sqrt{d}$, and we have already argued that $M_{ij}=\frac{\sqrt{d}}{2}\lambda_{ij}(1+\widetilde O(\sqrt{n/d}))$. Hence, $\ba_i=\sum_{s<j\leq n}\lambda_{ij}\lambda_{sj}(1+\widetilde O(\sqrt{n/d}))$, and so
\[\|\ba\|_\infty\leq \widetilde O\Big(\sqrt{np}+\sqrt{n/d}\cdot \sum_{s<j\leq n}|\lambda_{sj}||\lambda_{ij}|\Big)\leq \widetilde O(\sqrt{np}+\sqrt{n/d}\cdot np^2)=\widetilde O(\sqrt{np}),\]
 where we have used properties~\ref{item:degree_cancellation} and \ref{item:degree_concentration} of the event $\cE$ to bound the main term and error respectively. Consequently, we also have $\|\ba\|_2^2\leq n\|\ba\|_\infty^2\leq \widetilde O(n^2p)$, thus completing the proof.
\end{proof}

Now, $\delta_1=\max\{\sigma^2\|M\|_{op}, \sigma\|\ba\|_\infty\}$, where $\sigma=\widetilde O(1/\sqrt{d})$. By Claim~\ref{claim:calculations_norms}, we have $\sigma^2\|M\|_{op}\leq \widetilde O(\sqrt{np/d}+\sqrt{n^3p}/d)\ll 1$, and $\sigma\|\ba\|_\infty\leq \widetilde O(\sqrt{np/d})\ll 1$, since $d\gg n^2p$. Hence, $\delta_1$ is sufficiently small, and we can apply Proposition~\ref{prop:lindeberg}.

A similar calculation shows that $\delta_2=\sigma^4\|M\|_{HS}^2+\sigma^2\|\ba\|_2^2\leq \widetilde O(n^2p/d+n^2p/d)=\widetilde O(n^2p/d)$. By our assumption on $d$, we also get that $\delta_2\leq 1$. Hence, we get
\[\bE[e^X]=\exp\Big(\bE X+\Var(X)/2+O(\delta_1\delta_2)\Big)= \exp\Big(\bE X+\Var(X)/2+\widetilde O\Big(\frac{n^{5/2}p^{3/2}}{d^{3/2}}+\frac{n^{7/2}p^{3/2}}{d^{2}}\Big)\Big).\qedhere\]
\end{proof}

\subsection{The final calculation}

In order to complete the proof of Lemma~\ref{lemma:S_bound}, the only real piece of work left is to estimate $\bE X+\frac{1}{2}\Var(X)-\bE Y$, which we will now do.

\begin{lemma}\label{lemma:EX_formula}
Recalling the definition of $A_s$ from the statement of Claim~\ref{claim:Z_s_manipulation}, we have
\begin{align}\label{eqn:EX_formula}
\bE X\!=\!A_s\!+\!\frac{1}{\sqrt{d}}\!\sum_{s<i<j\leq n} \!\!\!\! \lambda_{si}\lambda_{sj}\lambda_{ij}+\frac{1}{d}\!\!\!\!\sum_{s<i<j\leq n} \!\!\!\! \lambda_{si}\lambda_{sj}\lambda_{ij}(z_p-\lambda_{ij}) \!\sum_{s<r\leq n}\!\!\lambda_{ir}\lambda_{jr} +\widetilde O\Big(\frac{n^3p^2}{d^{3/2}}\Big).
\end{align}
\end{lemma}
\begin{proof}
From the definition of $X$, we have
\begin{align}\label{eqn:EX_definition}
    \bE X=\sqrt{d}\sum_{s<i<j\leq n}\lambda_{sij}\bE[\bu_i(s)]\bE[\bu_j(s)]+\sqrt{d}\sum_{s<i<j\leq n}q_{sij}\bE[\bu_i(s)]\bE[\bu_j(s)].
\end{align}
We will show that the first sum corresponds to $A_s+\frac{1}{\sqrt{d}}\sum_{s<i<j\leq n}
 \lambda_{si}\lambda_{sj}\lambda_{ij}$, while the second sum corresponds to the last term of (\ref{eqn:EX_formula}).
By Proposition~\ref{prop:comparing_moments} \ref{item:comparing_first_order}, and by Claim~\ref{claim:error_in_lambda}, we have
\begin{align*}
\bE[\bu_i(s)]&=\frac{\lambda_{si}}{\sqrt{d}}+\lambda_{si}(z_p-\lambda_{si})\langle \pi_{s-1}(\bv_i),\pi_{s-1}(\bv_s)\rangle+\widetilde O\Big(\frac{n}{d^{3/2}}\Big),\\
\lambda_{sij}&=\lambda_{ij}+\lambda_{ij}(z_p-\lambda_{ij})\sqrt{d}\,\langle \pi_{s-1}(\bv_i),\pi_{s-1}(\bv_j)\rangle+\widetilde O\Big(\frac{n}{d}\Big),\\
\bE[\bu_j(s)]&=\frac{\lambda_{sj}}{\sqrt{d}}+\lambda_{sj}(z_p-\lambda_{sj})\langle \pi_{s-1}(\bv_j),\pi_{s-1}(\bv_s)\rangle+\widetilde O\Big(\frac{n}{d^{3/2}}\Big),
\end{align*}
To simplify the notation, let us briefly write $\Delta_{si}=\lambda_{si}(z_p-\lambda_{si})\langle \pi_{s-1}(\bv_i),\pi_{s-1}(\bv_s)\rangle$ and similarly for $\Delta_{sj}, \Delta_{ij}$. Then, the first sum of (\ref{eqn:EX_definition}) equals
\[\sqrt{d}\!\!\!\sum_{s<i<j\leq n}\!\!\!\lambda_{sij}\bE[\bu_i(s)]\bE[\bu_j(s)]=d\!\!\!\sum_{s<i<j\leq n}\!\!\!\Big(\frac{\lambda_{si}}{\sqrt{d}}+\Delta_{si}+\widetilde O\big(\frac{n}{d^{3/2}}\big)\Big)\Big(\frac{\lambda_{ij}}{\sqrt{d}}+\Delta_{ij}+\widetilde O\big(\frac{n}{d^{3/2}}\big)\Big)\Big(\frac{\lambda_{sj}}{\sqrt{d}}+\Delta_{sj}+\widetilde O\big(\frac{n}{d^{3/2}}\big)\Big).\]
In the above expression, every parenthesis has a leading term (e.g. $\lambda_{si}/\sqrt{d}$), a correction $\Delta_{si}$ and an error term. Note that, due to property~\ref{item:inner_product} of the event $\Pi_{s-1}$, we have $|\Delta_{si}|\leq \widetilde O(|\lambda_{si}|\sqrt{n}/d)$.

The total contribution of all error terms is at most $\widetilde O(d\frac{n}{d^{3/2}})\cdot \sum_{s<i<j\leq n}\big(|\lambda_{si}\lambda_{sj}|+|\lambda_{si}\lambda_{ij}|+|\lambda_{ij}\lambda_{sj}|\big)/d\leq \widetilde O({n^3p^2}/{d^{3/2}})$, which is acceptable.

Similarly, all the terms coming from the product of at least two of the correction factors $\Delta_{si},\Delta_{sj},\Delta_{ij}$ contribute at most $O(d)\cdot \sum_{s<i<j\leq n} |\lambda_{si}||\lambda_{ij}||\lambda_{sj}|\cdot \widetilde O(\frac{\sqrt{n}}{d}\cdot \frac{\sqrt{n}}{d}\cdot \frac{1}{\sqrt{d}})=\widetilde O(n^3p^3/d^{3/2})$ due to property~\ref{item:triangle_concentration} of the event $\cE$.

Hence, we conclude that \[\sqrt{d}\!\!\!\sum_{s<i<j\leq n}\!\!\!\lambda_{sij}\bE[\bu_i(s)]\bE[\bu_j(s)]=\!\!\!\sum_{s<i<j\leq n}\!\!\!\frac{\lambda_{si}\lambda_{sj}\lambda_{ij}}{\sqrt{d}}+\!\!\!\sum_{s<i<j\leq n}\!\!\! \big(\lambda_{si}\lambda_{sj}\Delta_{ij}+\lambda_{si}\lambda_{ij}\Delta_{sj}+\lambda_{sj}\lambda_{ij}\Delta_{si}\big)+\widetilde O\Big(\frac{n^3p^2}{d^{3/2}}\Big).\]
The latter sum is $A_s$, by definition, which completes the first part of the proof.

In the second part, we evaluate $\sqrt{d}\sum_{s<i<j\leq n}q_{sij}\bE[\bu_i(s)]\bE[\bu_j(s)]$. By using Proposition~\ref{prop:comparing_moments} \ref{item:comparing_first_order}, we get that $\bE[\bu_i(s)]=(1+\widetilde O(\sqrt{n/d}))\frac{\lambda_{si}}{\sqrt{d}}, \bE[\bu_j(s)]=(1+\widetilde O(\sqrt{n/d}))\frac{\lambda_{sj}}{\sqrt{d}}$, and so
\[\sqrt{d}\sum_{s<i<j\leq n}q_{sij}\bE[\bu_i(s)]\bE[\bu_j(s)]=\frac{1}{d}\sum_{s<i<j\leq n} (1+\widetilde O(\sqrt{n/d}))\lambda_{si}\lambda_{sj}\lambda_{ij}(z_p-\lambda_{ij})\sum_{s<r\leq n}\lambda_{ri}\lambda_{rj}.\]
To complete the proof, note that the error term contributed by $\widetilde O(\sqrt{n/d})$ is at most
\[\widetilde O\Big(\frac{\sqrt{n/d}}{d}\Big)\sum_{s<i<j\leq n}|\lambda_{si}\lambda_{sj}\lambda_{ij}|\cdot \Big|\sum_{s<r\leq n}\lambda_{ri}\lambda_{rj}\Big|\leq \widetilde O\Big(\frac{\sqrt{n/d}}{d}\sum_{s<i<j\leq n}|\lambda_{si}\lambda_{sj}\lambda_{ij}|\cdot \sqrt{n} p\Big)\leq \widetilde O\Big(\frac{n^3p^4}{d^{3/2}}\Big),\]
where we have used properties~\ref{item:degree_cancellation} and \ref{item:triangle_concentration} of the event $\cE$ to respectively bound $\Big|\sum_{s<r\leq n}\lambda_{ri}\lambda_{rj}\Big|\leq \widetilde O(\sqrt{n}p)$ and $\sum_{s<i<j\le n}|\lambda_{si}\lambda_{sj}\lambda_{ij}|\le \widetilde{O}(n^2p^3)$.
\end{proof}

\begin{lemma}\label{lemma:final_computation}
We have \begin{align*}
    \frac{1}{2}\Var(X)-\bE[Y]=-\!\!\!\!\!\sum_{s< i<j\leq n}\!\!\!\!\! \frac{\lambda_{si}^2\lambda_{sj}^2\lambda_{ij}^2}{2d}+ \frac{1}{d}\!\!\!\!\!\sum_{\substack{s<r\leq n\\s< i<j\leq n}}\!\!\!\!\! &\lambda_{si}\lambda_{sj}\lambda_{ri}\lambda_{rj}(1+\lambda_{sr}(z_p-\lambda_{sr}))+\widetilde O\Big( \frac{n^3p^2}{d^{3/2}} \Big).
\end{align*}
\end{lemma}
\begin{proof}
By a slight abuse of notation, we will write $u_i=\bu_i(s)$, in the hope of simplifying the formulas.

Since $X=\sum_{s<i, j\leq n} u_i M_{ij} u_j$, we can write \[\Var(X)=\sum_{s<i, j\leq n} M_{ij}^2\Var(u_iu_j)+\sum_{\substack{s<i, j\leq n\\s<i', j'\leq n\\(i,j)\neq(i',j')}} M_{ij}M_{i'j'}{\rm Cov}(u_iu_j, u_{i'}u_{j'}).\]

Observe that if $\{i, j\}\cap \{i', j'\}=\varnothing$, then ${\rm Cov}(u_iu_j, u_{i'}u_{j'})=0$, since $u_i, u_j, u_{i'}, u_{j'}$ are independent. On the other hand, ${\rm Cov}(u_iu_j, u_iu_k)=\bE[u_i^2u_ju_k]-\bE[u_iu_j]\bE[u_iu_k]=(\bE[u_i^2]-\bE[u_i]^2)\bE[u_j]\bE[u_k]$. So,
\begin{align}\label{eqn:VarX_expression}
\frac{1}{2}\Var(X)=\frac{d}{2}\!\!\!\!\sum_{s<i< j\leq n} \!\!\!\!(\lambda_{sij}+q_{sij})^2\big(\bE[u_i^2u_j^2]-\bE[u_iu_j]^2\big)+4\!\!\!\!\sum_{\substack{s<i\leq n\\s<j<k\leq n}} \!\!\!\!M_{ij}M_{ik}\Var(u_i)\bE[u_j]\bE[u_k].
\end{align}

On the other hand, we have $\bE[Y]=\frac{d}{2}\sum_{s<i<j\leq n} \lambda_{sij}^2 \bE[u_i^2u_j^2]$. So, writing $S:= \frac{1}{2}\Var(X)-\bE Y$, we get that
\begin{align*}
S =&\frac{d}{2}\!\!\sum_{s<i<j\leq n}
\!\!\!\!\!\!(2\lambda_{sij}q_{sij} + q_{sij}^2)\bE[u_i^2u_j^2]
-\!\!\sum_{\substack{s<i,j\leq n}}\!\!\!\!M_{ij}^2\bE[u_i]^2\bE[u_j]^2+4
\!\!\!\!\!\!\sum_{\substack{s<i\leq n\\s<j<k\leq n}}\!\!\!\!\!\!
M_{ij}M_{ik}\Var(u_i)\bE[u_j]\bE[u_k].
\end{align*}

The first sum is negligible: since $|2\lambda_{sij}+q_{sij}|\leq O(|\lambda_{ij}|)$, $|q_{sij}|\leq \widetilde O(p\sqrt{n/d})$ and $\bE[u_i^2u_j^2]\leq \widetilde O(1/d^2)$, the total contribution of the first sum is at most $\widetilde O(d\cdot \sum_{s<i<j\leq n}|\lambda_{ij}|p\sqrt{n/d}/d^2)=\widetilde O(n^{5/2}p^2/d^{3/2})$, where we have used property~\ref{item:degree_concentration} of the event $\cE$.

The second sum can be estimated as follows. We have $M_{ij}=\frac{\sqrt{d}}{2}(\lambda_{sij}+q_{sij})=(1+\widetilde O(\sqrt{n/d}))\frac{\sqrt{d}}{2}\lambda_{ij}$, since $\lambda_{sij}=(1+\widetilde O(\sqrt{n/d}))\lambda_{ij}$ and $q_{sij}=\widetilde O(|\lambda_{ij}|\sqrt{n/d})$.\hide{can win other lambda} Also, by Proposition~\ref{prop:comparing_moments} \ref{item:comparing_first_order} we have $\bE[u_i]=\frac{\lambda_{si}}{\sqrt{d}}(1+\widetilde O(\sqrt{n/d}))$. So
\[\sum_{s<i,j\leq n}M_{ij}^2\bE[u_i]^2\bE[u_j]^2=(1+\widetilde O(\sqrt{n/d}))\sum_{s<i<j\leq n}\frac{\lambda_{si}^2\lambda_{sj}^2\lambda_{ij}^2}{2d}.\]
Additionally, note that the error term $\widetilde O(\sqrt{n/d})$ can be removed, since the total contribution of this term is at most $\widetilde O(\sqrt{n/d} \sum_{s<i<j\leq n}\lambda_{ij}^2\lambda_{si}^2\lambda_{sj}^2/d)=\widetilde O(n^{5/2}p^{3}/d^{3/2})$, by property~\ref{item:triangle_concentration} of the event $\cE$.

Let us now discuss the last sum of (\ref{eqn:VarX_expression}). If we had $M_{ij}=\frac{\sqrt{d}}{2}\lambda_{ij}, M_{ik}=\frac{\sqrt{d}}{2}\lambda_{ik}$, $\bE[u_j]=\frac{\lambda_{sj}}{\sqrt{d}}, \bE[u_k]=\frac{\lambda_{sk}}{\sqrt{d}}$ and $\Var(u_i)=\frac{1+\lambda_{si}(z_p-\lambda_{si})}{d}$, then the last sum would exactly equal $\frac{1}{d}\sum_{\substack{s<r\leq n\\s<i<j\leq n}}\lambda_{si}\lambda_{sj}\lambda_{ri}\lambda_{rj}(1+\lambda_{sr}(z_p-\lambda_{sr}))$. However, in all of the above expressions, we have a multiplicative error of $1+\widetilde O(\sqrt{n/d})$, and controlling their contribution requires a little bit of work in this situation. Calling this sum $T$, we have
\begin{align*}
    T=4\!\!\!\!\!\sum_{\substack{s<i\leq n\\s<j<k\leq n}}\!\!\!\!\!M_{ij}M_{ik}\Var(u_i)\bE[u_j]\bE[u_k]=&4\sum_{s<i\leq n}\Var(u_i)\sum_{s<j<k\leq n}M_{ij}M_{ik}\bE[u_j]\bE[u_k]\\
    =&2\!\!\sum_{s<i\leq n}\!\!\Var(u_i)\Big( \sum_{s<j\leq n}M_{ij}\bE[u_j]\Big)^2\!-2\!\sum_{s<i\leq n}\!\!\Var(u_i)\sum_{s<j\leq n}M_{ij}^2\bE[u_j]^2.
\end{align*}
For each fixed $i$, we have $\sum_{s<j\leq n}M_{ij}\bE[u_j]=\frac{1}{2}\sum_{s<j\leq n}(1+\widetilde O(\sqrt{n/d}))\lambda_{ij}\lambda_{sj}=\frac{1}{2}\sum_{s<j\leq n}\lambda_{ij}\lambda_{sj}+\widetilde O(n^{3/2}p^2/\sqrt{d})$, where the error term comes from property~\ref{item:degree_concentration} of the event $\cE$, since $\sum_{s<j\leq n}|\lambda_{ij}\lambda_{sj}|\leq \widetilde O(np^2)$. Thus, we have $\big(\sum_{s<j\leq n}M_{ij}\bE[u_j]\big)^2=\frac{1}{4}\big(\sum_{s<j\leq n}\lambda_{ij}\lambda_{sj}\big)^2+\widetilde O(n^{2}p^3/\sqrt{d})$, where we have used that $\big|\sum_{s<j\leq n}\lambda_{ij}\lambda_{sj}\big|\leq \widetilde O(\sqrt{n} p)$ by property~\ref{item:degree_cancellation} of the event $\cE$.

Similarly, $\sum_{s<j\leq n}M_{ij}^2\bE[u_j]^2=\frac{1}{4}(1+\widetilde O(\sqrt{n/d}))\sum_{s<j\leq n}\lambda_{ij}^2\lambda_{sj}^2=\frac{1}{4}\sum_{s<j\leq n}\lambda_{ij}^2\lambda_{sj}^2+\widetilde O(n^{3/2}p^2/\sqrt{d})$, where we have used property~\ref{item:degree_concentration} of the event $\cE$ to bound $\sum_{s<j\leq n}\lambda_{ij}^2\lambda_{sj}^2\leq \widetilde O(np^2)$. Putting all of this together (and using that $p\ge n^{-1/2}$), we get that
\begin{align*}
    T&=\frac{1}{2}\sum_{s<i\leq n}\Var(u_i)\Big(\sum_{s<j\leq n}\lambda_{ij}\lambda_{sj}\Big)^2-\frac{1}{2}\sum_{s<i\leq n}\Var(u_i)\sum_{s<j\leq n}\lambda_{ij}^2\lambda_{sj}^2+\widetilde O\Big(\frac{n^{3}p^3}{d^{3/2}}\Big).
\end{align*}
Since $\Var(u_i)=(1+\widetilde O(\sqrt{n/d}))\frac{1+\lambda_{si}(z_p-\lambda_{si})}{d}$, we can bound the error term induced by $\widetilde O(\sqrt{n/d})$ in each of the sums by $\widetilde O(\sqrt{n/d}\cdot n/d\cdot np^2)=\widetilde O(n^{5/2}p^2/d^{3/2})$ (again using the bounds $\big|\sum_{s<j\leq n}\lambda_{ij}\lambda_{sj}\big|\leq \widetilde O(\sqrt{n}p)$ and $\sum_{s<j\leq n}\lambda_{ij}^2\lambda_{sj}^2\leq \widetilde O(np^2)$ to control the first and second sums, respectively). Hence, we get that
\begin{align*}
    T&=\frac{1}{2}\sum_{s<i\leq n}\frac{1+\lambda_{si}(z_p-\lambda_{si})}{d}\Big(\sum_{s<j\leq n}\lambda_{ij}\lambda_{sj}\Big)^2-\frac{1}{2}\sum_{s<i\leq n}\frac{1+\lambda_{si}(z_p-\lambda_{si})}{d}\sum_{s<j\leq n}\lambda_{ij}^2\lambda_{sj}^2+\widetilde O\Big(\frac{n^{3}p^3}{d^{3/2}}\Big)\\
    &=\frac{1}{d}\sum_{\substack{s<i\leq n\\s<j<k\leq n}}(1+\lambda_{si}(z_p-\lambda_{si}))\lambda_{ij}\lambda_{ik}\lambda_{sj}\lambda_{sk}+\widetilde O\Big(\frac{n^{3}p^3}{d^{3/2}}\Big).\qedhere
\end{align*}
\end{proof}

\begin{proof}[Proof of Lemma~\ref{lemma:S_bound}.]
The first step of the proof is to apply Lemma~\ref{lemma:tilting} to the random variables $X$ and $Y$. To do that, we must first verify that $\bE[e^{2(X-\bE X)}]\leq 2$.

\begin{claim}\label{claim:MGF_X_bound}
We have $\bE[e^{2(X-\bE X)}]\leq 2$.
\end{claim}
\begin{proof}
Using the same argument as in Lemma~\ref{lemma:MGF_X}, we can show that $\bE[e^{2(X-\bE X)}]=\exp(2\Var(X)+O(\delta_1\delta_2))$, where $\delta_1=\max\{\sigma^2\|M\|_{op}, \sigma\|\ba\|_\infty\}$ and $\delta_2=\sigma^4\|M\|_{HS}^2+\sigma^2\|\ba\|_2^2$. By Claim~\ref{claim:calculations_norms}, we have $\delta_1=\widetilde O(\sqrt{np/d})+\widetilde O(\sqrt{n^3p}/d)\ll 1$ and $\delta_2=\widetilde O(n^2p/d)\ll 1$, since $d\gg n^2p$.

Further, by Lemma~\ref{lemma:final_computation}, we have
\[\Var(X)\leq 2\bE Y+ \frac{2}{d}\sum_{s<i\leq n}(1+\lambda_{si}(z_p-\lambda_{si}))\Big(\sum_{s<j\leq n}\lambda_{ij}\lambda_{sj}\Big)^2+\widetilde O\Big(\frac{n^3p^2}{d^{3/2}}\Big)\leq 2\bE Y+ \widetilde O\Big(\frac{n}{d}\cdot np^2\Big)+\widetilde O\Big(\frac{n^3p^2}{d^{3/2}}\Big),\]
where we have used property~\ref{item:degree_cancellation} of the event $\cE$ to bound $\big|\sum_{s<j\leq n}\lambda_{ij}\lambda_{sj}\big|\leq \widetilde O(\sqrt{n}p)$. Additionally, we have $2\bE Y=d\sum_{s<i<j\leq n}\lambda_{sij}^2\bE[\bu_i(s)^2\bu_j(s)^2]\leq \widetilde O(d\cdot \sum_{s<i<j\leq n}\lambda_{ij}^2/d^2)=\widetilde O(n^2p/d)$, by property~\ref{item:degree_concentration} of the event $\cE$. Hence, we get that $\Var(X)\leq \widetilde O(n^2p/d+n^3p^2/d^{3/2})\ll 1$, since $d\gg n^2p$. In particular, we have $\bE[e^{2(X-\bE X)}]=\exp(2\Var(X)+O(\delta_1\delta_2))\leq 2$.
\end{proof}

Applying Lemma~\ref{lemma:tilting} to $X$ and $Y$, and taking the bounds on $\Var(Y)$ and $\bE[Y^4]$ from Claim~\ref{claim:Y_bounds}, we get
\begin{align*}
    \bE[e^{S}]&=\bE[e^{X-Y}]= \bE[e^X]\exp\Big(-\bE[Y]+\widetilde O\Big(\frac{n^{3/2}p}{d}+\frac{n^4p^2}{d^2}\Big)\Big).
\end{align*}
From Lemma~\ref{lemma:MGF_X}, we have $\bE[e^X]=\exp(\bE X+\Var(X)/2+\widetilde O(n^{5/2}p^{3/2}/d^{3/2}+n^{7/2}p^{3/2}/d^2))$. Finally, by Lemmas~\ref{lemma:EX_formula} and \ref{lemma:final_computation}, we can put all of the pieces together to get
\begin{align*}
    \log \bE[e^{S}]&=\bE X+\frac{1}{2}\Var(X)-\bE[Y]+\widetilde O\Big(\frac{n^{3/2}p}{d}+\frac{n^4p^2}{d^2}\Big)+\widetilde O\Big(\frac{n^{5/2}p^{3/2}}{d^{3/2}}+\frac{n^{7/2}p^{3/2}}{d^2}\Big)\\
    &=A_s+\frac{1}{\sqrt{d}}\sum_{s<i<j\leq n}\lambda_{si}\lambda_{sj}\lambda_{ij}+\frac{1}{d}\sum_{\substack{s<r\leq n\\s<i<j\leq n}}\lambda_{si}\lambda_{sj}\lambda_{ri}\lambda_{rj}\lambda_{ij}(z_p-\lambda_{ij})\\
    &\quad-\frac{1}{2d}\sum_{s<i<j\leq n}\lambda_{si}^2\lambda_{sj}^2\lambda_{ij}^2+\frac{1}{d}\sum_{\substack{s<r\leq n\\s<i<j\leq n}}\!\!\!\!\lambda_{si}\lambda_{sj}\lambda_{ri}\lambda_{rj}(1+\lambda_{sr}(z_p-\lambda_{sr}))+\widetilde O\Big(\frac{n^{3/2}p}{d}+\frac{n^{3}p^{2}}{d^{3/2}}+\frac{n^{4}p^{2}}{d^2}\Big),
\end{align*}
where we have aggregated the error terms by using the assumption $p\geq n^{-1}$. This completes the proof of Lemma~\ref{lemma:S_bound}.
\end{proof}

\appendix
\section{Proof of Preliminaries}

\subsection{Deferred proofs from Section~\ref{subsec:gaussian_and_sphere}}

\begin{proof}[Proof of Lemma~\ref{lemma:tail-expansion}.]
The logarithmic derivative of $\Phi$ is $(\log \Phi)'(t)=\frac{\phi(t)}{\Phi(t)}=\lambda(t)$. Since $\lambda'(t)\leq 0$ from Lemma~\ref{lemma:mills_ratio}, we find that $\Phi$ is log-concave, i.e. that for any two $\alpha, \beta\in \bR$ we have
\[\log \Phi (\alpha+\beta)\leq \log \Phi(\alpha)+\beta (\log \Phi)'(\alpha)=\log \Phi(\alpha)+\beta \lambda(\alpha).\]
Exponentiating this relation gives precisely~(\ref{eqn:log_concavity}). On the other hand, using the second-order Taylor expansion combined with the fact that $(\log \Phi)'''(t)=\lambda''(t)\in [0, 2]$, we find that
\[\log \Phi (\alpha+\beta)\leq \log \Phi(\alpha)+\beta (\log \Phi)'(\alpha)+\frac{\beta^2}{2} (\log \Phi)''(\alpha)+\frac{|\beta|^3}{6}\cdot 2\leq \log \Phi(\alpha)+\beta \lambda(\alpha)-\lambda(\alpha)(\alpha+\lambda(\alpha))\frac{\beta^2}{2}+|\beta|^3.\]
Exponentiating this relation gives~(\ref{eqn:second_order_taylor}).
\end{proof}

\begin{proof}[Proof of Lemma~\ref{lemma:mills_ratio}.]
Since both $\phi$ and $\Phi$ are analytic and nonvanishing on $\bR$, so is $\lambda$. Moreover, the statements $-1\leq \lambda'(t)\leq 0$ (for all $t$) and $|t|\leq \lambda(t)\leq |t|+1$ for $t\leq 0$ are standard facts, proven in\footnote{We note that in \cite{Sampford53}, they use slightly different notation. There, they write ``$\nu(x)$'' to denote what is $-\lambda(x)$ in our notation, and ``$\lambda(x)=\nu'(x)$'' which is $-\lambda(x)$ for us.} \cite[Eq.~(3)]{Sampford53} and \cite{Birnbaum42}, respectively. The first derivative is simple to compute
\[\lambda'(t)=\frac{\phi'(t)\Phi(t)-\Phi'(t)\phi(t)}{\Phi(t)^2}=\frac{-t\phi(t)\Phi(t)-\phi(t)^2}{\Phi(t)^2}=-\lambda(t)(t+\lambda(t)).\]
This leaves us with the task of proving $0\leq \lambda''(t)\leq 2$. The lower bound was established by \cite{Sampford53}, so we only focus on the upper one. Note that $\lambda''(t)=-\lambda'(t)(t+\lambda(t))-\lambda(t)(1+\lambda'(t))\leq -\lambda'(t)(t+\lambda(t))$,
since $\lambda(t)\geq 0$ and $|\lambda'(t)|\leq 1$ guarantee that the second term is negative. If $t\leq 0$, we have $|\lambda'(t)|\leq 1$ and $t+\lambda(t)\leq 1$, implying $\lambda''(t)\leq 1$.

If $t\geq 0$, we have $\lambda(t)=\frac{\phi(t)}{\Phi(t)}\leq 2\phi(t)\leq 1$ and so $\lambda''(t)\leq 2\phi(t)(t+1)^2$. A simple exercise in calculus shows that $\max_{t\geq 0} \phi(t)(t+1)^2\leq 1$. To see this, note that the only critical points of $f(t)=\phi(t)(t+1)^2$ are $t=-2, -1, 1$, and furthermore its maximum is attained at $t=1$, equaling $f(1)=4/\sqrt{2\pi e}\leq 1$.
\end{proof}

\begin{proof}[Proof of Lemma~\ref{lemma:expectation_truncated_gaussian}.]
By rescaling, we may assume $\sigma=1$. Now let $Z\sim \cN(0,1)$ be a standard Gaussian. We record the relations $\phi'(z) = -z e^{-z^2/2}/\sqrt{2\pi} = -z\phi(z)$ and $\big(\Phi(z)-z\phi(z)\big)' = z^2\phi(z)$. Thus, for every $t\in \bR$, we get
$\int_{-\infty}^t z\phi(z) \, dz = -\phi(t)$ and $\int_{-\infty}^t z^2\phi(z) \, dz = \Phi(t)-t\phi(t)$.
Consequently,
{\smaller[0.3]
\begin{align*}
\bE\big[X_t^- \big] = \frac{\int_{-\infty}^{t} z\phi(z) \, dz}{\Phi(t)} = \frac{-\phi(t)}{\Phi(t)} = -\lambda(t),\quad \text{ and }\quad
\bE\big[X_t^+ \big] = \frac{\int_t^{\infty} z\phi(z) \,dz}{1-\Phi(t)} = \frac{\phi(t)}{\Phi(-t)} = \lambda(-t).
\end{align*}}
We can compute the second moments in a similar fashion:
{\smaller[0.3]
\begin{align*}
\bE\big[(X_t^-)^2\big] &= \frac{\int_{-\infty}^{t} z^2\phi(z) \, dz}{\Phi(t)} = \frac{\Phi(t)-t\phi(t)}{\Phi(t)}= 1-t\lambda(t),\\
\bE\big[(X_t^+)^2\big] &= \frac{\int_t^{\infty} z^2\phi(z) \,dz}{1-\Phi(t)} = \frac{\Phi(-t)+t\phi(t)}{\Phi(-t)}= 1+t\lambda(-t).
\end{align*}}
Given the two above relations, the variances of $X_t^+$ and $X_t^-$ can be computed from $\Var(X_t^+)=\bE[(X_t^+)^2]-\bE[X_t^+]^2$ and $\Var(X_t^-)=\bE[(X_t^-)^2]-\bE[X_t^-]^2$.
\end{proof}

\begin{proof}[Proof of Lemma~\ref{lemma:sphere_to_gaussian_CDF}.]
The first coordinate of a uniform point on $\mathbb S^{d-1}$ has the probability density function $c_d(1-t^2)^{(d-3)/2}\mathbf 1_{\{|t|< 1\}}$, where $c_d=\Gamma(d/2)/\sqrt{\pi}\Gamma((d-1)/2)$.
Scaling by $\sqrt{d}$ gives the law $\cS_d$ with density function
\[ \psi_d(x)=\frac{1}{\sqrt d}c_d\Big(1-\frac{x^2}{d}\Big)^{(d-3)/2}\mathbf 1_{\{|x|< \sqrt d\}}. \]

We first compare the prefactor with the Gaussian density. By the standard ratio estimate for Gamma functions, ${\Gamma(d/2)}/{\Gamma((d-1)/2)} = \big(1+ O(1/d)\big)\sqrt{\frac{d-1}{2}}$, and hence $c_d/\sqrt{d} = \big(1+ O(1/d)\big)/\sqrt{2\pi}$.

Next, for $|x|< \sqrt{d}$, we use $1-x^2/d\leq e^{-x^2/d}$ to get $\Big(1-\frac{x^2}{d}\Big)^{(d-3)/2} \leq \exp\Big(-\frac{d-3}{2}\frac{x^2}{d}\Big) = e^{-x^2/2}\exp\Big(\frac{3x^2}{2d}\Big)$.
Combining the last two estimates, we obtain $\psi_d(x) \leq \frac{1}{\sqrt{2\pi}}e^{-x^2/2}\exp\Big(O\Big(\frac1d\Big)+\frac{3x^2}{2d}\Big) =\phi(x)\exp\Big(O\Big(\frac{1+x^2}{d}\Big)\Big)$. 
 Similarly, using that $1-u\ge e^{-u-{u^2}/{2(1-u)}}$ we get that $\psi_d(x)\ge \frac{e^{-x^2/2}}{\sqrt{2\pi}}\exp\big(-O\big(\frac{1+x^4}{d}\big)\big)$ for $|x|<(1/2)\sqrt{d}$, as then $\frac{d-3}{2}\frac{(x^2/d)^2}{2(1-x^2/d)}\le \frac{d-3}{2}(x^2/d)^2\le x^4/d$.

We now quickly get the $|\psi_d(t)-\phi(t)|\le O(\frac{1+t^4}{d})\phi(t)$. Indeed, if $|t|>\sqrt{d}$, we have the trivial upper bound $\phi(t)$. Otherwise, we have $t^2\le d$, and thus have that $\psi_d(t)\le \exp(\frac{3t^2}{2d})\phi(t)\le (1+O(t^2/d))\phi(t)$, which handles when $\psi_d>\phi$. Lastly, when $\phi>\psi_d$, the error is at most $\phi(t)$, so we may in fact assume $|t|<d^{1/4}\le \sqrt{d}/2$. Thus, using the general bound $\exp(-x)\ge 1-x$, we $\phi(t)-\psi_d(t)\le O(\frac{1+t^4}{d})\phi(t)$ for $|t|<\sqrt{d}/2$.

\hide{Assuming $d\ge 4$, write $\alpha = \sqrt{1-3/d}\ge 1-3/d$, and note $1/\alpha \le 1+100/d =\exp(O(1/d))$. We explicitly showed that $g_d(y)\le \exp(O(1/d))\phi(\alpha y)$ for all $y$, and thus have for $|x|\le \sqrt{d}$ that
\[\Pb[Y_d\le x] = \int_{-\infty}^{x} g_d(y)\,dy \le\exp(O(1/d)) \frac{1}{\alpha}\int_{-\infty}^{\alpha x} \phi(y)\,dy =\exp(O(1/d))\Phi(\alpha x)\le  \exp(O(1/d))(\Phi(x) + (1-\alpha)|x|\phi(\alpha x))\]\[ = \Phi(x)(1+O(1/d))+O(|x|\phi(\alpha x)/d).\]This gives the additive error version, as $\sup_x \{\Phi(x)+|x|\phi(\alpha x)\} = O(1)$. By symmetry, we have $\Pb[Y_d\ge x] = \Pb[Y_d\le-x] \le \Phi(-x)+O(1/d)$, and since $\Pb[Y_d\le x]+\Pb[ Y_d\ge x]=1$ it follows that $\Psi_d(x)\ge \Phi(x)-O(1/d)$ also holds.}

We now record that \[|\Psi_d(t)-\Phi(t)|\le \int_{-\infty}^t |\psi_d(u)-\phi(u)|\,du\le O\Big(\frac{1}{d}\Big)\int_{-\infty}^t(1+u^4)\phi(u)\,du,\]meaning we are left to show $I(t):=\int_{-\infty}^t(1+u^4)\phi(u)\,du\le O(1+t^4)\Phi(t)$. We now fix the large constant $C=1000$. If $t\ge -C$, then $\Phi(t)\ge \Phi(-C)= \Omega(1)$, while $I(t)\le \int_{-\infty}^\infty (1+u^4)\phi(u)\,du = O(1) = O(\Phi(t))$.

It remains to handle when $t\le -C$. For any such $t$, the change of variables $w=2u$ gives \[I(2t)-I(4t) = 2\int_{2t}^{t}(1+(2w)^4)\phi(2w)\,dw\le (I(t)-I(2t))\max_{w\in [t,2t]}\Big\{\frac{2(1+(2w)^4)\exp(-2w^2)}{(1+w^4)\exp(-w^2/2)}\Big\}\]\[\le (I(t)-I(2t))2^5\exp(-(3/2)t^2)\le (I(t)-I(2t))/2. \]Whence, we get $I(t)= \sum_{i=0}^\infty (I(2^it)-I(2^{i+1}t))\le (I(t)-I(2t))\sum_{i=0}^{\infty}2^{-i}\le 2\int_{2t}^t(1+u^4)\phi(u)\,du\le (2+2^5t^4)\int_{2t}^t\phi(u)\,du\le O((1+t^4)\Phi(t))$.
\end{proof}

\begin{proof}[Proof of Lemma~\ref{lemma:quantile_asymptotics}]
Observe that $z_p\to \infty$ as $p\to 0$. From Lemma~\ref{lemma:mills_ratio}, we have $z_p\leq \frac{\phi(-z_p)}{\Phi(-z_p)}$, which can be rewritten as $\phi(-z_p)\geq z_p\Phi(-z_p)=pz_p$. So, $e^{-z_p^2/2}\geq \sqrt{2\pi}pz_p$, and by taking logarithms and square roots, we arrive at $z_p^2\leq 2\log p^{-1}-\log(2\pi)-\log z_p$, showing that $z_p$ is at most logarithmic in $p^{-1}$. So, $z_p^2\leq (2+o(1))\log p^{-1}$. On the other hand, $\phi(z_p)\leq (z_p+1)\Phi(-z_p)=(z_p+1)p$, thus giving $z_p^2\geq 2\log p^{-1}-\log(2\pi)-\log (z_p+1)$. Recalling that $z_p\leq O(\log p^{-1})$ gives $z_p=(1+o(1))\sqrt{2\log p^{-1}}$.

To prove the second statement, note that $c_{p, k}\leq \log k$, as otherwise Lemma~\ref{lemma:sphere_to_gaussian_CDF} would give $\Psi_k(-c_{p, k})=\Psi_k(-\log k)\leq O(\Phi(-\log k))\ll 1/k\leq p$, contradicting the definition of $c_{p, k}$. Hence, by invoking Lemma~\ref{lemma:sphere_to_gaussian_CDF} once more, we have $p=\Psi_{k}(-c_{p, k})=\big(1+\widetilde O(1/k)\big)\Phi(-c_{p, k})$,
and so $\Phi(-c_{p, k})=p(1+\widetilde O(1/k))$. We conclude that that \[|\log \Phi(-c_{p, k})-\log \Phi(-z_p)|=|\log p+\log (1+\widetilde O(1/k))-\log p|=\widetilde O(1/k).\]
Consider now the function $f(t)=\log \Phi(t)$ and note that for $t\geq 0$ its derivative satisfies $f'(-t)=\lambda(-t)\geq t$, due to Lemma~\ref{lemma:mills_ratio}. In particular, on the interval $[-2z_p, -z_p/2]$, we have $f'(t)\geq \widetilde\Omega(1)$. This shows that if $|f(-c_{p, k})-f(-z_p)|\leq \widetilde O(1/k)$, then we have $c_{p, k}=z_p+\widetilde O(1/k)$.
\end{proof}

\begin{proof}[Proof of Lemma~\ref{lemma:sphere_coupling_moments}.]
If we define $U=\sqrt{m}X$, $V=\sqrt{m} Y$, then it suffices to prove that $|\bE[U^k]-\bE[V^k]|\leq O_k\big((1+|t|^{k+4})/m\big)$ for $k=1,2$. The moments of $U, V$ can be expressed as follows
\[\bE[U^k]=\frac{I_1}{1-\Psi_m(t)}, \quad \bE[V^k]=\frac{I_2}{1-\Phi(t)}, \quad \text{ where }\quad I_1=\int_t^{\infty}u^k \psi_m(u)\,du,\quad  I_2=\int_t^\infty u^k\phi(u)\,du.\]
By Lemma~\ref{lemma:sphere_to_gaussian_CDF}, we have $1-\Psi_m(t)=\Psi_m(-t)=\Phi(-t)\big(1\pm O\big(\frac{1+t^4}{m}\big)\big)=(1-\Phi(t))\big(1\pm O\big(\frac{1+t^4}{m}\big)\big)$. Assuming $|t|<m^{1/8}$, this error is bounded away from $1$, whence $\frac{1}{1-\Psi_m(t)}= \frac{1}{1-\Phi(t)}\Big(1+O\Big(\frac{1+t^4}{m}\Big)\Big)$. Meanwhile, for all $u\in \bR$ we have that $|\phi(u)-\psi_m(u)|\le O\big(\frac{1+u^4}{m}\phi(u)\big)$. Thus,
\[\Big|\int_t^{\infty}u^k \psi_m(u)\,du-\int_t^\infty u^k\phi(u)\,du\Big|\leq \int_t^\infty |u|^k\phi(u)O\Big(\frac{1+u^4}{m}\Big)\,du.\]Suppose that we also have $E_1:=\int_t^{\infty}|u|^k(1+u^4)\phi(u)\,du\leq O_k\big((1+|t|^{k+4})(1-\Phi(t))\big)$.

Then, noting $|I_1-I_2|\le E_1/m$, we can finish the proof as follows
\begin{align*}
\big|\bE[U^k]-\bE[V^k]\big|=\Big|\frac{I_1}{1-\Phi(t)}\Big(1+O\Big(\frac{1+t^4}{m}\Big)\Big)-\frac{I_2}{1-\Phi(t)}\Big|&\leq \Big|\frac{I_2+E_1/m}{1-\Phi(t)}\Big(1+O\Big(\frac{1+t^4}{m}\Big)\Big)- \frac{I_2}{1-\Phi(t)}\Big|\\
&\leq \Big| O\Big(\frac{1+|t|^{k+4}}{m}\Big)+O\Big(\frac{1+t^4}{m}\Big)\frac{I_2}{1-\Phi(t)}\Big|.
\end{align*} Note that we used the assumption $|t|<m^{1/8}$ in the last line so that $\Big(1+O\Big(\frac{1+t^4}{m}\Big)\Big)\frac{E_1}{1-\Phi(t)} = O\Big(\frac{E_1}{1-\Phi(t)}\Big)$.
Note that the first error term was precisely what we were looking to get, while the second one equals $O\big((1+t^4)|\bE[V^k]|/m\big)$. Thus, it now suffices to show that $E_2:=\int_t^\infty (1+t^4)|u|^k\phi(u)\,du = O((1+|t|^{k+4})(1-\Phi(t))$. If $t<0$, then $1-\Phi(t)\ge 1/2$, and then $E_2/(1-\Phi(t))\le 2(1+t^4)\int_{-\infty}^\infty |u|^k\phi(u)\,du=O_k(1+t^4)$. Otherwise, by a pointwise comparison of integrands, we have that $E_2< E_1 =O_k((1+|t|^{k+4})(1-\Phi(t))$, giving what we desire (modulo our bound on $E_1$).

It remains to justify that $E_1=\int_t^\infty |u|^k\phi(u)(1+u^4)\,du\leq O_k\big((1+|t|^{k+4})(1-\Phi(t))\big)$, which amounts to proving that $\bE\big[|V|^k\big], \bE\big[|V|^{k+4}\big]\leq O_k\big(1+|t|^{k+4}\big)$. For $t\leq 0$, this is clear since $\Phi(t)\leq 1/2$ implies that $\bE\big[|V|^k\big], \bE\big[|V|^{k+4}\big]\leq O_k\big(\int_{-\infty}^{\infty}(1+|u|^{k+4})\phi(u)du\big)\leq O(1)$. In fact, the same argument works for any $t=O_k(1)$. For any $t\geq C_k$, we have that $2^{k}t^{k}(1-\Phi(t))\ge \int_t^{2t}|u|^k\phi(u)\,du\ge \frac{e^{(3/2)t^2}}{2\cdot 4^k}\int_{2t}^{4t}|u|^k\phi(u)\,du\ge 2\int_{2t}^{4t}|u|^k\phi(u)\,du$. Whence summing a geometric series we get $\int_t^\infty|u|^k\phi(u)\,du\le \sum_{i=0}^\infty 2^{-i}\cdot \int_t^{2t} |u|^k\phi(u)\,du\le 2^{k+1}t^k(1-\Phi(t)) =O_k(t^k(1-\Phi(t))$ as desired. Clearly the same analysis shows $\int_t^\infty |u|^{k+4}\phi(u)\,du = O_k(t^{k+4}(1-\Phi(t)))$, as required. This completes the proof.
\end{proof}

\subsection{Deferred proofs from Section~\ref{subsec:events}}

\begin{proof}[Proof of Lemma~\ref{lemma:geometric_graph_event_E}.]
It suffices to prove each of the four properties holds for $G\sim G(n, d, p)$ with high probability. Thus, we tackle one property at the time.
\medskip

\noindent\textbf{Proof of \ref{item:degree_concentration}.}
When $s_k i$ is an edge, then $|A_{s_ji}-p|= 1-p$, and otherwise we have $|A_{s_ji}-p|=p$. Thus, we can write
\[\sum_{1\leq i\leq n} |A_{s_1i}-p|\cdots |A_{s_ki}-p|=\sum_{J\subseteq [k]} (1-p)^{|J|}p^{k-|J|}\Big|\bigcap_{i\in J} N(s_i)\cap \bigcap_{i\notin J} N(s_i)^c\Big|.
\]
Hence, to control the left-hand side, it suffices to control the size of the intersection of neighborhoods and their complements. To do this, we will actually show a more general statement, so that we can also use it to prove \ref{item:degree_cancellation}.

\medskip

We will show that $G$ satisfies the following property with high probability: for any distinct $s_1, \dots, s_m$, $t_1, \dots, t_\ell\in [n]$ with $m+\ell\leq 2$ and any interval $I\subseteq [n]$, we have \begin{equation}\label{eqn:neighborhood_intersection}
    \Big|\bigcap_{i=1}^m N(s_i)\cap \bigcap_{j=1}^\ell N(t_j)^c\cap I\Big|=p^m(1-p)^\ell |I|+O\Big(\sqrt{p^m(1-p)^\ell n \log n}\Big).
\end{equation}
In particular, if we fix $s_1, \dots, s_m$ and $t_1, \dots, t_\ell$ and $I$, it suffices to show that (\ref{eqn:neighborhood_intersection}) holds with probability at least $1-4n^{-5}$, since one can then take the union bound over at most $n^{m+\ell} \cdot n^2\leq n^{4}$ choices of $s_1, \dots, s_m, t_1, \dots, t_\ell$ and $I$.

To see this, we proceed in two steps. Firstly, if we denote by $\bu_1, \dots, \bu_m$ the vectors corresponding to $s_1, \dots, s_m$ and by $\bw_1, \dots, \bw_\ell$ the vectors corresponding to $t_1, \dots, t_\ell$, we estimate the probability that a randomly sampled point $\bx\in \bS^d$ belongs to $\bigcap_{i=1}^m N(s_i)\cap \bigcap_{i=1}^\ell N(t_j)^c$, i.e. satisfies $\langle \bu_i, x\rangle \geq c_{p, d}/\sqrt{d}$ for all $i\in [m]$ and $\langle \bw_j, x\rangle < c_{p, d}/\sqrt{d}$ for all $j\in [\ell]$. A very convenient tool for doing this was developed in \cite{LMSY22} (Lemma 5.1). Here, we give its simplified statement, in the regime where $m$ and $\ell$ are constant, and $p\geq n^{-1/2}$.

\begin{lemma}
Let $\bu_1, \dots, \bu_m, \bw_1, \dots, \bw_\ell$ be sampled uniformly at random from $\bS^{d-1}$, and let \[L=\{x:\langle \bu_i, x\rangle \geq c_{p, d}/\sqrt{d} \text{ for all } i\in [m] \text{ and } \langle \bw_j, x\rangle < c_{p, d}/\sqrt{d} \text{ for all } j\in [\ell]\}.\] If $\mu=\Pb[\bx\in L]$ is the probability that a uniformly sampled point $\bx\in \bS^{d-1}$ lands in $L$, then there is an absolute constant $c>0$ such that for every $\eps>0$ we have
\[\Pb\Big[\big|\frac{\mu}{p^m(1-p)^\ell}-1\big|\geq \eps \Big]\leq 2\exp\left(-\frac{cd\eps^2}{(\log d)^2\log p^{-1}}\right)+O\bigg(p^{-2}\exp\Big(-\frac{cd}{(\log d)^2\log p^{-1}}\Big)\bigg).\]
\end{lemma}

Note that the probability in the above lemma is taken over the choice of $\bu_1, \dots, \bu_m, \bw_1, \dots, \bw_\ell$. The second term of the above equation is much smaller than $n^{-5}$, since $d\geq n^2p(\log n)^A\gg n$. Next, setting $\eps={1}/{\sqrt{p^m (1-p)^\ell n\log n}}$, the first term of the above inequality simplifies to
\[\Pb\Big[\big|\frac{\mu}{p^m(1-p)^\ell}-1\big|\geq \eps \Big]\leq 2\exp\left(-\frac{c'd}{p^m(1-p)^\ell n(\log n)^3\log p^{-1}}\right)+O\Big(p^{-2}e^{-{c'n}/{(\log n)^3}}\Big).\]
Since $d\geq Cn^2p(\log n)^A$ for a large constant $C$, we have that the first exponential is also bounded by $n^{-5}$, showing that with probability at least $1-2n^{-5}$, we have $\mu=p^m(1-p)^\ell+O\big(\sqrt{p^m(1-p)^\ell/(n\log n)}\big)$. 

Let us now pass to the second step of our argument. Observe that for fixed $s_1, \dots, t_\ell$ and $I$, the size $X=\Big|\bigcap_{i=1}^m N(s_i)\cap \bigcap_{i=1}^\ell N(t_j)^c\cap I\Big|$ is distributed as a Binomial random variable $B(|I\setminus \{s_1,\dots,t_\ell\}|, \mu)$, and therefore it has expectation $\bE[X]=\mu |I|\pm O(1)$. By Chernoff bound (see e.g. Theorem 2.1 in \cite{JLR}), we have $\Pb[|X-\bE X|\geq t]\leq 2\exp\big(-t^2 /2(\bE X+t/3)\big)$. For $t=10\sqrt{p^m(1-p)^\ell n \log n}$, it is not hard to check that $t^2/2(\bE X+t/3)\geq 5\log n$ - indeed, if $t\leq \bE X$, then $\frac{100p^m(1-p)^\ell n\log n}{4 \mu |I|}\geq 5\log n$ and if $t\geq \bE X$ we have $t/4\geq 5 p^{m/2} \sqrt{n\log n}\geq 5\log n$, since $p\geq \sqrt{\log n/n
}$ and $m\leq 2$. All in all, we find that \[\Pb\Big[|X-\bE X|\geq 10\sqrt{p^m(1-p)^\ell n \log n}\Big]\leq 2\exp\big(-5\log n\big)\leq 2n^{-5}.\]

Since $\bE X=\mu |I|\pm O(1)=p^m(1-p)^\ell |I|+O\big(\sqrt{p^m(1-p)^\ell n \log n}\big)$, we conclude that with probability at least $1-4n^{-5}$ the estimate (\ref{eqn:neighborhood_intersection}) holds. Hence, a straightforward union bound shows that with high probability all $(m+\ell)$-tuples $s_1, \dots, t_\ell$ and all intervals $I$ satisfy this estimate.

\medskip

We now deduce \ref{item:degree_concentration} from (\ref{eqn:neighborhood_intersection}) deterministically. Fix $1\leq k\leq 2$ and distinct vertices $s_1,\dots,s_k\in [n]$. For every $J\subseteq [k]$, let $X_J=\bigcap_{j\in J}N(s_j)\cap\bigcap_{j\notin J} N^c(s_j)$. Applying (\ref{eqn:neighborhood_intersection}) with $I=[n]$ gives $|X_J|=O(p^{|J|}(1-p)^{k-|J|}n)$. So,
\[\sum_{i=1}^n |A_{s_1i}-p|\cdots |A_{s_ki}-p|=\sum_{J\subseteq [k]} |X_J| p^{k-|J|}(1-p)^{|J|}\leq \sum_{J\subseteq [k]} O(p^k n)\le O(p^k n), \]
where the last inequality follows since $k$ is a constant.

\medskip
\noindent\textbf{Proof of \ref{item:degree_cancellation}.}
We apply the same strategy as in \ref{item:degree_concentration}, but we now keep track of the error terms produced by (\ref{eqn:neighborhood_intersection}) more carefully. Namely, we set $I=[s+1, \dots, n]$ and note that every $i\in X_J$ now contributes $(1-p)^{|J|}p^{k-|J|}$ to the sum. Hence, we can write
\begin{align*}
\Big|\sum_{i\in I} (A_{s_1i}-p)\cdots (A_{s_ki}-p)\Big|&=\bigg|\sum_{J\subseteq [k]} |X_J|(1-p)^{|J|}(-p)^{k-|J|}\bigg|\\
&\leq \!\bigg|\!\sum_{J\subset [k]}\! (-1)^{k-|J|}|I| p^{k} (1-p)^{k}\!+ O\Big((1-p)^{|J|}p^{k-|J|}\sqrt{p^{|J|}(1-p)^{k-|J|} n \log n}\Big)\bigg|.
\end{align*}
Note that the first term cancels out when summed over all $J\subseteq[k]$ because of the binomial formula, and the second term is bounded by $O(p^{k/2} \sqrt{n\log n})$. This completes the proof of \ref{item:degree_cancellation}.

\medskip
\noindent\textbf{Proof of \ref{item:eigenvalue_bound}.}
This item is an immediate consequence of known facts about the eigenvalues of random geometric graphs. For example, Abdalla, Bandeira and Invernizzi \cite[Lemma~2]{ABI24} showed the following theorem (see also \cite{CZ25} and \cite{LMSY22} for similar bounds on the second eigenvalue of $G(n, d, p)$).

\begin{theorem}
Let $A$ be the adjacency matrix of a random graph $G$ sampled from $G(n,d,p)$. There is an absolute constant $C>0$ such that if $np>(\log n)^2$, and $d\geq C (n^2p^2+\log^4n)\log^4n$, then with probability at least $1-n^{-\Omega(1)}$,
\[ \|A-p J\|_{op}\leq O\left(\sqrt{np}\right).\]
\end{theorem}

\medskip
\noindent\textbf{Proof of \ref{item:triangle_concentration}.}
Let us rewrite the sum as follows
\begin{align*}
\sum_{s<i<j\leq n}\big|(A_{si}-p)(A_{sj}-p)(A_{ij}-p)\big|
&=\sum_{s<i\leq n}|A_{si}-p|\Big(\sum_{i<j\leq n}|A_{ij}-p||A_{sj}-p|\Big)\\
&\leq \sum_{s<i\leq n}|A_{si}-p|\cdot \max_{s<i\leq n}\sum_{i<j\leq n}|A_{ij}-p||A_{sj}-p|.
\end{align*}
By property \ref{item:degree_concentration}, we have that the latter maximum is bounded by $O(p^2n)$, and the sum over $i$ is bounded by $O(pn)$. By multiplying these two quantities, we get $\sum_{s<i<j\leq n}\big|(A_{si}-p)(A_{sj}-p)(A_{ij}-p)\big|\leq O(p^3 n^2)$.
\end{proof}

\begin{proof}[Proof of Lemma~\ref{lemma:vector_event_F}.]
Even though the vectors $\bv_1, \dots, \bv_n$ come from a special distribution, described by Lemma~\ref{lemma:distribution_v_i}, we note that the events $\cT_s$ and $\Pi_s$ depend only on the first $s$ coordinates of the vectors $\bv_{s+1}, \dots, \bv_n$, and the distribution of these is the same as if they were uniformly sampled on the sphere (see Lemma~\ref{lemma:distribution_v_i}). Hence, for $s<i\leq n$, we will sample the vector $\bv_i$ on the sphere by sampling a $d$-dimensional Gaussian vector $\bg_i\sim N(0, I_d)$ and setting $\bv_i(j)=\bg_i(j)/\|\bg_i\|$ for $j=1, \dots, i-1$ and $\bv_i(i)=\sqrt{1-\sum_{k=1}^{i-1} \bv_i(k)^2}$. Further, we prove all the relevant bounds with $\log n$ instead of $\log d$, noting that this is stronger since $d\geq n$.

An important ingredient of the proof will be the $\chi_k^2$ concentration inequality of Laurent and Massart \cite{LM00}. Recall that the $\chi_k^2$ distribution is defined as the sum of squares of $k$ independent standard Gaussian random variables. This inequality states that if $Y\sim\chi^2_k$ and $t\ge0$, then
\[ \Pb\big[Y-k\ge 2\sqrt{kt}+2t\big]\le e^{-t} \quad\text{and}\quad \Pb\big[d-Y\ge 2\sqrt{kt}\big]\le e^{-t}. \]

By the inequality of Laurent and Massart mentioned above, setting $k=d$ and $t=3\log n$, we find $\Pb\big[\big|\|\bg_i\|^2-d\big|\leq 2\sqrt{3d\log n}+3\log n\big]\leq 2e^{-3\log n}\leq 2/n^{3}$. Hence, with high probability, for all $i\in [n]$ simultaneously, we have $\big|\|\bg_i\|^2-d\big|\leq 4\sqrt{d\log n}$. We will assume that this holds throughout the rest of the proof.

\medskip
\medskip
\noindent\textbf{Proof of \ref{item:entry_size}.}
For $1\leq s<i\leq n$, we have $\Pb\big[|\bg_i(s)|\geq 3\sqrt{\log n}\big]\leq e^{-9(\log n)/2}\leq \frac{1}{n^4}$. Since $v_i(s)=\bg_i(s)/\|\bg_i\|$, we can write
\[
\Pb\Big[|\bv_i(s)|\geq 10\sqrt{\frac{\log n}{d}}\Big]\leq \Pb\Big[|\bg_i(s)|\geq 3\sqrt{\log n}\Big]+\Pb\big[\|\bg_i\|\leq d/2\big]\leq \frac{1}{n^4}+\frac{1}{n^3}\leq O(n^{-3}),
\]
A simple union bound over at most $n^2$ pairs $(s, i)$ gives the desired bound on $|\bv_i(s)|$.

\medskip
\noindent\textbf{Proof of \ref{item:norm_concentration}.}
Fix $1\leq s<i\leq n$ and let us discuss the first inequality of $\Pi_s$. Then $\|\pi_s(\bg_i)\|_2^2\sim \chi_s^2$, so the chi-square inequality with $t=3\log n$ gives
\[\Pb\Big[\big|\|\pi_s(\bg_i)\|_2^2-s\big|\geq 2\sqrt{3s\log n}+3\log n\Big]\leq 2e^{-3\log n}\leq 2n^{-3}.\]
A union bound over $\leq n^2$ pairs $1\leq s<i\leq n$ shows that with high probability we have $\|\pi_s(\bg_i)\|_2^2\in (s-4\sqrt{n\log n}, s+4\sqrt{n\log n})$ for all $1\leq s<i\leq n$. If this holds for all $i$ and $s$, combined with $\|\bg_i\|^2\in (d-4\sqrt{d\log n}, d+4\sqrt{d\log n})$ we get
\[\Big|d\|\pi_s(\bg_i)\|_2^2-s\|\bg_i\|^2\Big|\leq \Big| 4d\sqrt{n\log n}+4s\sqrt{d\log n}\Big|\leq 8d\sqrt{n\log n}.\]
Dividing both sides by $d\|\bg_i\|^2$ and observing that $\|\bg_i\|^2\geq 8d/10$ gives:
\[\Big|\|\pi_s(\bv_i)\|_2^2-\frac{s}{d}\Big|\leq \frac{8d\sqrt{n\log n}}{d\|\bg_i\|^2}\leq 10\frac{\sqrt{n\log n}}{d}.\]

\medskip
Let us now discuss the second inequality of $\Pi_s$. Conditionally on $\bg_i$, the inner product $\langle \pi_s(\bg_i),\pi_s(\bg_j)\rangle=\sum_{r=1}^s \bg_i(r)\bg_j(r)$ is a centered Gaussian with variance $\|\pi_s(\bg_i)\|_2^2$. In the above paragraph, we have shown that with high probability, $\|\pi_s(\bg_i)\|_2^2\leq s+4\sqrt{n\log n}\leq 2n$ for all $i, s$. Hence, the standard Gaussian tail estimate gives
\[\Pb\big[|\langle \pi_s(\bg_i),\pi_s(\bg_j)\rangle|\geq 3\sqrt{2n\log n}\,\big|\, \bg_i\big]\leq 2\exp(-9(\log n)/2)\leq n^{-4}.\]
A union bound over the $O(n^3)$ triples $(s, i, j)$ shows that with high probability, $\big|\langle \pi_s(\bg_i),\pi_s(\bg_j)\rangle\big|\leq 5\sqrt{n\log n}$ for all $1\leq s<i<j\leq n$. Combined with $\|\bg_i\|^2\geq 9d/10$ and $\|\bg_j\|^2\geq 9d/10$, we get
\begin{align*}
\big|\langle \pi_s(\bv_i),\pi_s(\bv_j)\rangle\big|&=\Big|\frac{\langle \pi_s(\bg_i),\pi_s(\bg_j)\rangle}{\|\bg_i\|\|\bg_j\|}\Big|\leq \frac{5\sqrt{n\log n}}{(9/10)^2d}\leq 10\frac{\sqrt{n\log n}}{d}.\qedhere
\end{align*}
\end{proof}

\subsection{Deferred proofs from Section~\ref{sec:main_estimate}}\label{sec:geometric_graph_event_D}
\begin{proof}[Proof of Lemma~\ref{lemma:geometric_graph_event_D}.]
\noindent\textbf{Proof of \ref{item:triangle_cancellation}.}
The proof of item~\ref{item:triangle_cancellation} follows from a second moment argument. In order to show that $\Pb[|N_\triangle(G)|\leq K^{\cD} n^{3/2}p^{3/2}]\geq 1-o(1)$, as $K^{\cD}\to\infty$, it suffices to show that $0\leq \bE[N_\triangle(G)]\leq O(n^{3/2}p^{3/2})$ and $\Var[N_\triangle(G)]\leq O(n^3p^3)$.

The first inequality follows directly from the following result of Bok, Li and Yu \cite[Proposition~6.1]{BLY26}, which we state only for triangles.

\begin{theorem}\label{thm:bok-li-yu}
Let $p\leq 1/2$ and $d\geq (5\log (1/p))^4$. Then, for $G\sim G(3, d, p)$, we have
\[\bE_G\big[(A_{12}-p)(A_{23}-p)(A_{13}-p)\big]=\Theta\Big(\frac{p^3 (\log 1/p)^{3/2}}{\sqrt{d}}\Big),\]
where $A$ is the adjacency matrix of $G$ and the constants in the $\Theta$ notation are absolute.
\end{theorem}

When $d\geq Cn^3p^3(\log p^{-1})^3$, by linearity of expectation we have $\bE[N_\triangle(G)]=\binom{n}{3}\Theta(p^3(\log 1/p)^{3/2}/\sqrt{d})\leq n^{3/2}p^{3/2}$, if the constant $C$ is large enough. On the other hand, the variance of $N_\triangle(G)$ was computed in \cite{LMSY22} and \cite[Equations~(B.3)-(B.6)]{BLY26} (under the minimal assumption $d\gg \log(1/p)^4$). Namely, they have shown that if $\delta>0$ is chosen such that $\Pb[G(n, d, p)=\triangle]=p^3 (1+\delta)$, then $\Var[N_\triangle(G)]\leq O(n^3p^3(1+\delta)+p^5\delta^2n^4+p^6\delta n^4)$.

The precise asymptotics of $\delta$ can be derived from the Theorem~\ref{thm:bok-li-yu}. To do this, observe the random variables $A_{ij}, A_{ik}$ are pairwise independent when $j\neq k$, due to the independence of the latent vectors $\bv_j, \bv_k$. This implies $\bE[A_{ij}A_{ik}]=\bE[A_{ij}]\bE[A_{ik}]=p^2$. Thus, we can use the linearity of expectation to write
\begin{align*}
  \bE_G\big[(A_{12}\!-\!p)(A_{23}\!-\!p)(A_{13}\!-\!p)\big]\!&=\!\bE_G\big[A_{12}A_{23}A_{13}\!-\!p(A_{12}A_{23}\!+\!A_{12}A_{13}\!+\!A_{23}A_{13})\!+\!p^2(A_{12}\!+\!A_{23}\!+\!A_{13})\!-\!p^3\big]\\
  &=\bE_G[A_{12}A_{23}A_{13}]-3p^3+3p^3-p^3=\bE_G[A_{12}A_{23}A_{13}]-p^3.
\end{align*}
So, $\delta p^3=\Pb[G(n, d, p)=\triangle]-p^3=\Theta(p^3(\log 1/p)^{3/2}/\sqrt{d})$ which shows that $\delta=O((\log 1/p)^{3/2}/\sqrt{d})=O(1/n^{3/2}p^{3/2})$, due to our assumption on $d$. Hence, the formula for $\Var[N_\triangle(G)]$ from \cite{LMSY22} or \cite{BLY26} gives $\Var[N_\triangle(G)]\leq O(n^3p^3+p^2 n +p^{9/2}n^{5/2})$. Hence, $\Var[N_\triangle(G)]\leq O(n^3p^3)$ (since $p=\Omega(1/n^2)$), which completes the proof of item~\ref{item:triangle_cancellation}.

\medskip
\noindent\textbf{Proofs of \ref{item:c4_cancellation} and \ref{item:diamond_cancellation}.}
Our goal is now to show that for $G\sim G(n, d, p)$, we have
\[\big|N_{\square}(G)\big|\leq K^{\cD} n^2p^2\text{ and }\big|N_{\diamond}(G)\big|\leq K^{\cD} n^2p^{5/2}\]
with probability $1-o(1)$ as $K^{\cD}\to\infty$.

As before, we will rely on the second moment method. We will show that $\bE[N_{\square}(G)^2]\leq O(n^4p^4)$ and $\bE[N_{\diamond}(G)^2]\leq O(n^4p^5)$, since then a simple application of Markov's inequality suffices to complete the proof. Since both proofs follow the same strategy, we begin by outlining the strategy in general, and only then applying it to the exact subgraph counts. Let $F\in \{\square, \diamond\}$ be a fixed graph, and let $\cC$ to be the collection of the copies of $F$ in $\{1, \dots, n\}$. Then, we write $X=\sum_{C\in \cC}\prod_{e\in E(C)}(A_e-p)$, with the goal of bounding $\bE[X^2]$.

Computing $\bE[X^2]$ would be easy if we were working with an Erd\H{o}s-R\'enyi random graph, since in that case the random variables $A_{ij}$ would be independent. In our case, to estimate $\bE[X^2]$, we will use the following tool from \cite{BB24}. Here, we state only a special case of the main theorem of \cite{BB24} when $F$ is a constant-size graph, in order to simplify the presentation.

\begin{theorem}[{\cite[Theorem~1.1 with Proposition~1.2]{BB24}}]\label{thm:bangachev-bresler}
Suppose that $n^{-1+\gamma}\leq p\leq 1/2$ and $d\geq n^{\gamma}$ for some fixed $\gamma>0$, and suppose that $F$ is a connected graph of constant size. Then, for $G\sim G(n,d,p)$,
\[
\Big|\mathbb{E}\Big[\prod_{(ji)\in E(F)}(A_{ji}-p)\Big]\Big|\leq O\Big(p^{|E(F)|}\times \Big(\frac{(\log d)^{3/2}}{\sqrt d}\Big)^{\big\lceil \frac{|V(F)|-1}{2}\big\rceil}\Big).
\]
\end{theorem}

Since $X=\sum_{F\in\cC}\prod_{e\in E(F)}(A_e-p)$, we have that
\[\bE[X^2]=\sum_{F_1,F_2\in\cC}\bE\!\left[\prod_{e\in E(F_1)}(A_e-p)\prod_{f\in E(F_2)}(A_f-p)\right].\]

Let $F'$ be the graph obtained by overlapping $F_1$ and $F_2$, i.e. by taking the union of their vertex sets and edge sets. We split the above sum based on the overlap pattern $F'$ and analyze the contribution of each pattern separately.

For fixed copies $F_1,F_2\in \cC$, let $R=E(F_1)\cap E(F_2)$, $r=|R|$, $m=|E(F)|$ and $\Delta=E(F_1)\triangle E(F_2)$. The first observation is that the edges $R$ appear twice in the product $\prod_{e\in E(F_1)}(A_e-p)\prod_{f\in E(F_2)}(A_f-p)$, giving the square term $(A_e-p)^2$. We linearize this by writing $(A_e-p)^2=p(1-p)+(1-2p)(A_e-p)$. Then, if we denote by $S\subset R$ the set of edges where we pick out the term $(1-2p)(A_e-p)$, the product corresponding to the pair $(F_1, F_2)$ can be rewritten as
\[\bE\!\Big[\prod_{e\in E(F_1)}(A_e-p)\prod_{f\in E(F_2)}(A_f-p)\Big]=\sum_{S\subseteq R}(1-2p)^{|S|}\bigl(p(1-p)\bigr)^{r-|S|}\bE\Big[\prod_{e\in\Delta\cup S}(A_e-p)\Big].\]
Since $p\in (0, 1)$, we have $|1-p|,|1-2p|\leq 1$ and so \[\Big|\bE\!\Big[\prod_{e\in E(F_1)}(A_e-p)\prod_{f\in E(F_2)}(A_f-p)\Big]\Big|\le\sum_{S\subseteq R}p^{r-|S|}\Big|\bE\prod_{e\in\Delta\cup S}(A_e-p)\Big|.\]

Now, we can factorize the product corresponding to $\Delta\cup S$ into products corresponding to its connected components $C_1, \dots, C_t$ and apply Theorem~\ref{thm:bangachev-bresler} to obtain
\[\Big|\bE\prod_{e\in E(\Delta\cup S)}(A_e-p)\Big|= \Big|\prod_{i=1}^t \bE\Big[\prod_{e\in E(C_i)}(A_e-p)\Big]\Big|\leq O\Big(p^{e(\Delta\cup S)}\Big(\frac{(\log d)^{3/2}}{\sqrt d}\Big)^{\sum_{C}\lceil\frac{v(C)-1}{2}\rceil}\Big).\]
Let us denote $\rho(\Delta\cup S)=\sum_{i=1}^t\left\lceil\frac{v(C_i)-1}{2}\right\rceil$. Furthermore, $e(\Delta\cup S)=|\Delta|+|S|=2m-2r+|S|$, so we can write the bound
\begin{align*}
\Big|\bE\Big[\prod_{e\in E(F_1)}\!\!(A_e-p)\!\!\prod_{f\in E(F_2)}\!\!(A_f-p)\Big]\Big|&\le O\bigg(\sum_{S\subseteq R}p^{r-|S|+2m-2r+|S|}\Big(\frac{(\log d)^{3/2}}{\sqrt d}\Big)^{\rho(\Delta\cup S)}\bigg)\\
&=O\bigg(\sum_{S\subseteq R}p^{2m-r}\Big(\frac{(\log d)^{3/2}}{\sqrt d}\Big)^{\rho(\Delta\cup S)}\bigg).
\end{align*}Let us note that $\rho(\Delta\cup S)=\big(|V(F_1\cup F_2)|-|\{i\in [t]:|V(C_i)|\text{ is odd}\}|\big)/2$ is minimized when $S=\emptyset$, as deleting edges can never decrease the number of odd-sized components.

Having obtained this general estimate, we can specialize to the cases where $F=\square$ or $F=\diamond$. Let us start from $F=\square$.
Note that two four-cycles may intersect in $r=0, 1, 2,$ or $4$ edges, and we will analyze each of these separately. To shorten notation, write $\eta=(\log d)^{3/2}/\sqrt d$, and note the crude bound $\eta \le d^{o(1)-1/2} \le O(\frac{1}{np})$.

Let us start with $r=4$, when we have $E(F_1)=E(F_2)=R$ and $\Delta=\emptyset$. There are $O(n^4)$ such ordered pairs, and each of them has a contribution of at most $O(p^4)$. Thus pairs with $r=4$ contribute $O(n^4 p^4)$.

Suppose $r=2$. The two common edges are either adjacent or opposite, and in both cases $\Delta$ is itself a $4$-cycle. Hence $\rho(\Delta)\ge 2$, since already this four-cycle contributes $\lceil (4-1)/2\rceil=2$. Since $|V(C_1)\cup V(C_2)|\le 5$, we have at most $O(n^5)$ ordered pairs with $r=2$, thus giving the contribution of at most $O(n^5p^6 \eta^2)=O(n^3p^4) =O(n^4p^4)$.

If $r=1$, the cycles share a single edge. Since deleting an edge from a cycle leaves it connected, the resulting graph $\Delta$, and have at least $5$ vertices. Hence $\rho(\Delta)\ge 2$, and since $|V(F_1)\cup V(F_2)|\leq 6$, the total contribution is at most $O(n^6p^7\eta^2) =O(n^4p^5)$.

If $r=0$, the two cycles are edge-disjoint. If they are also vertex-disjoint, then $\Delta$ is a union of two disjoint four-cycles and so $\rho(\Delta)\ge 4$. The contribution of this case is at most $O(n^8 p^8 \eta^4)=O(n^4p^4)$.

If $r=0$, the cycles share at least one vertex, and $\Delta$ spans a single connected component on $6$ or $7$ vertices. Hence $\rho(\Delta)\ge 3$, and these pairs contribute $O(n^7 p^8 \eta^3)= O(n^4p^5)$.

Summing over these estimates, we get $\bE[N_\square(G)^2]=O(n^4p^4)$.

Let us now turn to $F=\diamond$, so $m=5$ and $X=N_{\diamond}(G)$. Again, we split the second moment according to $r=|E(F_1)\cap E(F_2)|$. And again, we write $\eta$ as shorthand for $(\log d)^{3/2}/\sqrt{d} = O(\frac{1}{np})$.

If $r=5$, there is only one way $F_1$ and $F_2$ can intersect: they must have the same edge set. There are $O(n^4)$ such pairs, and their total contribution is $O(n^4p^5)$.

If $r=4$, then the vertex sets of $F_1$ and $F_2$ must again be the same (though now there are two ways this could occur: the two unique edges could form a matching, or path of length 2). There are $O(n^4)$ such pairs, and by the general bound their contribution is at most $O(n^4p^6)$.

If $r=3$, then $F_1$ and $F_2$ must share all three edges of one of the triangles (since sharing three edges of a four-cycle forces the fourth edges to be shared as well). Hence, $|V(F_1)\cup V(F_2)|= 5$. Moreover, $\Delta$ is a four-cycle, and so we have $\rho(\Delta)\geq 2$, bounding the final contribution by $O(n^5p^7 \eta^2)=O(n^4p^5)$.

Suppose $r=2$. Then $F_1, F_2$ share at least $3$ vertices, and so $|V(F_1)\cup V(F_2)|\leq 5$. Very crudely, we get $\rho(\Delta)\ge 1$ as $\Delta$ contains an edge. So, the total contribution is therefore at most $O(n^5p^8 \eta)= O(n^4p^7)$.

If $r=1$, then $|V(F_1)\cap V(F_2)|\geq 2$, meaning that $|V(F_1)\cup V(F_2)|\le 6$. Moreover, since deleting one edge cannot disconnect the diamond, the graph $\Delta$ is connected. So, $\rho(\Delta)\geq 2$ and the total contribution is at most $O(n^6p^9\eta^2)=O(n^4p^7)$.

Finally, suppose $r=0$. If the two diamonds are vertex-disjoint, then $\Delta$ is a disjoint union of two diamonds, so $\rho(\Delta)=4$, and this case contributes $O(n^8p^{10} \eta^4)=O(n^4p^6)$.

If the are not vertex-disjoint, their union has $6$ or $7$ vertices (since two edge-disjoint diamonds cannot fit in a $K_5$). Hence, we have $\rho(\Delta)\geq 3$ and at most $O(n^7)$ contributing terms, giving a total contribution of at most $ O(n^7p^{10}\eta^3) =O(n^4p^7)$.

Combining these estimates gives $\bE[X^2] = O(n^4p^5)$. So indeed, $\bE[N_{\diamond}(G)^2]\leq O(n^4p^5)$, and Markov's inequality completes the proof of \ref{item:diamond_cancellation}.
\end{proof}

\subsection{Deferred proofs from Section~\ref{sec:MGF}}\label{sec:lindeberg}

Let us start with some discussion to motivate the proof of Proposition~\ref{prop:lindeberg}, which states the following. If $\xi\in \bR^m$ is a vector with independent centered subgaussian coordinates of variance proxy $\sigma^2$, $M$ is a $m\times m$ matrix with zero diagonal, and $\ba\in \bR^m$ is a fixed vector, then
\begin{align*}
\log \bE e^{\xi^T M \xi+\ba^T \xi}=\frac{1}{2}\Var(\xi^T M \xi+\ba^T \xi)+O\Big(e^{O(\delta_2)}\delta_1\delta_2\Big),
\end{align*}
where $\delta_1=\max\{\sigma^2\|M\|_{op}, \sigma\|\ba\|_\infty\}$ and $\delta_2=\sigma^4 \|M\|_{HS}^2+\sigma^2\|\ba\|_2^2$, provided that $\delta_1$ is sufficiently small. 

The idea behind Proposition~\ref{prop:lindeberg} is simple - if $\xi$ was a gaussian vector, then the relevant moment generating function could be computed explicitly. Hence, we will define $g$ to be a gaussian vector matching $\xi$ in the first two moments. The Gaussian computation will show that $\log \bE e^{g^T M g+\ba^Tg}$ differs from $\frac{1}{2}\Var(\xi^T M \xi+\ba^T \xi)$ by at most $O(\delta_1\delta_2)$, so it remains to prove that $|\bE[e^{\xi^T M \xi+\ba^T\xi}]-\bE[e^{g^T M g+\ba^tg}]|\leq O(e^{O(\delta_2)}\delta_1\delta_2)$. To prove this inequality, we will use a replacement argument, very similar in spirit to that of Chatterjee~\cite{C05}. However, since we are dealing with exponentials, whose third partial derivatives are unbounded, we cannot directly appeal to~\cite{C05}. So, we must redo the proof, showing that the terms coming from the third-order terms are sufficiently small on average.

Let us first present the replacement argument relating $\bE[e^{\xi^T M \xi+\ba^T\xi}]$ to $\bE[e^{g^T M g+\ba^tg}]$, and close the section by computing $\bE[e^{g^T M g+\ba^tg}]$ precisely. Before we present the replacement argument however, let us remind ourselves of some standard properties of subgaussian random variables which we will use in the proof (for a proof, see Proposition 2.6.1 in \cite{Ver}).

As we saw in the preliminaries, a centered random variable $Z$ is subgaussian with variance proxy $\sigma^2$ if $\bE[e^{tZ}]\leq e^{\sigma^2t^2/2}$. If $K$ is a sufficiently large absolute constant, these variables also have $\bE[e^{Z^2/(K\sigma^2)}]\leq 2$, $\bE[Z^q]^{1/q}\leq K\sigma\sqrt{q}$. Moreover, these properties can be used as equivalent definitions of subgaussianity, if one is willing to change $\sigma$ by a constant factor. Finally, if $Z=\sum_{i}c_iZ_i$, and the variables $Z_i$ are independent subgaussians with variance proxy $\sigma_i^2$, then $Z$ is also subgaussian with variance proxy $\sum_i c_i^2\sigma_i^2$, which can be seen directly from the definition.

Let us now state and prove the inequality relating the two relevant moment generating functions.

\begin{proposition}\label{prop:comparison}
Let $\xi \in \bR^m$ be a vector with independent subgaussian coordinates, each with expectation $0$ and variance proxy at most $\sigma^2$, and let $g\in \bR^m$ be a Gaussian random vector, matching $\xi$ in the first two moments. Let $M$ be a $m\times m$ symmetric matrix with zero diagonal, and let $\ba\in \bR^m$ be a fixed vector. Additionally, let $\delta_1=\max\{\sigma^2\|M\|_{op}, \sigma\|\ba\|_\infty\}$ and $\delta_2=\sigma^4 \|M\|_{HS}^2+\sigma^2\|\ba\|_2^2$. If $\delta_1$ is sufficiently small, then
 \[\Big|\bE[e^{\xi^T M \xi+\ba^T \xi}]-\bE[e^{g^T M g+\ba^T g}]\Big|=O\Big(e^{O(\delta_2)}\delta_1\delta_2\Big).\]
\end{proposition}
\begin{proof}
Replacing $\xi$ and $g$ by $\xi/\sigma$ and $g/\sigma$, and replacing $M$ and $\ba$ by $\sigma^2M$ and $\sigma\ba$, respectively, does not change either exponential in the conclusion or the values of $\delta_1$ and $\delta_2$. Hence, we may assume without loss of generality that $\sigma=1$.

For a vector $\bx\in \bR^m$, let $f(\bx)=e^{\bx^T M \bx+\ba^T \bx}$ and let $\bz^{(i)}=(\xi_1, \dots, \xi_{i-1}, 0, g_{i+1}, \dots, g_m)$. Then, the triangle inequality gives
\begin{align*}
\Big|\bE f(\xi)-\bE f(g)\Big|\leq \sum_{i=1}^m \big|&\bE f(\xi_1, \dots, \xi_{i-1}, \xi_{i}, g_{i+1}, \dots, g_m)-\bE f(\xi_1, \dots, \xi_{i-1}, g_{i}, g_{i+1}, \dots, g_m)\big|
\end{align*}

The difference in step $i$, which we denote by $\Delta_i$, can be written as
\begin{align*}
\Delta_i \!&= \!\Big|\bE \exp\big({(\bz^{(i)})^T\!M\bz^{(i)}\!+\!\ba^T\!\bz^{(i)}\!+\!\xi_i(2M\bz^{(i)})_i\!+\!\ba_i\xi_i}\big)\!-\!\bE \exp\big({(\bz^{(i)})^T\!M\bz^{(i)}\!+\!\ba^T\!\bz^{(i)}\!+\!g_i(2M\bz^{(i)})_i\!+\!\ba_ig_i}\big)\!\Big|\\
&=\Big|\bE_{\bz^{(i)}}\Big[e^{(\bz^{(i)})^TM\bz^{(i)}+\ba^T\bz^{(i)}}\bE_{\xi_i, g_i}\big[e^{c_i\xi_i}-e^{c_ig_i}\big]\Big]\Big|,\text{ where }c_i=2(M\bz^{(i)})_i+\ba_i.
\end{align*}
We can write $e^{c_i\xi_i}=1+c_i\xi_i+\frac{c_i^2\xi_i^2}{2}+O(|c_i\xi_i|^3e^{|c_i\xi_i|})$ and $e^{c_ig_i}=1+c_ig_i+\frac{c_i^2g_i^2}{2}+O(|c_ig_i|^3e^{|c_ig_i|})$. Since we have assumed that $\xi_i$ and $g_i$ have the same mean and variance, the first three terms cancel after subtracting, and we are left with
\[\Big|\bE_{\xi_i, g_i}\big[e^{c_i\xi_i}-e^{c_ig_i}\big]\Big|\leq O\Big(\bE_{\xi_i}\big[|c_i\xi_i|^3e^{|c_i\xi_i|}\big]+\bE_{g_i}\big[|c_ig_i|^3e^{|c_ig_i|}\big]\Big).\]

Since $\xi_i$ is subgaussian, we have $\bE\big[|\xi_i|^3e^{|c_i\xi_i|}\big]\leq \sqrt{\bE [\xi_i^6]}\sqrt{\bE[e^{2|c_i\xi_i|}]}\leq O(1\cdot e^{O(|c_i|^2)})$, where the first bound follows from the standard properties of subgaussians mentioned above, and the second by $\bE[e^{2|c_i\xi_i|}]\leq \bE[e^{2|c_i|\xi_i}]+\bE[e^{-2|c_i|\xi_i}]\leq 2e^{O(|c_i|^2)}$ (here, we have used the definition of subgaussian random variable, which says that $\bE e^{t\xi_i}\leq e^{t^2\sigma^2/2}$). Similarly, $\bE\big[|g_i|^3e^{|c_ig_i|}\big]\leq O(e^{O(|c_i|^2)})$. Hence,
\[\Delta_i\leq O\Big(\bE_{\bz^{(i)}}\Big[e^{(\bz^{(i)})^TM\bz^{(i)}+\ba^T\bz^{(i)}}\cdot |c_i|^3e^{O(|c_i|^2)}\Big]\Big).\]

Using H\"older's inequality, we can bound this by
\[\Delta_i\leq O\Big(\bE_{\bz^{(i)}}\Big[e^{4(\bz^{(i)})^TM\bz^{(i)}}\Big]^{1/4}\cdot \bE_{\bz^{(i)}}\Big[e^{4\ba^T\bz^{(i)}}\Big]^{1/4}\cdot\bE_{\bz^{(i)}}\big[c_i^{12}\big]^{1/4}\cdot  \bE_{\bz^{(i)}}\Big[e^{O(|c_i|^2)}\Big]^{1/4}\Big).\]

Since $c_i-\ba_i=2(M\bz^{(i)})_i$ is a linear combination of centered subgaussian random variables, it is itself a centered subgaussian, with variance proxy $\sigma_i^2\leq 4\sum_{j=1}^m M_{ij}^2$. Hence,
\[\bE_{\bz^{(i)}}\big[c_i^{12}\big]^{1/4}\leq O\Big(\bE_{\bz^{(i)}}\big[\ba_i^{12}+(c_i-\ba_i)^{12}\big]^{1/4}\Big)\leq O\Big( |\ba_i|^3 + \sigma_i^3\Big)\leq O\Big(|\ba_i|^3+\Big(\sum_{j=1}^m M_{ij}^2\Big)^{3/2}\Big) .\]

Further, we have $\bE_{\bz^{(i)}}\big[e^{O(|c_i|^2)}\big]\leq e^{O(|\ba_i|^2)}\bE_{\bz^{(i)}}\big[e^{O(|c_i-\ba_i|^2)}\big]\leq e^{O(|\ba_i|^2)}$, where we have used that $\sigma_i^2\leq 4\sum_{j=1}^m M_{ij}^2\leq 4\|M\|_{op}^2 \leq 4\delta_1^2\ll 1$.

Further, the random variable $4\ba^T\bz^{(i)}$ is a linear combination of independent centered subgaussians, with variance proxy at most $16\|\ba\|_2^2$. Hence, $\bE_{\bz^{(i)}}[e^{4\ba^T\bz^{(i)}}]\leq e^{O(\|\ba\|_2^2)}$.

Finally, the random variable $4(\bz^{(i)})^TM\bz^{(i)}$ is a quadratic form in independent centered subgaussians, and as such its expectation can be bounded by $e^{O(\|M\|_{HS}^2)}$, using the Hanson-Wright inequality (see Theorem~\ref{thm:hanson_wright}). Hence, we have
\[\Delta_i\leq O\Big(e^{O(\|M\|_{HS}^2+\|\ba\|_2^2)}\Big(|\ba_i|^3+\Big(\sum_{j=1}^m M_{ij}^2\Big)^{3/2}\Big)\Big)\leq e^{O(\delta_2)}O\Big(\delta_1\Big(|\ba_i|^2+\Big(\sum_{j=1}^m M_{ij}^2\Big)\Big)\Big),\]
where the last inequality follows since $\delta_1\geq \|\ba\|_\infty$ and $\delta_1\geq \|M\|_{op}\geq \sqrt{\sum_{j=1}^m M_{ij}^2}$. Summing over $i$, we get
\[\Big|\bE e^{\xi^T M \xi+\ba^T \xi}-\bE e^{g^T M g+\ba^T g}\Big|\leq \sum_{i=1}^m |\Delta_i|\leq e^{O(\delta_2)}\delta_1O\Big(\sum_{i=1}^m |\ba_i|^2+\sum_{i,j=1}^m M_{ij}^2\Big)=O\Big(e^{O(\delta_2)}\delta_1\delta_2\Big),\]
which completes the proof of Proposition~\ref{prop:comparison}.
\end{proof}

The final step of the proof of Proposition~\ref{prop:lindeberg} is to compute $\bE[e^{g^T M g+\ba^T g}]$ precisely, which we do in the following lemma. This is a well-known fact (see e.g., \cite[Eq.~(1.3.6)]{MathaiProvostHayakawa1995}), but we include a proof for completeness.

\begin{lemma}\label{lemma:gaussian_mgf}
Let $g\sim \cN(0, I_m)$ be a standard normal in $\bR^m$, let $M$ be a symmetric zero-diagonal matrix satisfying $\|M\|_{op}\leq 1/4$, and let $\ba\in \bR^m$. Then $\log \bE[e^{g^T M g+\ba^T g}]=\frac{1}{2}\ba^T(I-2M)^{-1}\ba-\frac{1}{2}\log\det(I-2M)$.
\end{lemma}
\begin{proof}
By definition, we have 
\[\bE[e^{g^T M g+\ba^T g}]=\frac{1}{(2\pi)^{m/2}}\int_{\bR^m} \!\!\!\exp\Big(-\frac{1}{2} x^Tx+x^TMx+\ba^Tx\!\Big)dx=\frac{1}{(2\pi)^{m/2}}\int_{\bR^m} \!\!\!\exp\Big(-\frac{1}{2} x^T(I-2M)x+\ba^Tx\!\Big)dx.\]
Setting $A=I-2M$, we can complete the square to write $-\frac12 x^TAx+\ba^Tx=-\frac12(x-A^{-1}\ba)^TA(x-A^{-1}\ba)+\frac12\ba^TA^{-1}\ba$. Setting $y=x-A^{-1}\ba$, we have
\[\bE[e^{g^T M g+\ba^T g}]=\frac{e^{\frac12\ba^TA^{-1}\ba}}{(2\pi)^{m/2}}\int_{\bR^m} \exp\Big(-\frac{1}{2} y^TAy\Big)dy=\frac{e^{\frac12\ba^TA^{-1}\ba}}{\sqrt{\det A}}.\]
Taking logarithms then gives the desired formula.
\end{proof}

Let us put together the pieces to complete the proof of Proposition~\ref{prop:lindeberg}.

\begin{proof}[Proof of Proposition~\ref{prop:lindeberg}.]
Let $g$ be a Gaussian vector matching $\xi$ in the first two moments. As we have previously observed, we can scale $\xi$ and $g$ by $\sigma$, and $M$ and $\ba$ by $\sigma^2$ and $\sigma$, respectively, without changing the conclusion. Hence, we may assume without loss of generality that $\sigma$ is a small constant, chosen such that all components of $\xi$ and $g$ have variance at most $1$ (which is possible, since $\bE[\xi_i^2]\leq O(\sigma^2)$ due to standard properties of subgaussian random variables). Let us call $D$ the covariance matrix of $\xi$ and $g$, noting that it is a diagonal matrix with all entries at most $1$.

By Proposition~\ref{prop:comparison}, we then have $\Big|\bE[e^{\xi^T M \xi+\ba^T \xi}]-\bE[e^{g^T M g+\ba^T g}]\Big|=O\Big(e^{O(\delta_2)}\delta_1\delta_2\Big)$.
Since $\bE[\xi^T M \xi+\ba^T \xi]=\bE[g^T M g+\ba^T g]=0$, Jensen's inequality gives $\bE[e^{\xi^T M \xi+\ba^T \xi}]\geq e^{\bE[\xi^T M \xi+\ba^T \xi]}=1$ and from this we get $\bE[e^{g^T M g+\ba^T g}]\geq e^{\bE[g^T M g+\ba^T g]}=1$. Since $t\mapsto \log t$ is $1$-Lipschitz on $[1, \infty)$, we have
\[\Big|\log \bE[e^{\xi^T M \xi+\ba^T \xi}]-\log \bE[e^{g^T M g+\ba^T g}]\Big|=O\Big(e^{O(\delta_2)}\delta_1\delta_2\Big).\]

Further, by applying Lemma~\ref{lemma:gaussian_mgf} to the matrix $M'=D^{1/2}MD^{1/2}$ and the vector $g'=D^{-1/2}g\sim \cN(0, I_m)$, $\ba'=D^{1/2}\ba$, we find that
$\log \bE[e^{g^T M g+\ba^T g}]=\log \bE[e^{{g'}^T {M'} {g'}+{\ba'}^T g'}]=\frac{1}{2}{\ba'}^T(I-2M')^{-1}\ba'-\frac{1}{2}\log\det(I-2M')$. Note that Lemma~\ref{lemma:gaussian_mgf} still applies since $\|M'\|_{op}\leq \|M\|_{op}\ll 1$, due to the fact $\|D\|_{op}\leq 1$.

And thus to complete the proof, it remains to show that ${\ba'}^T(I-2M')^{-1}\ba'-\log\det(I-2M')=\Var(g^T M g+\ba^T g)+O(\delta_1\delta_2)$.

First, we claim ${\ba'}^T(I-2M')^{-1}\ba'=(1+O(\delta_1))\|\ba'\|_2^2$. This is simple to see since 
\[|{\ba'}^T(I-2M')^{-1}\ba'-{\ba'}^T\ba'|= |{\ba'}^T((I-2M')^{-1}-I)\ba'|\leq \|\ba\|_2^2 \|((I-2M')^{-1}-I)\|_{op}\leq O(\|M'\|_{op}\|\ba\|_2^2),\]
and recall that $\|M'\|_{op}\leq \|M\|_{op}\leq \delta_1$.

On the other hand, using the Taylor series of the function $A\mapsto \log \det A$, we find that when $\|M'\|_{op}$ is sufficiently small, then we have $\log\det(I-2M')=-\sum_{k=1}^\infty {{\rm tr}((2M')^k)}/{k}$. The first term of this sum vanishes since ${\rm tr}(M')=0$. Further, all terms with $k\geq 3$ can be bounded using ${\rm tr}((2M')^k)\leq \|2M\|_{op}^{k-2}{\rm tr}((2M')^2)\leq O(\delta_1)^k \|2M'\|_{HS}^2$. Hence, $\log\det(I-2M')=-\|2M'\|_{HS}^2\big(1/2+O(\sum_{k\geq 3}^\infty\delta_1^{k-2}/k))=-(2+O(\delta_1))\|M'\|_{HS}^{2}$.
So,
\begin{align*}
{\ba'}^T(I-2M')^{-1}\ba'-\log\det(I-2M')&=(1+O(\delta_1))\big(\|\ba'\|_2^2+2\|M'\|_{HS}^{2}\big)=\Var({g'}^T M' {g'}+{\ba'}^T g')+O(\delta_1\delta_2)\\
&=\Var(g^T M g+\ba^T g)+O(\delta_1\delta_2),
\end{align*}
which completes the proof of Proposition~\ref{prop:lindeberg}.
\end{proof}
\end{document}